\documentclass[journal]{IEEEtran}
\ifCLASSINFOpdf
\else
\fi

\usepackage{algorithm}
\usepackage{lipsum}
\usepackage{amsfonts}
\usepackage{amsmath}
\usepackage{mathtools}
\usepackage{multirow}
\usepackage{graphicx}
\usepackage{epstopdf}
\usepackage{algorithmic}
\usepackage{marvosym}
\usepackage{bm}
\usepackage{microtype}
\usepackage{subfigure}
\usepackage{booktabs} 
\usepackage{amssymb}
\usepackage{times}
\usepackage{enumerate}
\usepackage{tablefootnote}
\usepackage{color}
\usepackage{cite}
\ifpdf
  \DeclareGraphicsExtensions{.eps,.pdf,.png,.jpg}
\else
  \DeclareGraphicsExtensions{.eps}
\fi

\newcommand{\x}{\mathbf{x}}
\newcommand{\y}{\mathbf{y}}
\newcommand{\z}{\mathbf{z}}
\newcommand{\w}{\mathbf{w}}

\newcommand{\s}{\mathbf{s}}
\newcommand{\1}{\mathbf{1}}
\newcommand{\0}{0}

\renewcommand{\v}{\mathbf{v}}

\renewcommand{\b}{b}

\newcommand{\A}{A}

\newcommand{\M}{M}
\newcommand{\W}{W}

\newcommand{\I}{I}
\newcommand{\U}{U}
\newcommand{\V}{V}
\newcommand{\R}{\mathbb{R}}
\newcommand{\N}{\mathcal{N}}

\newcommand{\trace}{\mbox{trace}}

\newcommand{\<}{\left\langle}
\renewcommand{\>}{\right\rangle}

\DeclareMathOperator*{\argmin}{argmin}

\newtheorem{assumption}{Assumption}
\newtheorem{lemma}{Lemma}
\newtheorem{theorem}{Theorem}
\newtheorem{proof}{Proof}
\newtheorem{corollary}{Corollary}

\begin{document}
%
\title{Decentralized Accelerated Gradient Methods With Increasing Penalty Parameters}
%
%
%

\author{Huan~Li,~\IEEEmembership{Member, IEEE,}
        Cong~Fang,
        Wotao~Yin,
        Zhouchen~Lin,~\IEEEmembership{Fellow, IEEE,}
\thanks{H. Li is sponsored by Zhejiang Lab (grant no. 2019KB0AB02). Z. Lin is supported by NSF China (grant no.s 61625301 and 61731018), Major Research Project of Zhejiang Lab (grant no.s 2019KB0AC01 and 2019KB0AB02) and Beijing Academy of Artificial Intelligence.}
\thanks{H. Li is with the Institute of Robotics and Automatic Information Systems, College of Artificial Intelligence, Nankai University, Tianjin, China (lihuan\_ss@126.com), and the the College of Computer Science and Technology, Nanjing University of Aeronautics and Astronautics, Nanjing, China. This work was done when Li was a Ph.D student at the Key Lab. of Machine Perception (MOE), School of EECS, Peking University, Beijing, China.}
\thanks{C. Fang is with the Department of Electrical Engineering, Princeton University, Princeton, New Jersey, USA (fangcong@pku.edu.cn).}
\thanks{W. Yin is with the Department of Mathematics, University of California, Los Angeles, USA (wotaoyin@math.ucla.edu).}
\thanks{Z. Lin is with the Key Lab. of Machine Perception (MOE), School of EECS, Peking University, Beijing, China (zlin@pku.edu.cn). Z. Lin is the corresponding author.}
\thanks{This paper has supplementary downloadable material available at http://ieeexplore.ieee.org., provided by the author. The material includes some additional experimental results. This material is 2M in size.}
}
\maketitle

\begin{abstract}
  In this paper, we study the communication and (sub)gradient computation costs in distributed optimization. We present two algorithms based on the framework of the accelerated penalty method with increasing penalty parameters. Our first algorithm is for smooth distributed optimization and it obtains the near optimal $O(\sqrt{\frac{L}{\epsilon(1-\sigma_2(\W))}}\log\frac{1}{\epsilon})$ communication complexity and the optimal  $O(\sqrt{\frac{L}{\epsilon}})$ gradient computation complexity for $L$-smooth convex problems, where $\sigma_2(\W)$ denotes the second largest singular value of the weight matrix $\W$ associated to the network and $\epsilon$ is the target accuracy. When the problem is $\mu$-strongly convex and $L$-smooth, our algorithm has the near optimal $O(\sqrt{\frac{L}{\mu(1-\sigma_2(\W))}}\log^2\frac{1}{\epsilon})$ complexity for communications and the optimal $O(\sqrt{\frac{L}{\mu}}\log\frac{1}{\epsilon})$ complexity for gradient computations. Our communication complexities are only worse by a factor of $(\log\frac{1}{\epsilon})$ than the lower bounds. 
  Our second algorithm is designed for nonsmooth distributed optimization and it achieves both the optimal $O(\frac{1}{\epsilon\sqrt{1-\sigma_2(\W)}})$ communication complexity and $O(\frac{1}{\epsilon^2})$ subgradient computation complexity, which match the lower bounds for nonsmooth distributed optimization.
\end{abstract}

\begin{IEEEkeywords}
Distributed accelerated gradient algorithms, accelerated penalty method, optimal (sub)gradient computation complexity, near optimal communication complexity.
\end{IEEEkeywords}

%
\IEEEpeerreviewmaketitle

\section{Introduction}
In this paper, we consider the following distributed convex optimization problem:
\begin{eqnarray}
\min_{x\in\R^n}\frac{1}{m}\sum_{i=1}^m F_i(x)\equiv f_i(x)+h_i(x),\label{problem}
\end{eqnarray}
where $m$ agents form a connected and undirected network $\mathcal{G}=(\mathcal{V},\mathcal{E})$, $\mathcal{V}=\{1,2,...,m\}$ is the set of agents and $\mathcal{E}\subset\mathcal{V}\times\mathcal{V}$ is the set of edges, $F_i$ is the local objective function only available to agent $i$ and $x$ is the decision variable. $f_i$ is a convex and smooth function while $h_i$ is a convex but possibly nonsmooth one. We consider distributed algorithms using only local computations and communications, i.e., each agent $i$ makes its decision only based on the local computations on $F_i$ (i.e., the gradient of $f_i$ and the subgradient of $h_i$) and the local information received from its neighbors in the network. A pair of agents can exchange information if and only if they are directly connected in the network. Distributed computation has been widely used in signal processing \cite{Bazeruqe10}, automatic control \cite{sam09infcoom,ren06} and machine learning \cite{Dekel2012,Forero2012,Agarwal2011}.

\subsection{Literature Review}
\begin{table*}
\begin{center}
\scriptsize
\begin{tabular}{c|c|c}
\hline\hline
\multicolumn{3}{c}{Non-strongly convex and smooth case}\\
\hline
 Methods & Complexity of gradient computations & Complexity of communications\\
\hline
DNGD\footnotemark[1] &$O\left(\frac{1}{\epsilon^{5/7}}\right)$ \cite{qu2017-2} &$O\left(\frac{1}{\epsilon^{5/7}}\right)$ \cite{qu2017-2}\\
DN-C  & $O\left({\sqrt{\frac{L}\epsilon}}\right)$ \cite{Jakovetic-2014} & $O\left(\sqrt{\frac{L}{\epsilon}}\frac{1}{1-\sigma_2(\W)}\log\frac{1}{\epsilon}\right)$ \cite{Jakovetic-2014}\\
Accelerated Dual Ascent& $O\left(\frac{L}{\epsilon\sqrt{1-\sigma_2(\W)}}\log^2\frac{1}{\epsilon}\right)$ \cite{Uribe-2017} & $O\left(\sqrt{\frac{L}{\epsilon(1-\sigma_2(\W))}}\log\frac{1}{\epsilon}\right)$ \cite{Uribe-2017}\\
Our APM-C & $O\left({\sqrt{\frac{L}\epsilon}}\right)$ & $O\left(\sqrt{\frac{L}{\epsilon(1-\sigma_2(\W))}}\log\frac{1}{\epsilon}\right)$\\
\hline
Lower Bound & $O\left({\sqrt{\frac{L}\epsilon}}\right)$ \cite{Nesterov-2004} & $O\left(\sqrt{\frac{L}{\epsilon(1-\sigma_2(\W))}}\right)$\cite{scaman-2019}\\
\hline\hline
\multicolumn{3}{c}{Strongly convex and smooth case}\\
\hline
 Methods & Complexity of gradient computations & Complexity of communications\\
\hline
DNGD &$O\left(\big(\frac{L}{\mu}\big)^{5/7}\frac{1}{(1-\sigma_2(\W))^{1.5}}\log\frac{1}{\epsilon}\right)$ \cite{qu2017-2} & $O\left(\big(\frac{L}{\mu}\big)^{5/7}\frac{1}{(1-\sigma_2(\W))^{1.5}}\log\frac{1}{\epsilon}\right)$ \cite{qu2017-2} \\
Accelerated Dual Ascent& $O\left(\frac{L}{\mu\sqrt{1-\sigma_2(\W)}}\log^2\frac{1}{\epsilon}\right)$\hspace*{0.08cm}\footnotemark[3] \cite{Uribe-2017}& $O\left(\sqrt{\frac{L}{\mu(1-\sigma_2(\W))}}\log\frac{1}{\epsilon}\right)$ \cite{dasent,Uribe-2017}\\
Our APM-C & $O\left(\sqrt{\frac{L}{\mu}}\log\frac{1}{\epsilon}\right)$ & $O\left(\sqrt{\frac{L}{\mu(1-\sigma_2(\W))}}\log^2\frac{1}{\epsilon}\right)$\\
\hline
Lower Bound & $O\left(\sqrt{\frac{L}{\mu}}\log\frac{1}{\epsilon}\right)$ \cite{Nesterov-2004}& $O\left(\sqrt{\frac{L}{\mu(1-\sigma_2(\W))}}\log\frac{1}{\epsilon}\right)$ \cite{dasent}\\
\hline\hline
\multicolumn{3}{c}{Convex and Nonsmooth case}\\
\hline
 Methods & Complexity of subgradient computations & Complexity of communications\\
\hline
Primal-dual method     &  $O\left(\frac{1}{\epsilon^2}\right)$ \cite{scaman-2019}     &    $O\left(\frac{1}{\epsilon\sqrt{1-\sigma_2(\W)}}\right)$ \cite{scaman-2019}    \\
Smoothed accelerated gradient sliding method     &  $O\left(\frac{1}{\epsilon^2}\right)$ \cite{Uribe-2017}     &    $O\left(\frac{1}{\epsilon\sqrt{1-\sigma_2(\W)}}\right)$ \cite{Uribe-2017}    \\
Our APM                     &  $O\left(\frac{1}{\epsilon^2}\right)$      &    $O\left(\frac{1}{\epsilon\sqrt{1-\sigma_2(\W)}}\right)$     \\
\hline
Lower Bound & $O\left(\frac{1}{\epsilon^2}\right)$ \cite{scaman-2018} & $O\left(\frac{1}{\epsilon\sqrt{1-\sigma_2(\W)}}\right)$ \cite{scaman-2018}\\
\hline\hline
\end{tabular}
\end{center}
\caption{Complexity comparisons between accelerated dual ascent, DN-C, DNGD, the primal-dual method and our methods (APM-C,APM) for distributed convex problems.}\label{table1}
\end{table*}

Among the classical distributed first-order algorithms, two different types of methods have been proposed, namely, the primal-only methods and the dual-based methods.

The distributed subgradient method is a representative primal-only distributed optimization algorithm over general networks \cite{Nedic-2009}, while its stochastic version was studied in \cite{nedic2011asynchronous}, and asynchronous variant in \cite{Ram-2010}. In the distributed subgradient method, each agent performs a consensus step and then follows a subgradient descent with a diminishing step-size. To avoid the diminishing step-size, three different types of methods have been proposed. The first type of methods \cite{aug-dgm,qu2017,shi2017,qu2017-2} rely on tracking differences of gradients, which keep a variable to estimate the average gradient and use this estimation in the gradient descent step. The second type of methods, called EXTRA \cite{Shi-2015,Shi-2015-2}, introduce two different weight matrices as opposed to a single one with the standard distributed gradient method \cite{Nedic-2009}. EXTRA also uses the gradient tracking. The third type of methods employ a multi-consensus inner loop \cite{Jakovetic-2014,Wei-2018} and thus improve the consensus of the variables at each outer iteration.

The dual-based methods introduce the Lagrangian function and work in the dual space. Many classical methods can be used to solve the dual problem, e.g., the dual subgradient ascent \cite{Terelius-2011}, dual gradient ascent \cite{yuhao-2018}, accelerated dual gradient ascent \cite{dasent,Uribe-2017}, the primal-dual method \cite{Lam-2017,scaman-2018}, and ADMM \cite{Erseghe-2011,Shi-2014,Wei-2013,Iutzeler-2016,makhdoumi-2017,Aybat2018}. In general, most dual-based methods require the evaluation of the Fenchel conjugate of the local objective function $f_i(x)$ and thus have a larger gradient computation cost per iteration than the primal-only algorithms for smooth distributed optimization. For nonsmooth problems, the authors of \cite{Lam-2017,scaman-2018,scaman-2019} studied the communication-efficient primal-dual method. Specifically, they use the classical primal-dual method \cite{Chambolle-2011} in the outer loop and the subgradient method in the inner loop. The authors of \cite{scaman-2018} used Chebyshev acceleration \cite{Arioli-2014} to further reduce the computation complexity while the authors of \cite{scaman-2019} did it via carefully setting the parameters.

Among the methods described above, the distributed Nesterov gradient with consensus iterations (D-NC) proposed in \cite{Jakovetic-2014} and the distributed Nesterov gradient descent (DNGD) proposed in \cite{qu2017-2} employ Nesterov's acceleration technique in the primal space, and the accelerated dual ascent proposed in \cite{dasent} use the standard accelerated gradient descent in the dual space. Moreover, D-NC attains the optimal gradient computation complexity for nonstrongly convex and smooth problems, and the accelerated dual ascent achieves the optimal communication complexity for strongly convex and smooth problems, which match the complexity lower bounds \cite{Nesterov-2004,dasent}. For nonsmooth problems, the primal-dual method proposed in \cite{scaman-2018,scaman-2019} and the smoothed accelerated gradient sliding method in \cite{Uribe-2017} achieve both the optimal communication and subgradient computation complexities, which also match the lower bounds \cite{scaman-2018}. We denote the communication and computation complexities as the numbers of communications and (sub)gradient computations to find an $\epsilon$-optimal solution $x$ such that $\frac{1}{m}\sum_{i=1}^mF_i(x)-\min_{x}\frac{1}{m}\sum_{i=1}^mF_i(x)\leq \epsilon$, respectively.

\vspace*{-0.3cm}
\subsection{Contributions}

In this paper, we study the decentralized accelerated gradient methods with near optimal complexities from the perspective of the accelerated penalty method. Specifically, we propose an \emph{A}ccelerated \emph{P}enalty \emph{M}ethod with increasing penalties for smooth distributed optimization by employing a multi-\emph{C}onsensus inner loop (APM-C). The theoretical significance of our method is that we show the near optimal communication complexities and the optimal gradient computation complexities for both strongly convex and nonstrongly convex problems. Our communication complexities are only worse by a logarithm factor than the lower bounds.

Table \ref{table1} summarizes the complexity comparisons to the state-of-the-art distributed optimization algorithms (the notations in Table \ref{table1} will be specified precisely soon), namely, DNGD, D-NC, and the accelerated dual ascent reviewed above, as well as the complexity lower bounds. Our complexities match the lower bounds except that the communication ones have an extra factor of $\log\frac{1}{\epsilon}$. The communication complexity of the accelerated dual ascent matches ours for nonstrongly convex problems and is optimal for strongly convex problems (thus better than ours by $\log\frac{1}{\epsilon}$). On the other hand, our gradient computation complexities match the lower bounds and they are better than the compared methods. It should be noted that due to term $\log^2\frac{1}{\epsilon}$, our communication complexity for strongly convex problems is not a linear convergence rate.
\footnotetext[1]{The authors of \cite{qu2017-2} did not give the dependence on $1-\sigma_2(\W)$. It does not mean that their complexity has no dependence on $1-\sigma_2(\W)$.}

Our framework of accelerated penalty method with increasing penalties also applies to nonsmooth distributed optimization. It drops the multi-consensus inner loop but employs an inner loop with several runs of subgradient method. Both the optimal communication and subgradient computation complexities are achieved, which match the lower bounds for nonsmooth distributed optimization. Although the theoretical complexities are the same with the methods \cite{scaman-2019,Uribe-2017}, our method gives the users a new choice in practice.


\subsection{Notations and Assumptions}
Throughout the paper, the variable $x\in\R^n$ is the decision variable of the original problem (\ref{problem}). We denote $x_{(i)}\in\R^n$ to be the local estimate of the variable $x$ for agent $i$. To simplify the algorithm description in a compact form, we introduce the aggregate variable $\x$, aggregate objective function $f(\x)$ and aggregate gradient $\nabla f(\x)$ as
\begin{equation}\notag
\x=\hspace*{-0.15cm}\left(\hspace*{-0.15cm}
  \begin{array}{c}
    x_{(1)}^T\\
    \vdots\\
    x_{(m)}^T
  \end{array}
\hspace*{-0.15cm}\right),f(\x)=\hspace*{-0.15cm}\sum_{i=1}^m f_i(x_{(i)}),\nabla f(\x)=\hspace*{-0.15cm}\left(\hspace*{-0.2cm}
  \begin{array}{c}
    \nabla f_1(x_{(1)})^T\\
    \vdots\\
    \nabla f_m(x_{(m)})^T
  \end{array}
\hspace*{-0.2cm}\right),
\end{equation}
where $\x\in\R^{m\times n}$, whose value at iteration $k$ is denoted by $\x^k$. For the double loop algorithms, we denote $\x^{k,t}$ as its value at the $k$th outer iteration and $t$th inner iteration. Assume that the set of minimizers is non-empty. Denote $x^*$ as one minimizer of problem (\ref{problem}), and let $\x^*=\1(x^*)^T\in\R^{m\times n}$, where $\1=(1,1,\cdots,1)^T\in\R^m$ is the vector with all ones. Denote $\partial h_i(x)$ as the subdifferential of $h_i(x)$ at $x$, and specifically, $\hat\nabla h_i(x)\in \partial h_i(x)$ as its one subgradient. For $h_i$, we introduce its aggregate objective function $h(\x)$ and aggregate subgradient $\hat\nabla h(\x)$ as
\begin{equation}\notag
h(\x)=\sum_{i=1}^m h_i(x_{(i)})\quad\mbox{and}\quad \hat\nabla h(\x)=\left(
  \begin{array}{c}
    \hat\nabla h_1(x_{(1)})^T\\
    \vdots\\
    \hat\nabla h_m(x_{(m)})^T
  \end{array}
\right).
\end{equation}

We use $\|\cdot\|$ and $\|\cdot\|_1$ as the $l_2$ Euclidean norm and $l_1$ norm for a vector, respectively. For matrices $\x$ and $\y$, we denote $\|\x\|_F$ as the Frobenius norm, $\|\x\|_2$ as the spectral norm and $\<\x,\y\>=\trace(\x^T\y)$ as their inner product. Denote $\I\in\R^{m\times m}$ as the identity matrix and $\N_i$ as the neighborhood of agent $i$ in the network. Define
\begin{eqnarray}
&\alpha(\x)=\frac{1}{m}\sum_{i=1}^mx_{(i)}\label{define_alpha}
\end{eqnarray}
as the average across the rows of $\x$. Define two operators
\begin{eqnarray}
&\Pi=\I-\frac{1}{m}\1\1^T\quad \mbox{ and }\quad \U=\sqrt{\I-\W}\label{define_u}
\end{eqnarray}
to measure the consensus violation, where $\W$ is the weight matrix associated to the network, which describes the information exchange through the network. Especially, $\|\Pi\x\|_F$ directly measures the distance between $x_{(i)}$ and $\alpha(\x)$. We follow \cite{dasent} to define $\sqrt{\A}=\V\sqrt{\Lambda}\V^T$, given the eigenvalue decomposition $\A=\V\Lambda\V^T$ of the symmetric positive semidefinite matrix $\A$.

We make the following assumptions for each function $f_i(x)$.
\begin{assumption}\label{assumption_f}
\begin{enumerate}
\item $f_i(x)$ is $\mu$-strongly convex: $f_i(y)\geq f_i(x)+\<\nabla f_i(x),y-x\>+\frac{\mu}{2}\|y-x\|^2$. Especially, we allow $\mu$ to be zero through this paper, and in this case we say $f_i(x)$ is convex.
\item $f_i(x)$ is $L$-smooth: $f_i(y)\leq f_i(x)+\<\nabla f_i(x),y-x\>+\frac{L}{2}\|y-x\|^2$.
\end{enumerate}
\end{assumption}

In Assumption \ref{assumption_f}, $\mu$ and $L$ are the strong-convexity constant and smoothness constant, respectively. Assumption \ref{assumption_f} yields that the aggregate function $f(\x)$ is also $\mu$-strongly convex and $L$-smooth. For the nonsmooth function $h_i(x)$, we follow \cite{Lam-2017} to make the following assumptions.
\begin{assumption}\label{assumption_h}
\begin{enumerate}
\item $h_i(x)$ is convex.
\item $h_i(x)$ is $M$-Lipschitz continuous: $h_i(y)\leq h_i(x)+\<\hat\nabla h_i(x),y-x\>+M\|y-x\|$.
\end{enumerate}
\end{assumption}

We can simply verify that $h(\x)$ is $(\sqrt{m}M)$-Lipschitz continuous. For the weight matrix $\W$, we make the following assumptions.
\begin{assumption}\label{assumption_w}
\begin{enumerate}
\item $\W\in\R^{m\times m}$ is a symmetric matrix with $\W_{i,j}\neq 0$ if and only if agents $i$ and $j$ are neighbors or $i=j$. Otherwise, $\W_{i,j}=0$.
\item $\I\succeq\W\succeq \0$, and $\W\1=\1$.
\end{enumerate}
\end{assumption}

Examples satisfying Assumption \ref{assumption_w} can be found in \cite{Shi-2015}. We denote by $1=\sigma_1(\W)\geq\sigma_2(\W)\geq\cdots\geq\sigma_m(\W)$ the spectrum of $\W$. Note that for a connected and undirected network, we always have $\sigma_2(\W)<1$, and $\frac{1}{1-\sigma_2(\W)}$ is a good indication of the network connectivity.
For many commonly used networks, we can give order-accurate estimate on $\frac{1}{1-\sigma_2(\W)}$ \cite[Proposition 5]{nedic-2018}. For example, $\frac{1}{1-\sigma_2(\W)}=O(m\log m)$ for the geometric graph, and $\frac{1}{1-\sigma_2(\W)}=O(1)$ for the expander graph and Erd\H{o}s$-$R\'{e}nyi random graph. Moreover, for any connected and undirected graph, $\frac{1}{1-\sigma_2(\W)}=O(m^2)$ in the worst case \cite{nedic-2018}.

In this paper, we focus on the communication and (sub)gradient computation complexity development for the proposed algorithms. We define one communication to be the operation that all the agents exchange information with their neighbors once, i.e., $\sum_{j\in\N_i}\W_{ij}x_{(j)}$ for all $i=1,2,...,m$. One (sub)gradient computation is defined to be the (sub)gradient evaluations of all the agents once, i.e., $\nabla f_i(x_{(i)})$ ($\hat\nabla h_i(x_{(i)})$) for all $i$.

\section{Development of the Accelerated Penalty Method}\label{sec:algorithm}

\subsection{Accelerated Penalty Method for Smooth Distributed Optimization}\label{sec:dnc}
In this section, we consider the smooth distributed optimization, i.e., $h_i(x)=0$ in problem (\ref{problem}). From the definition of $\Pi$ in (\ref{define_u}), we know that $x_{(1)}=\cdots=x_{(m)}$ is equivalent to $\Pi\x=\0$. Thus, we can reformulate the smooth distributed problem as
\begin{eqnarray}
\begin{aligned}\label{problem1}
\min_{\x\in\R^{m\times n}}f(\x)\quad \mbox{s.t.}\quad\Pi\x=\0.
\end{aligned}
\end{eqnarray}
Problem (\ref{problem1}) is a standard linearly constrained convex problem, and many algorithms can be used to solve it, e.g., the primal-dual method \cite{Lam-2017,scaman-2018,hong-2017,jakovetic-2017} and dual ascent \cite{Terelius-2011,dasent,Uribe-2017}. In order to propose an accelerated distributed gradient method based on the gradient of $f(\x)$, rather than the evaluation of its Fenchel conjugate or proximal mapping, we follow \cite{Yuan-2016} to use the penalty method to solve problem (\ref{problem1}) in this paper. Specifically, the penalty method solves the following problem instead:
\begin{eqnarray}
\min_{\x\in\R^{m\times n}} f(\x)+\frac{\beta}{2}\|\Pi\x\|_F^2,\label{penalty2}
\end{eqnarray}
where $\beta$ is a large constant. However, one big issue of the penalty method is that problems (\ref{problem1}) and (\ref{penalty2}) are not equivalent for finite $\beta$. When solving problem (\ref{penalty2}), we can only obtain an approximate solution of (\ref{problem1}) with small $\|\Pi\x\|_F$, rather than $\|\Pi\x\|_F\rightarrow 0$, and the algorithm only converges to a neighborhood of the solution set of problem (\ref{problem}) \cite{Yuan-2016}. Moreover, to find an $\epsilon$-optimal solution of (\ref{problem1}), we need to pre-define a large $\beta$ of the order $\frac{1}{\epsilon}$ \cite{Yuan-2016}. Thus, the parameter setting depends on the precision $\epsilon$. When $\beta$ is fixed as a constant of the order $\frac{1}{\epsilon}$, we can only get the $\epsilon$-accurate solution after some fixed iterations described by $\epsilon$, and more iterations will not give a more accurate solution. Please see Section \ref{relation_apm} for more details. To solve the above two problems, we use the gradually increasing penalty parameters, i.e., at the $k$th iteration, we use $\beta=\frac{\beta_0}{\vartheta_k}$ with fixed $\beta_0$ and diminishing $\vartheta_k\rightarrow 0$. The increasing penalty strategy has two advantages: 1) The solution of (\ref{penalty2}) approximates that of (\ref{problem1}) infinitely when the iteration number $k$ is sufficiently large. 2) The parameter setting does not depend on the accuracy $\epsilon$. The algorithm can be run without defining the accuracy $\epsilon$ in advance. It can reach arbitrary accuracy if run for arbitrarily long time.

We use the classical accelerated proximal gradient method (APG) \cite{Beck-2009-APG} to minimize the penalized objective in (\ref{penalty2}), i.e., at the $k$th iteration, we first compute the gradient of $f(\x)$ at some extrapolated point, and then compute the proximal mapping of $\frac{\beta_0}{2\vartheta_k}\|\Pi\x\|_F^2$ at some $\z$, defined as
\begin{equation}
\argmin_{\x\in\R^{m\times n}} \frac{\beta_0}{2\vartheta_k}\|\Pi\x\|_F^2+\frac{L}{2}\|\x-\z\|_F^2.\label{proximal_def}
\end{equation}

Due to the special form of $\Pi$ defined in (\ref{define_u}), a simple calculation yields $\frac{L\vartheta_k\z+\beta_0\1\alpha(\z)^T}{L\vartheta_k+\beta_0}$ as the solution of (\ref{proximal_def}), where $\alpha(\x)$ is defined in (\ref{define_alpha}). However, in the distributed setting, we can only compute $\alpha(\z)$ approximately in finite communications. Thus, we use the inexact APG to minimize (\ref{penalty2}), i.e., we compute the proximal mapping inexactly. Specifically, the algorithm framework consists of the following steps:
\begin{subequations} \begin{align}
&\y^k=\x^k+\frac{L\theta_k-\mu}{L-\mu}\frac{1-\theta_{k-1}}{\theta_{k-1}}(\x^k-\x^{k-1})\label{apg1},\\
&\z^k=\y^k-\frac{1}{L}\nabla f(\y^k)\label{apg2},\\
&\x^{k+1}\approx\argmin_{\x\in\R^{m\times n}}\frac{\beta_0}{2\vartheta_k}\|\Pi\x\|_F^2+\frac{L}{2}\left\|\x-\z^k\right\|_F^2\label{apg3},
\end{align} \end{subequations}
where the sequences $\{\theta_k\}$ and $\{\vartheta_k\}$ and the precision in step (\ref{apg3}) will be specified in Theorems \ref{theorem_D_NC_sc} and \ref{theorem_D_NC_ns} latter. Now, we consider the subproblem in procedure (\ref{apg3}). As discussed above, we only need to approximate $\alpha(\z^k)$, which can be obtained by the classical average consensus \cite{xiao-2004} or the accelerated average consensus \cite{liu-2011}. We only consider the accelerated average consensus, which consists of the following iterations:
\begin{equation}
\z^{k,t+1}=(1+\eta)\W\z^{k,t}-\eta\z^{k,t-1},\label{aac}
\end{equation}
where we initialize at $\z^{k,0}=\z^{k,-1}=\z^k$. The advantage of using the special $\Pi$ in (\ref{problem1}) is that we only need to call the classic average consensus to solve the subproblem in (\ref{apg3}), which has been well studied in the literatures, including the extensions over directed network and time-varying network \cite{avecons2010}. In fact, in Lemma \ref{lemma1}, we only require $\|\z^{k,T_k}-\1\alpha(\z^k)\|_F^2$ to be within some precision for the method used in the inner loop. Any average consensus method over undirected graph, directed graph or time-varying graph can be used in the inner loop, as long as it has a linear convergence.

Combing (\ref{apg1})-(\ref{apg3}) and (\ref{aac}), we can give our method, which is presented in a distributed way in Algorithm \ref{D-NC}. We use notations $\x^{-1}$ and $\z^{k,-1}$ in Algorithm \ref{D-NC} only for the writing consistency when beginning the recursions from $k=0$ and $t=0$.

\begin{algorithm}[t]
   \caption{Accelerated Penalty Method with Consensus (APM-C)}
   \label{D-NC}
\begin{algorithmic}
   \STATE Initialize $x_{(i)}^0=x_{(i)}^{-1}$ for all $i$, and $\eta=\frac{1-\sqrt{1-\sigma_2^2(\W)}}{1+\sqrt{1-\sigma_2^2(\W)}}$.
   \FOR{$k=0,1,2,\cdots$}
   \STATE $y_{(i)}^k=x_{(i)}^k+\frac{L\theta_k-\mu}{L-\mu}\frac{1-\theta_{k-1}}{\theta_{k-1}}\left(x_{(i)}^k-x_{(i)}^{k-1}\right)\quad\forall i,$
   \STATE $z_{(i)}^k=y_{(i)}^k-\frac{1}{L}\nabla f_i(y_{(i)}^k)\quad\forall i,$
   \STATE $z_{(i)}^{k,0}=z_{(i)}^{k,-1}=z_{(i)}^k\quad\forall i,$
   \FOR{$t=0,1,\cdots,T_k-1$}
   \STATE $z_{(i)}^{k,t+1}=(1+\eta)\sum_{j\in\N_i}\W_{ij}z_{(j)}^{k,t}-\eta z_{(i)}^{k,t-1}\quad\forall i,$
   \ENDFOR
   \STATE $x_{(i)}^{k+1}=\frac{L\vartheta_kz_{(i)}^k+\beta_0z_{(i)}^{k,T_k}}{L\vartheta_k+\beta_0}\quad\forall i.$
   \ENDFOR
\end{algorithmic}
\end{algorithm}

\subsubsection{Complexities}

In this section, we discuss the complexities of Algorithm \ref{D-NC}. We first consider the strongly convex case and give the complexities in the following theorem.
\begin{theorem}\label{theorem_D_NC_sc}
Assume that Assumptions \ref{assumption_f} and \ref{assumption_w} hold with $\mu>0$. Setting $\theta_k=\theta=\sqrt{\frac{\mu}{L}}$ for all $\forall k$, $\vartheta_k=(1-\theta)^{k+1}$, and $T_k=O\left(\frac{k\sqrt{\mu/L}}{\sqrt{1-\sigma_2(\W)}}\right)$. Then, Algorithm \ref{D-NC} needs $O\left(\sqrt{\frac{L}{\mu}}\log\frac{1}{\epsilon}\right)$ gradient computations and $O\left(\sqrt{\frac{L}{\mu(1-\sigma_2(\W))}}\log^2\frac{1}{\epsilon}\right)$ total communications to achieve an $\epsilon$-optimal solution $\x$ such that
\begin{eqnarray}
\begin{aligned}\label{cont14}
&\frac{1}{m}\sum_{i=1}^mf_i(\alpha(\x))-\frac{1}{m}\sum_{i=1}^mf_i(x^*)\leq \epsilon\\
&\qquad\frac{1}{m}\sum_{i=1}^m \left\|x_{(i)}-\alpha(\x)\right\|^2\leq \epsilon^2.
\end{aligned}
\end{eqnarray}
\end{theorem}

When we drop the strong-convexity assumption, we have the following theorem.
\begin{theorem}\label{theorem_D_NC_ns}
Assume that Assumptions \ref{assumption_f} and \ref{assumption_w} hold with $\mu=0$. Let sequences $\{\theta_k\}$ and $\{\vartheta_k\}$ satisfy $\theta_0=1$, $\frac{1-\theta_k}{\theta_k^2}=\frac{1}{\theta_{k-1}^2}$, and $\vartheta_k=\theta_k^2$. Setting $T_k=O\left(\frac{\log k}{\sqrt{1-\sigma_2(\W)}}\right)$ and $\beta_0\geq L\|\nabla f(\x^*)\|_F^2$. Then, Algorithm \ref{D-NC} needs $O\left(\sqrt{\frac{L}{\epsilon}}\right)$ gradient computations and $O\left(\sqrt{\frac{L}{\epsilon(1-\sigma_2(\W))}}\log\frac{1}{\epsilon}\right)$ total communications to achieve an $\epsilon$-optimal solution $\x$ such that (\ref{cont14}) holds.
\end{theorem}

\subsection{Accelerated Penalty Method for Nonsmooth Distributed Optimization}\label{sec:dng}
In this section, we consider the nonsmooth problem (\ref{problem}). From Assumption \ref{assumption_w} and the definition in (\ref{define_u}), we know $\I\succeq\U\succeq\0$, and $x_{(1)}=\cdots=x_{(m)}$ is equivalent to $\U\x=\0$ \cite{dasent}. Thus, similar to (\ref{problem1}), we can reformulate problem (\ref{problem}) as
\begin{eqnarray}
\begin{aligned}\label{problem3}
\min_{\x\in\R^{m\times n}}F(\x)\equiv f(\x)+h(\x)\quad \mbox{s.t.}\quad\U\x=\0.
\end{aligned}
\end{eqnarray}

Similar to Section \ref{sec:dnc}, we also further rewrite the problem as a penalized problem and use APG with increasing penalties to minimize the penalized objective $F(\x)+\frac{\beta_0}{2\vartheta_k}\|\U\x\|_F^2$. However, due to the nonsmooth term $h(\x)$, we cannot compute the proximal mapping of $h(\x)+\frac{\beta_0}{2\vartheta_k}\|\U\x\|_F^2$ efficiently. Thus, we use a slightly different strategy here. Specifically, we first compute the gradient of $f(\x)+\frac{\beta_0}{2\vartheta_k}\|\U\x\|_F^2$ at some extrapolated point $\y$, i.e., $\nabla f(\y)+\frac{\beta_0}{\vartheta_k}\U^2\y$, and then compute the inexact proximal mapping of $h(\x)$. We describe the iterations as follows:
\begin{subequations}
\begin{align}
&\hspace*{-0.2cm}\y^k=\x^k+\frac{\theta_k(1-\theta_{k-1})}{\theta_{k-1}}(\x^k-\x^{k-1}),\label{apgns1}\\
&\hspace*{-0.2cm}\s^k=\nabla f(\y^k)+\frac{\beta_0}{\vartheta_k}\U^2\y^k,\label{apgns2}\\
&\hspace*{-0.2cm}\x^{k+1}\hspace*{-0.05cm}\approx\hspace*{-0.05cm} \argmin_{\x\in\R^{m\times n}} h(\x)\hspace*{-0.05cm}+\hspace*{-0.05cm}\<\s^k,\x\>\hspace*{-0.05cm}+\hspace*{-0.05cm}\left(\hspace*{-0.05cm}\frac{L}{2}\hspace*{-0.05cm}+\hspace*{-0.05cm}\frac{\beta_0}{2\vartheta_k}\hspace*{-0.05cm}\right)\hspace*{-0.05cm}\|\x\hspace*{-0.05cm}-\hspace*{-0.05cm}\y^k\|_F^2\hspace*{-0.05cm}\label{apgns3}.
\end{align} \end{subequations}
The reason why we use $\U$ in (\ref{problem3}), rather than $\Pi$, is that $\U^2\y^k$ can be efficiently computed, which corresponds to the gossip-style communications. Otherwise, we need to compute the average across $y^k_{(1)},...,y^k_{(m)}$, which cannot be achieved with closed form solution in the distributed environment.

When the proximal mapping of $h(\x)$, i.e., $\mbox{Prox}_{h}(\z)=\argmin_{\x\in\R^{m\times n}} h(\x)+\frac{1}{2}\|\x-\z\|^2$ for some $\z$, has closed form solution or can be easily computed, step (\ref{apgns3}) has a low computation cost, which reduces to
\begin{eqnarray}
\begin{aligned}\label{apmns_s3}
\x^{k+1}\hspace*{-0.05cm}=\hspace*{-0.05cm}\mbox{Prox}_h\hspace*{-0.05cm}\left(\hspace*{-0.05cm}\y^k\hspace*{-0.05cm}-\hspace*{-0.05cm}\frac{1}{L\hspace*{-0.05cm}+\hspace*{-0.05cm}\beta_0/\vartheta_k}\hspace*{-0.05cm}\left(\hspace*{-0.05cm}\nabla f(\y^k)\hspace*{-0.05cm}+\hspace*{-0.05cm}\frac{\beta_0}{\vartheta_k}\U^2\y^k\hspace*{-0.05cm}\right)\hspace*{-0.05cm}\right).\hspace*{-0.05cm}
\end{aligned}
\end{eqnarray}
We can see that when we set a large penalty parameter $\beta$, i.e., exchange $\frac{\beta_0}{\vartheta_k}$ with a large $\beta$ such that $\beta\gg L$ in (\ref{apmns_s3}), (\ref{apmns_s3}) approximately reduces to $\x^{k+1}\approx\mbox{Prox}_h\left(\y^k-\U^2\y^k\right)$ and $\nabla f(\y^k)$ is flooded by the large penalty parameters. This is another reason to use the increasing penalty parameters.

When the proximal mapping of $h(\x)$ does not have a closed form solution, we borrow the idea of gradient and communication sliding proposed in \cite{lan16gs,lan16gs2,lan16gs3,Lam-2017}, which skips the computations of $\nabla f$ and the inter-node communications  from time to time so that only $O(1/\epsilon)$ gradient evaluations and communications are needed in the $O(1/\epsilon^2)$ iterations required to solve problem (\ref{problem3}). Specifically, we incorporate a subgradient descent procedure to solve the subproblem in (\ref{apgns3}) with a sliding period $T_k$, which is also adopted by \cite{scaman-2018}. The subgradient descent is described as follows for $T_k$ iterations:
\begin{eqnarray}
\begin{aligned}\notag
\z^{k,t+1}=\argmin_{\z\in\R^{m\times n}}& \<\hat\nabla h(\z^{k,t}),\z\>+\<\s^k,\z\>\\
&+\left(\hspace*{-0.05cm}\frac{L}{2}\hspace*{-0.05cm}+\hspace*{-0.05cm}\frac{\beta_0}{2\vartheta_k}\hspace*{-0.05cm}\right)\|\z\hspace*{-0.05cm}-\hspace*{-0.05cm}\y^k\|_F^2\hspace*{-0.05cm}+\hspace*{-0.05cm}\frac{1}{2\eta_k}\|\z\hspace*{-0.05cm}-\hspace*{-0.05cm}\z^{k,t}\|_F^2.
\end{aligned}
\end{eqnarray}
We describe the method in a distributed way in Algorithm \ref{D-NG}.
\begin{algorithm}[t]
   \caption{Accelerated Penalty Method (APM)}
   \label{D-NG}
\begin{algorithmic}
   \STATE Initialize $x_{(i)}^0=x_{(i)}^{-1}$, and $z_{(i)}^{-1,T_{-1}}=x_{(i)}^0$ for all $i$.
   \FOR{$k=0,1,2,\cdots,K$}
   \STATE $y_{(i)}^k=x_{(i)}^k+\frac{\theta_k(1-\theta_{k-1})}{\theta_{k-1}}(x_{(i)}^k-x_{(i)}^{k-1})\quad\forall i,$
   \STATE $\s_{(i)}^k=\nabla f_i(y_{(i)}^k)+\frac{\beta_0}{\vartheta_k}\left(y_{(i)}^k-\sum_{j=1}^m \W_{i,j}y_{(j)}^k\right)\quad\forall i,$
   \STATE $z_{(i)}^{k,0}=z_{(i)}^{k-1,T_{k-1}}\quad\forall i$,
   \FOR{$t=0,1,\cdots,T_k-1$}
   \STATE $z_{(i)}^{k,t+1}=\argmin_{z\in\R^n} \<\hat\nabla h_i(z_{(i)}^{k,t})+\s_{(i)}^k,z\>$
   \STATE $\hspace*{1cm}+\left(\frac{L}{2}+\frac{\beta_0}{2\vartheta_k}\right)\|z-y_{(i)}^k\|^2+\frac{1}{2\eta_k}\|z-z_{(i)}^{k,t}\|^2\quad\forall i.$
   \ENDFOR
   \STATE $x_{(i)}^{k+1}=\frac{\sum_{t=0}^{T_k-1}z_{(i)}^{k,t+1}}{T_k}\quad \forall i$.
   \ENDFOR
\end{algorithmic}
\end{algorithm}

\subsubsection{Complexities}
Introduce constants $R_1$ and $R_2$ such that
\begin{equation}
\hspace*{-0.2cm}\|x_{(i)}^0-x^*\|^2\leq R_1^2\quad\mbox{and}\quad\|\nabla f_i(x^*)\|^2\leq R_2^2\quad\mbox{for all } i, \label{constantR}
\end{equation}
and assume $R_1\geq 1$ for simplicity. Then, we describe the convergence rate for Algorithm \ref{D-NG} in the following theorem.
\begin{theorem}\label{the1}
Assume that Assumptions \ref{assumption_f}, \ref{assumption_h} and \ref{assumption_w} hold with $\mu=0$. Let sequences $\{\theta_k\}$ and $\{\vartheta_k\}$ satisfy $\theta_0=1$, $\frac{1-\theta_k}{\theta_k}=\frac{1}{\theta_{k-1}}$, and $\vartheta_k=\theta_k$. Set $T_k=K(1-\sigma_2(\W))$, $\eta_k=\frac{\theta_k}{KM\sqrt{1-\sigma_2(\W)}}$, and $\beta_0=\frac{\max\left\{M,L\right\}}{\sqrt{1-\sigma_2(\W)}}$, where $K$ is the number of outer iterations. Then, for Algorithm \ref{D-NG}, we have
\begin{eqnarray}
&&\frac{1}{m}\sum_{i=1}^m F_i\left(\alpha(\x^K)\right)-\frac{1}{m}\sum_{i=1}^m F_i(x^*)\notag\\
&\leq& \frac{\beta_0}{K}\left(31+\frac{8}{K\sqrt{1-\sigma_2(\W)}}\right)\left(R_1+\frac{R_2}{L}\right)^2,\notag
\end{eqnarray}
and
\begin{eqnarray}
&&\frac{1}{m}\sum_{i=1}^m \left\|x_{(i)}^K-\alpha(\x^K)\right\|^2\leq \frac{16\beta_0^2}{K^2M^2}\left(R_1+\frac{R_2}{L}\right)^2.\notag
\end{eqnarray}
\end{theorem}

Consider the simple problem of computing the average of $x_{(1)},\cdots,x_{(m)}$. The accelerated averaged consensus \cite{liu-2011} needs $O\left(\frac{1}{\sqrt{1-\sigma_2(\W)}}\log\frac{1}{\epsilon}\right)$ iterations to find an $\epsilon$-accurate solution. Thus, it is reasonable to assume $K\geq\frac{1}{\sqrt{1-\sigma_2(\W)}}$. Moreover, from the $L$-smoothness of $f_i(x)$, we know $R_2$ is often of the order $O(LR_1)$. Thus, Theorem \ref{the1} establishes the $O\left(\frac{\max\{M,L\}}{\epsilon\sqrt{1-\sigma_2(\W)}}\right)$ communication complexity and the $\sum_{k=1}^KT_k=K^2(1-\sigma_2(\W))=O\left(\frac{\max\{M,L\}^2}{\epsilon^2}\right)$ subgradient computation complexity such that (\ref{cont14}) holds for nonsmooth distributed optimization.

In Theorem \ref{the1}, we set $T_k$ and $\eta_k$ dependent on the number of outer iterations. As explained in Section \ref{sec:dnc}, it is a unpractical parameter setting and moreover, the large $T_k$ and $\frac{1}{\eta_k}$ make the algorithm slow in practice. In the following corollary, we give a more reasonable setting of the parameters at an expense of higher complexities by the order of $\log\frac{1}{\epsilon}$, i.e., $\log K$.
\begin{corollary}\label{lemma_cons-noncons}
Under the settings in Theorem \ref{the1} but with $T_k=\frac{1-\sigma_2(\W)}{\theta_k}$ and $\eta_k=\frac{\theta_k^2}{M\sqrt{1-\sigma_2(\W)}}$, we have
\begin{eqnarray}
&&\frac{1}{m}\sum_{i=1}^m F_i\left(\alpha(\x^K)\right)-\frac{1}{m}\sum_{i=1}^m F_i(x^*)\notag\\
&\leq& \frac{\beta_0\log K}{K}\left(31+\frac{8}{K\sqrt{1-\sigma_2(\W)}}\right)\left(R_1+\frac{R_2}{L}\right)^2,\notag
\end{eqnarray}
and
\begin{eqnarray}
&\frac{1}{m}\sum_{i=1}^m \left\|x_{(i)}^K-\alpha(\x^K)\right\|^2\leq \frac{16\beta_0^2\log K}{K^2M^2}\left(R_1+\frac{R_2}{L}\right)^2.\notag
\end{eqnarray}
\end{corollary}

When $f(\x)$ is strongly convex, we can prove a faster $O\left(\frac{1}{k^2}\right)$ convergence rate for Algorithm \ref{D-NG} with $\theta_k=\frac{2}{k+2}$ and $\vartheta_k=\theta_k^2$. However, the quickly diminishing step-size in step (\ref{apgns3}) makes the algorithm slow in practice. So we omit the discussion for the strongly convex case.

\subsection{Relations of APM-C and APM to the Existing Algorithm Frameworks}

\subsubsection{Difference from the classical penalty method}\label{relation_apm}
To the best of our knowledge, most traditional work analyze the penalty method with a fixed penalty parameter \cite{Lan-2013-PM,ION2017}. Let's discuss the disadvantage of the large and fixed penalty parameter. Take problem (\ref{problem1}) as an example. Let $\{\x^*,\lambda^*\}$ be a pari of KKT point of problem (\ref{problem1}) and $\hat \x^*$ be the minimizer of problem (\ref{penalty2}), from the proof in \cite[Proposition 10]{Lan-2013-PM}, we have
\begin{eqnarray}
\begin{aligned}\notag
f(\x^*)=f(\x^*)+\frac{\beta}{2}\|\Pi\x^*\|_F^2\geq f(\hat\x^*)+\frac{\beta}{2}\|\Pi\hat\x^*\|_F^2.
\end{aligned}
\end{eqnarray}
So for any $\varepsilon$-accurate solution $\x$ of problem (\ref{penalty2}), we have
\begin{eqnarray}
\begin{aligned}\notag
&f(\x)+\frac{\beta}{2}\|\Pi\x\|_F^2-f(\x^*)\\
\leq& f(\x)+\frac{\beta}{2}\|\Pi\x\|_F^2-f(\hat\x^*)-\frac{\beta}{2}\|\Pi\hat\x^*\|_F^2\leq \varepsilon.
\end{aligned}
\end{eqnarray}
On the other hand, since $\x^*=\argmin_{\x} f(\x)+\<\lambda^*, \Pi\x\>$ and $\Pi\x^*=\0$, we have
\begin{eqnarray}
\begin{aligned}\notag
&f(\x^*)=f(\x^*)+\<\lambda^*, \Pi\x^*\>\leq f(\x)+\<\lambda^*, \Pi\x\>\\
\Rightarrow&-\|\lambda^*\|_F\|\Pi\x\|_F\leq f(\x)-f(\x^*).
\end{aligned}
\end{eqnarray}
So $\frac{\beta}{2}\|\Pi\x\|_F^2-\|\lambda^*\|_F\|\Pi\x\|_F\leq \varepsilon$, which leads to
\begin{eqnarray}
\begin{aligned}\notag
\|\Pi\x\|_F\leq \frac{2\|\lambda^*\|_F}{\beta}+\sqrt{\frac{2\varepsilon}{\beta}}=O(\epsilon+\sqrt{\epsilon\varepsilon})
\end{aligned}
\end{eqnarray}
and
\begin{eqnarray}
\begin{aligned}\notag
|f(\x)-f(\x^*)|\leq\max\{\varepsilon,\epsilon+\sqrt{\epsilon\varepsilon}\}
\end{aligned}
\end{eqnarray}
by $\beta=\frac{1}{\epsilon}$. We can see that the accuracy is dominated by $\max\{\epsilon,\varepsilon\}$, and more iterations with smaller $\varepsilon$ will not produce a more accurate solution.

On the other hand, even if $\varepsilon=0$ and $\x=\hat\x^*$ with infinite iterations, we have $\nabla f(\x)+\beta\Pi\x=\0$, which only leads to $\|\Pi\x\|_F=\epsilon\|\nabla f(\x)\|_F=O(\epsilon)$ and $|f(\x)-f(\x^*)|\leq \epsilon$, rather than $\|\Pi\x\|_F= 0$ and $|f(\x)-f(\x^*)|=0$.

\subsubsection{Difference from the classical accelerated first-order algorithms}
We extend the classical accelerated gradient method \cite{Nesterov1983,Nesterov1988,Beck-2009-APG,devolder2014first,schnidt-2011-nips} from the unconstrained problems to the linearly constrained problems via the perspective of the penalty method. However, since we use the increasing penalty parameters at each iteration, i.e., the penalized objective varies at different iterations, the conclusion in \cite{Beck-2009-APG,devolder2014first,schnidt-2011-nips} for the unconstrained problems cannot be directly used for procedures (\ref{apg1})-(\ref{apg3}) and (\ref{apgns1})-(\ref{apgns3}). The increasing penalty parameters make the convergence analysis more challenging.

\subsubsection{Difference from the accelerated gradient sliding method}

\cite{Uribe-2017} combined Nesterov's smoothing technique \cite{Nesterov-smooth} with the accelerated gradient sliding methods \cite{lan16gs,lan16gs2,lan16gs3} to solve the nonsmooth problem (\ref{problem3}) with $f(\x)=0$. In fact, when fixing the penalty parameter as a large one of the order $O(1/\epsilon)$, Algorithm \ref{D-NG} is similar to the one in \cite[Section 6.3]{Uribe-2017}. However, our method adopts increasing penalty parameters such that it avoids having to set a large inner iteration number $T_k$ and a small step-size $\eta_k$ at the beginning of the outer loop, as shown in Corollary \ref{lemma_cons-noncons}. On the other hand, when $f(\x)\neq 0$, as explained in Section \ref{sec:dng}, $\nabla f(\y^k)$ is flooded if we set a large and fixed penalty parameter.


\subsubsection{Difference from the D-NC and D-NG in \cite{Jakovetic-2014}}
Algorithm \ref{D-NC} can be seen as an improvement over the D-NC proposed in \cite{Jakovetic-2014}. Both Algorithm \ref{D-NC} and D-NC use Nesterov's acceleration technique and multi-consensus, and both attain the optimal computation complexity for the nonstrongly convex problems. However, Algorithm \ref{D-NC} is motivated by a constraint-penalty approach while D-NC is developed from the inexact accelerated gradient method \cite{devolder2014first} directly. Moreover, Algorithm \ref{D-NC} can solve both the strongly convex and nonstrongly convex problems while \cite{Jakovetic-2014} only studied the nonstrongly convex case.



As for Algorithm \ref{D-NG}, consider the simple case with $h(\x)=0$ and $\frac{\beta_0}{\vartheta_k}=\frac{k+1}{c}$, then steps (\ref{apgns2}) and (\ref{apgns3}) become
\begin{eqnarray}
\x^{k+1}=\frac{L\y^k+(k+1)\W\y^k/c}{L+(k+1)/c}-\frac{\nabla f(\y^k)}{L+(k+1)/c}.\notag
\end{eqnarray}
Thus, when $(k+1)/c\gg L$, we have $\x^{k+1}\approx \W\y^k-\frac{c}{k+1}\nabla f(\y^k)$ and it approximates the D-NG in \cite{Jakovetic-2014}. Algorithm \ref{D-NG} gives a different explanation of the D-NG, and it improves the D-NG in the sense that it handles a possible nondifferentiable function $h_i(\x)$. The complexity of D-NG is $O\left(\frac{1}{\epsilon(1-\sigma_2(\W))^{1+\xi}}\log\frac{1}{\epsilon}\right)$, where $\xi$ is a small constant. Our complexity, i.e., $O\left(\frac{1}{\epsilon\sqrt{1-\sigma_2(\W)}}\right)$, is better because theirs has the extra $\log\frac{1}{\epsilon}$ factor and is more sensitive to $1-\sigma_2(\W)$.

\section{Proof of Theorems}\label{sec:analysis}
\subsection{Supporting Lemmas}
Before providing a comprehensive convergence analysis for Algorithms \ref{D-NC} and \ref{D-NG}, we first present some useful technical lemmas. We first give the following easy-to-identify identifies.
\begin{lemma}\label{lemma01}
For any $\x,\y,\z,\w\in\R^{m\times n}$, we have the following two identities:
\begin{eqnarray}
\begin{aligned}\notag
&2\hspace*{-0.05cm}\<\x\hspace*{-0.05cm}-\hspace*{-0.05cm}\z,\y\hspace*{-0.05cm}-\hspace*{-0.05cm}\z\>\hspace*{-0.08cm}=\hspace*{-0.08cm}\|\x\hspace*{-0.05cm}-\hspace*{-0.05cm}\z\|_F^2\hspace*{-0.08cm}+\hspace*{-0.08cm}\|\y\hspace*{-0.05cm}-\hspace*{-0.05cm}\z\|_F^2\hspace*{-0.08cm}-\hspace*{-0.08cm}\|\x\hspace*{-0.05cm}-\hspace*{-0.05cm}\y\|_F^2,\\
&2\hspace*{-0.05cm}\<\x\hspace*{-0.05cm}-\hspace*{-0.05cm}\z,\y\hspace*{-0.05cm}-\hspace*{-0.05cm}\w\>\hspace*{-0.08cm}=\hspace*{-0.08cm}\|\y\hspace*{-0.05cm}-\hspace*{-0.05cm}\z\|_F^2\hspace*{-0.08cm}-\hspace*{-0.08cm}\|\w\hspace*{-0.05cm}-\hspace*{-0.05cm}\z\|_F^2\hspace*{-0.08cm}+\hspace*{-0.08cm}\|\x\hspace*{-0.05cm}-\hspace*{-0.05cm}\w\|_F^2\hspace*{-0.08cm}-\hspace*{-0.08cm}\|\x\hspace*{-0.05cm}-\hspace*{-0.05cm}\y\|_F^2.
\end{aligned}
\end{eqnarray}
\end{lemma}

In the following Lemma, we bound the Lagrange multiplier, which is useful for the complexity analysis in the distributed optimization community.
\begin{lemma}\label{lemma02}
Assume that Assumptions \ref{assumption_f}, \ref{assumption_h} and \ref{assumption_w} hold with $\mu\geq 0$. Then, we have the following properties:
\begin{enumerate}
\item There exists a pair of KKT points $(\x^*,\lambda^*)$ of saddle point problem $\min_{\x}\max_{\lambda} f(\x)+\<\lambda,\Pi\x\>$, such that $\|\lambda^*\|_F\leq\|\nabla f(\x^*)\|_F$.
\item There exists a pair of KKT points $(\x^*,\lambda^*)$ of saddle point problem $\min_{\x}\max_{\lambda} f(\x)+\<\lambda,\U\x\>$, such that $\|\lambda^*\|_F\leq \frac{\|\nabla f(\x^*)\|_F}{\sqrt{1-\sigma_2(\W)}}$.
\item There exists a pair of KKT points $(\x^*,\lambda^*)$ of saddle point problem $\min_{\x}\max_{\lambda} F(\x)+\<\lambda,\U\x\>$, such that $\|\lambda^*\|_F\leq \frac{\sqrt{m}M+\|\nabla f(\x^*)\|_F}{\sqrt{1-\sigma_2(\W)}}$.
\end{enumerate}
\end{lemma}

The proof can be found in \cite[Theorem 2]{Lam-2017}. The following lemma is a corollary of the saddle point property.
\begin{lemma}\cite{Li-2016-ALADM}\label{lemma03}
If $f(\x)$ is convex and $(\x^*,\lambda^*)$ is a pair of KKT points of saddle point problem $\min_{\x}\max_{\lambda} f(\x)+\<\lambda,\A\x\>$, then we have $f(\x)-f(\x^*)+\<\lambda^*,\A\x\>\geq 0$ for all $\x$.
\end{lemma}

The following lemma bounds the consensus violation of $\|\Pi\x\|_F$ from $\|\U\x\|_F$.
\begin{lemma}\label{lemma04}
Assume that Assumption \ref{assumption_w} holds. Then, we have $\|\Pi\x\|_F\hspace*{-0.08cm}\leq\hspace*{-0.08cm}\frac{1}{\sqrt{1-\sigma_2(\W)}}\|\U\x\|_F$.
\end{lemma}
\begin{proof}
From Assumption \ref{assumption_w}, we know $\U\1=\0$, $\U=\U^T$, and $\mbox{rank}(\U)=m-1$. For any $\x\in\R^{m\times n}$, denote $\overline\x=\Pi\x=\x-\frac{1}{m}\1\1^T\x$. Since $\1^T\overline\x=\0$, we know $\overline\x$ is orthogonal to the null space of $\U$, and thus it belongs to the row (i.e., column) space of $\U$. Let $\V\Sigma\V^T=\U$ be its economical SVD with $\V\in\R^{m\times (m-1)}$. Then we have
\begin{eqnarray}
\begin{aligned}\notag
&\|\U\x\|_F^2=\|\U\overline\x\|_F^2=\sum_{i=1}^n\overline\x_i^T\U^2\overline\x_i=\sum_{i=1}^n(\V^T\overline\x_i)^T\Sigma^2(\V^T\overline\x_i)\\
&\geq (1-\sigma_2(\W))\sum_{i=1}^n\|\V^T\overline\x_i\|_F^2=(1-\sigma_2(\W))\|\V^T\overline\x\|_F^2\\
&\overset{a}=(1-\sigma_2(\W))\|\overline\x\|_F^2=(1-\sigma_2(\W))\|\Pi\x\|_F^2,
\end{aligned}
\end{eqnarray}
where we denote $\x_i$ to be the $i$th column of $\x$, and $\overset{a}=$ follows from the fact that $\overline\x$ belongs to the column space of $\U$, i.e., there exists $\alpha\in\R^{(m-1)\times n}$ such tht $\overline\x=\V\alpha$.
\hfill$\Box$\end{proof}

At last, we present the following lemma, which can be used to analyze the algorithms with inexact subproblem computation.
\begin{lemma}\cite{schnidt-2011-nips}\label{schnidt-lemma}
Assume that $(s_k)$ is a sequence with increasing scalars and $(v_k)$, $(\alpha_i)$ are sequences with nonnegative scalars, $v_0^2\leq s_0$. If $v_k^2\leq s_k+\sum_{i=1}^k\alpha_i v_i$, then we have $v_k\leq \frac{1}{2}\sum_{i=1}^k\alpha_i+\sqrt{\left(\frac{1}{2}\sum_{i=1}^k\alpha_i\right)^2+s_k}$.
\end{lemma}

\subsection{Complexity Analysis for Algorithm \ref{D-NC}}\label{section:dnc}

\subsubsection{Inner Loop}\label{section:dnc1}
Before proving the convergence of procedure (\ref{apg1})-(\ref{apg3}), we first establish the required precision to approximate $\alpha(\z^k)$ for an $\varepsilon_k$-accurate solution of the subproblem in (\ref{apg3}).
\begin{lemma}\label{lemma1}
Let $\z^{k,T_k}$ be obtained by (\ref{aac}) and $\x^{k+1}=\frac{L\vartheta_k\z^k+\beta_0\z^{k,T_k}}{L\vartheta_k+\beta_0}$. If $\|\z^{k,T_k}-\1\alpha(\z^k)^T\|_F^2\leq \frac{2\vartheta_k\varepsilon_k}{\beta_0}$, then we have
\begin{eqnarray}
\qquad\begin{aligned}\label{inexact_assumption}
&\frac{L}{2}\left\|\x^{k+1}-\z^k\right\|_F^2+\frac{\beta_0}{2\vartheta_k}\|\Pi\x^{k+1}\|_F^2\\
\leq&\min_{\x\in\R^{m\times n}} \frac{L}{2}\left\|\x-\z^k\right\|_F^2+\frac{\beta_0}{2\vartheta_k}\|\Pi\x\|_F^2+\varepsilon_k.
\end{aligned}
\end{eqnarray}
\end{lemma}
\begin{proof}
Define $\x^{k,*}=\argmin_{\x}\frac{L}{2}\|\x-\z^k\|_F^2+\frac{\beta_0}{2\vartheta_k}\|\Pi\x\|_F^2$, $\widetilde\x^{k,*}=\frac{1}{m}\1\1^T\x^{k,*}$, and $\widetilde\z^k=\frac{1}{m}\1\1^T\z^k$. From the optimality condition of $\x^{k,*}$, we have
\begin{eqnarray}
\0=L(\x^{k,*}-\z^k)+\frac{\beta_0}{\vartheta_k}\Pi^2\x^{k,*}.\label{cont1}
\end{eqnarray}
From $\Pi=\Pi^2$ and its definition, we have $\0=L(\x^{k,*}-\z^k)+\frac{\beta_0}{\vartheta_k}(\x^{k,*}-\widetilde\x^{k,*})$, which leads to $\x^{k,*}=\frac{L\vartheta_k\z^k+\beta_0\widetilde\x^{k,*}}{L\vartheta_k+\beta_0}$. Multiplying both sides of (\ref{cont1}) by $\frac{1}{m}\1\1^T$, and using $\1^T\Pi=\0$, we have $\widetilde\x^{k,*}=\widetilde\z^k$, which further gives $\x^{k,*}=\frac{L\vartheta_k\z^k+\beta_0\widetilde\z^k}{L\vartheta_k+\beta_0}$. On the other hand, we have
\begin{eqnarray}
\hspace*{-0.7cm}\begin{aligned}\label{cont4}
& \frac{L}{2}\left\|\x^{k+1}-\z^k\right\|_F^2+\frac{\beta_0}{2\vartheta_k}\|\Pi\x^{k+1}\|_F^2\\
&-\frac{L}{2}\left\|\x^{k,*}-\z^k\right\|_F^2-\frac{\beta_0}{2\vartheta_k}\|\Pi\x^{k,*}\|_F^2\\
\overset{a}=&L\<\x^{k,*}-\z^k,\x^{k+1}-\x^{k,*}\>+\frac{L}{2}\|\x^{k+1}-\x^{k,*}\|_F^2\\
&+\hspace*{-0.05cm}\frac{\beta_0}{\vartheta_k}\hspace*{-0.05cm}\<\Pi^2\x^{k,*},\x^{k+1}\hspace*{-0.05cm}-\hspace*{-0.05cm}\x^{k,*}\>\hspace*{-0.05cm}+\hspace*{-0.05cm}\frac{\beta_0}{2\vartheta_k}\|\Pi(\x^{k+1}\hspace*{-0.05cm}-\hspace*{-0.05cm}\x^{k,*})\|_F^2\hspace*{-1cm}\\
\overset{b}=&\frac{L}{2}\|\x^{k+1}-\x^{k,*}\|_F^2+\frac{\beta_0}{2\vartheta_k}\|\Pi(\x^{k+1}-\x^{k,*})\|_F^2\\
\overset{c}=&\frac{\beta_0^2}{(L\vartheta_k\hspace*{-0.05cm}+\hspace*{-0.05cm}\beta_0)^2}\hspace*{-0.05cm}\left( \hspace*{-0.05cm}\frac{L}{2}\|\z^{k,T_k}\hspace*{-0.05cm}-\hspace*{-0.05cm}\widetilde\z^k\|_F^2\hspace*{-0.05cm}+\hspace*{-0.05cm}\frac{\beta_0}{2\vartheta_k}\|\Pi(\z^{k,T_k}\hspace*{-0.05cm}-\hspace*{-0.05cm}\widetilde\z^k)\|_F^2 \hspace*{-0.05cm}\right)\hspace*{-2cm}\\
\overset{d}\leq&\frac{\beta_0^2}{2\vartheta_k(L\vartheta_k+\beta_0)}\|\z^{k,T_k}\hspace*{-0.05cm}-\hspace*{-0.05cm}\widetilde\z^k\|_F^2\leq\frac{\beta_0}{2\vartheta_k}\|\z^{k,T_k}\hspace*{-0.05cm}-\hspace*{-0.05cm}\widetilde\z^k\|_F^2.\hspace*{-2cm}
\end{aligned}
\end{eqnarray}
where we use Lemma \ref{lemma01} in $\overset{a}=$, (\ref{cont1}) in $\overset{b}=$, the definition of $\x^{k+1}$ and $\x^{k,*}=\frac{L\vartheta_k\z^k+\beta_0\widetilde\z^k}{L\vartheta_k+\beta_0}$ in $\overset{c}=$, and $\|\Pi\z\|_F\leq\|\z\|_F$ in $\overset{d}\leq$.
\hfill$\Box$\end{proof}

Now we consider the iteration number of the accelerated average consensus in (\ref{aac}) to solve the subproblem in (\ref{apg3}) such that (\ref{inexact_assumption}) is satisfied.  From \cite[Proposition 3]{liu-2011}, we have
\begin{eqnarray}
\hspace*{-0.7cm}&&\|\z^{k,T_k}-\1\alpha(\z^k)^T\|_F\leq \left(\frac{\sigma_2(\W)}{1+\sqrt{1-\sigma_2^2(\W)}}\right)^{T_k}\hspace*{-0.1cm}\|\Pi\z^k\|_F\label{consens_il}\\
\hspace*{-0.7cm}&&\leq \hspace*{-0.1cm}\left(\hspace*{-0.1cm}\frac{\sigma_2(\W)}{1\hspace*{-0.08cm}+\hspace*{-0.08cm}\sqrt{1\hspace*{-0.08cm}-\hspace*{-0.08cm}\sigma_2(\W)}}\hspace*{-0.08cm}\right)^{T_k}\hspace*{-0.15cm}\|\Pi\z^k\|_F\hspace*{-0.08cm}=\hspace*{-0.08cm}\left(\hspace*{-0.08cm}1\hspace*{-0.08cm}-\hspace*{-0.08cm}\sqrt{1\hspace*{-0.08cm}-\hspace*{-0.08cm}\sigma_2(\W)}\hspace*{-0.02cm}\right)^{T_k}\hspace*{-0.12cm}\|\Pi\z^k\|_F\hspace*{-0.04cm}.\notag
\end{eqnarray}
Thus, from Lemma \ref{lemma1}, we only need
\begin{equation}
T_k=\frac{1}{-2\log\left(1-\sqrt{1-\sigma_2(\W)}\right)}\log \frac{\beta_0\|\Pi\z^k\|_F^2}{2\vartheta_k\varepsilon_k}\label{tk}
\end{equation}
such that
(\ref{inexact_assumption}) is satisfied.

At last, we study the property when the proximal mapping in (\ref{apg3}) is inexactly computed. When it is computed exactly, i.e., $\varepsilon_k=0$ in (\ref{inexact_assumption}), we have $L(\x^{k+1}-\z^k)+\frac{\beta_0}{\vartheta_k}\Pi^2\x^{k+1}=0$. However, when it is computed inexactly, we should modify the conclusion accordingly. Specifically, we give the following lemma.
\begin{lemma}\label{precision_lemma}
Assume that (\ref{inexact_assumption}) holds. Then, there exists $\delta^k$ with $\|\delta^k\|_F\leq\sqrt{\frac{2\varepsilon_k}{L}}$ and $\frac{\beta_0}{\vartheta_k}\|\Pi\delta^k\|_F^2\leq2\varepsilon_k$ such that
\begin{eqnarray}
\begin{aligned}\label{lemma_sub}
L(\x^{k+1}-\z^k+\delta^k)+\frac{\beta_0}{\vartheta_k}\Pi^2(\x^{k+1}+\delta^k)=0.
\end{aligned}
\end{eqnarray}
\end{lemma}
\begin{proof}
Define $\delta^k=\x^{k,*}-\x^{k+1}$. From (\ref{inexact_assumption}) and equation $\overset{b}=$ in (\ref{cont4}), we have $\|\delta^k\|_F\leq\sqrt{\frac{2\varepsilon_k}{L}}$ and $\frac{\beta_0}{\vartheta_k}\|\Pi\delta^k\|_F^2\leq2\varepsilon_k$. From (\ref{cont1}) and the definition of $\delta^k$, we have (\ref{lemma_sub}).
\hfill$\Box$\end{proof}

\subsubsection{Outer Loop}
Now we are ready to analyze procedure (\ref{apg1})-(\ref{apg3}). Define
\begin{equation}
\w^{k+1}\equiv\frac{\x^{k+1}}{\theta_k}-\frac{1-\theta_k}{\theta_k}\x^k\mbox{ for any } k\geq 0\mbox{ and }\w^0=\x^0.\notag
\end{equation}
From the definition of $\y^k$ in (\ref{apg1}), we can give the following easy-to-identify identities.
\begin{lemma}\label{lemma08}
For procedure (\ref{apg1})-(\ref{apg3}), we have
\begin{eqnarray}
\begin{aligned}\notag
&\x^*+\frac{(1-\theta_k)L}{L\theta_k-\mu}\x^k-\frac{L-\mu}{L\theta_k-\mu}\y^k=\x^*-\w^k,\\
&\theta_k\x^*+(1-\theta_k)\x^k-\x ^{k+1}=\theta_k\left(\x^*-\w^{k+1}\right).
\end{aligned}
\end{eqnarray}
\end{lemma}

Let $(\x^*,\lambda^*)$ be a pair of KKT points of saddle point problem $\min_{\x}\max_{\lambda} f(\x)+\<\lambda,\Pi\x\>$ satisfying Lemma \ref{lemma02}. Define
\begin{equation}
\rho_{k+1}=f(\x^{k+1})-f(\x^*)+\<\lambda^*,\Pi\x^{k+1}\>.\label{def_rho}
\end{equation}
From Lemma \ref{lemma03}, we know $\rho_{k+1}\geq 0$.

We first give the following lemma, which describes a progress in one iteration of procedure (\ref{apg1})-(\ref{apg3}).
\begin{lemma}\label{lemma0}
Assume that Assumption \ref{assumption_f} holds with $\mu\geq 0$. Let sequences $\{\theta_k\}$ and $\{\vartheta_k\}$ satisfy $\frac{1-\theta_k}{\vartheta_k}=\frac{1}{\vartheta_{k-1}}$ and $\theta_k\geq\frac{\mu}{L}$. Then, under the assumption of (\ref{inexact_assumption}), we have
\begin{eqnarray}
\begin{aligned}\label{lemma_nc}
&\rho_{k+1}+\frac{\vartheta_k}{2\beta_0}\left\|\frac{\beta_0}{\vartheta_k}\Pi\x^{k+1}-\lambda^*\right\|_F^2+\frac{L\theta_k^2}{2}\|\w^{k+1}-\x^*\|_F^2\\
\leq& (1-\theta_k)\rho_k+\frac{\vartheta_k}{2\beta_0}\left\|\frac{\beta_0}{\vartheta_{k-1}}\Pi\x^k-\lambda^*\right\|_F^2+\varepsilon_k\hspace*{-1cm}\\
&+\frac{(L\theta_k-\mu)\theta_k}{2}\left\|\w^k-\x^*\right\|_F^2+L\theta_k\sqrt{\frac{2\varepsilon_k}{L}}\|\w^{k+1}-\x^*\|_F.\hspace*{-2cm}
\end{aligned}
\end{eqnarray}
\end{lemma}
\begin{proof}
From the smoothness and convexity of $f(\x)$, we have
\begin{eqnarray}
\begin{aligned}\label{cont5}
&f(\x^{k+1})\\
\leq &f(\y^k)+\<\nabla f(\y^k),\x^{k+1}-\y^k\>+\frac{L}{2}\|\x^{k+1}-\y^k\|_F^2\\
= &f\hspace*{-0.02cm}(\hspace*{-0.02cm}\y^k\hspace*{-0.02cm})\hspace*{-0.09cm}+\hspace*{-0.09cm}\<\hspace*{-0.04cm}\nabla \hspace*{-0.02cm}f\hspace*{-0.02cm}(\hspace*{-0.02cm}\y^k)\hspace*{-0.02cm}\hspace*{-0.02cm},\hspace*{-0.02cm}\x\hspace*{-0.09cm}-\hspace*{-0.09cm}\y^k\hspace*{-0.04cm}\>\hspace*{-0.09cm}+\hspace*{-0.09cm}\<\hspace*{-0.04cm}\nabla \hspace*{-0.02cm}f\hspace*{-0.02cm}(\hspace*{-0.02cm}\y^k\hspace*{-0.02cm})\hspace*{-0.02cm},\hspace*{-0.02cm}\x^{k+1}\hspace*{-0.09cm}-\hspace*{-0.09cm}\x\hspace*{-0.04cm}\>\hspace*{-0.09cm}+\hspace*{-0.09cm}\frac{L}{2}\hspace*{-0.02cm}\|\hspace*{-0.02cm}\x^{k+1}\hspace*{-0.09cm}-\hspace*{-0.09cm}\y^k\hspace*{-0.02cm}\|_F^2\hspace*{-0.2cm}\hspace*{-0.4cm}\\
\leq &f(\x)\hspace*{-0.09cm}-\hspace*{-0.05cm}\frac{\mu}{2}\|\x\hspace*{-0.05cm}-\hspace*{-0.05cm}\y^k\|_F^2\hspace*{-0.05cm}+\hspace*{-0.05cm}\<\nabla f(\y^k),\x^{k+1}\hspace*{-0.05cm}-\hspace*{-0.05cm}\x\>\hspace*{-0.05cm}+\hspace*{-0.05cm}\frac{L}{2}\|\x^{k+1}\hspace*{-0.05cm}-\hspace*{-0.05cm}\y^k\|_F^2.\hspace*{-0.4cm}
\end{aligned}
\end{eqnarray}
Plugging $\z^k=\y^k-\frac{1}{L}\nabla f(\y^k)$ and (\ref{lemma_sub}) into the above inequality, we have
\begin{eqnarray}
\begin{aligned}\label{cont6}
&f(\x^{k+1})-f(\x)\notag\\
\leq&\frac{\beta_0}{\vartheta_k}\<\Pi\x^{k+1}+\Pi\delta^k,\Pi\x-\Pi\x^{k+1}\>+L\<\x^{k+1}-\y^k,\x-\y^k\>\\
&+L\<\delta^k,\x-\x^{k+1}\>-\frac{\mu}{2}\|\x-\y^k\|_F^2-\frac{L}{2}\|\x^{k+1}-\y^k\|_F^2
\end{aligned}
\end{eqnarray}
When we apply (\ref{cont6}) first with $\x=\x^k$ and then with $\x=\x^*$, we obtain two inequalities. Multiplying the first inequality by $(1-\theta_k)$, multiplying the second by $\theta_k$, adding them together with $\<\lambda^*,\Pi\x^{k+1}-(1-\theta_k)\Pi\x^k\>$ to both sides, and using $\Pi\x^*=\0$, we have
\begin{eqnarray}
\begin{aligned}\notag
&f(\x^{k+1})-(1-\theta_k)f(\x^k)-\theta_k f(\x^*)\\
&+\<\lambda^*,\Pi\x^{k+1}-(1-\theta_k)\Pi\x^k\>
\end{aligned}
\end{eqnarray}
\begin{eqnarray}
\begin{aligned}\notag
\leq&\<\frac{\beta_0}{\vartheta_k}(\Pi\x^{k+1}+\Pi\delta^k)-\lambda^*,(1-\theta_k)\Pi\x^k-\Pi\x^{k+1}\>\\
&+L\<\x^{k+1}-\y^k,(1-\theta_k)\x^k+\theta_k\x^*-\y^k\>\\
&+L\<\delta^k,(1-\theta_k)\x^k+\theta_k\x^*-\x^{k+1}\>\\
&-\frac{\mu\theta_k}{2}\|\x^*-\y^k\|_F^2-\frac{L}{2}\|\x^{k+1}-\y^k\|_F^2\\
\overset{a}=&\frac{\vartheta_k}{\beta_0}\<\frac{\beta_0}{\vartheta_k}(\Pi\x^{k+1}+\Pi\delta^k)-\lambda^*,\frac{\beta_0}{\vartheta_{k-1}}\Pi\x^k-\frac{\beta_0}{\vartheta_k}\Pi\x^{k+1}\>\\
&+L\<\x^{k+1}-\y^k,(1-\theta_k)\x^k+\theta_k\x^*-\y^k\>\\
&+L\<\delta^k,(1-\theta_k)\x^k+\theta_k\x^*-\x^{k+1}\>\\
&-\frac{\mu\theta_k}{2}\|\x^*-\y^k\|_F^2-\frac{L}{2}\|\x^{k+1}-\y^k\|_F^2,
\end{aligned}
\end{eqnarray}
where we use $\frac{1-\theta_k}{\vartheta_k}=\frac{1}{\vartheta_{k-1}}$ in $\overset{a}=$. Applying the identities in Lemma \ref{lemma01} to the two inner products, we have
\begin{eqnarray}
\begin{aligned}\notag
&\rho_{k+1}-(1-\theta_k)\rho_k\\
\leq&\frac{\vartheta_k}{2\beta_0}\hspace*{-0.12cm}\left[ \left\|\frac{\beta_0}{\vartheta_{k-1}}\Pi\x^k\hspace*{-0.08cm}-\hspace*{-0.08cm}\lambda^*\right\|_F^2\hspace*{-0.12cm}+\hspace*{-0.08cm}\left\|\frac{\beta_0}{\vartheta_k}\Pi\x^{k+1}\hspace*{-0.08cm}-\hspace*{-0.08cm}\frac{\beta_0}{\vartheta_k}\hspace*{-0.04cm}(\Pi\x^{k+1}\hspace*{-0.08cm}+\hspace*{-0.08cm}\Pi\delta^k)\right\|_F^2\right.\\
&\left.-\hspace*{-0.055cm}\left\|\frac{\beta_0}{\vartheta_k}\Pi\x^{k+1}\hspace*{-0.055cm}-\hspace*{-0.055cm}\lambda^*\right\|_F^2\hspace*{-0.15cm}-\hspace*{-0.055cm}\left\|\frac{\beta_0}{\vartheta_{k-1}}\Pi\x^k\hspace*{-0.055cm}-\hspace*{-0.055cm}\frac{\beta_0}{\vartheta_k}(\Pi\x^{k+1}\hspace*{-0.055cm}+\hspace*{-0.055cm}\Pi\delta^k)\right\|_F^2 \right]\\
&+\hspace*{-0.091cm}\frac{L}{2}\hspace*{-0.091cm}\left[ \|(\hspace*{-0.02cm}1\hspace*{-0.091cm}-\hspace*{-0.091cm}\theta_k\hspace*{-0.02cm})\x^k\hspace*{-0.091cm}+\hspace*{-0.091cm}\theta_k\x^*\hspace*{-0.091cm}-\hspace*{-0.091cm}\y^k\|_F^2\hspace*{-0.091cm}-\hspace*{-0.091cm}\|(\hspace*{-0.02cm}1\hspace*{-0.091cm}-\hspace*{-0.091cm}\theta_k\hspace*{-0.02cm})\x^k\hspace*{-0.091cm}+\hspace*{-0.091cm}\theta_k\x^*\hspace*{-0.091cm}-\hspace*{-0.091cm}\x^{k+1}\|_F^2 \right]\\
&+L\<\delta^k,(1-\theta_k)\x^k+\theta_k\x^*-\x^{k+1}\>-\frac{\mu\theta_k}{2}\|\x^*-\y^k\|_F^2\\
\overset{b}\leq&\frac{\vartheta_k}{2\beta_0}\hspace*{-0.07cm}\left[ \left\|\frac{\beta_0}{\vartheta_{k-1}}\Pi\x^k\hspace*{-0.07cm}-\hspace*{-0.07cm}\lambda^*\right\|_F^2\hspace*{-0.12cm}-\hspace*{-0.07cm}\left\|\frac{\beta_0}{\vartheta_k}\Pi\x^{k+1}\hspace*{-0.07cm}-\hspace*{-0.07cm}\lambda^*\right\|_F^2\hspace*{-0.12cm}+\hspace*{-0.07cm}\frac{\beta_0^2}{\vartheta_k^2}\|\Pi\delta^k\|_F^2\right]\\
&+\frac{L\theta_k^2}{2}\left[ \left\|\frac{\y^k}{\theta_k}-\frac{1-\theta_k}{\theta_k}\x^k-\x^*\right\|_F^2-\|\w^{k+1}-\x^*\|_F^2\right]\\
&-L\theta_k\<\delta^k,\w^{k+1}-\x^*\>-\frac{\mu\theta_k}{2}\|\x^*-\y^k\|_F^2.
\end{aligned}
\end{eqnarray}
where $\overset{b}\leq$ follows from the second identity in Lemma \ref{lemma08}. By reorganizing the terms in $\frac{\y^k}{\theta_k}-\frac{1-\theta_k}{\theta_k}\x^k-\x^*$ carefully, we have
\begin{eqnarray}
\begin{aligned}\notag
&\frac{L\theta_k^2}{2}\left\|\frac{\y^k}{\theta_k}-\frac{1-\theta_k}{\theta_k}\x^k-\x^*\right\|_F^2\\
=&\frac{L\theta_k^2}{2}\hspace*{-0.09cm}\left\|\hspace*{-0.09cm}\frac{\mu}{L\theta_k}\hspace*{-0.09cm}\left(\hspace*{-0.04cm}\y^k\hspace*{-0.09cm}-\hspace*{-0.09cm}\x^*\hspace*{-0.04cm}\right)\hspace*{-0.15cm}+\hspace*{-0.15cm}\left(\hspace*{-0.13cm}1\hspace*{-0.09cm}-\hspace*{-0.09cm}\frac{\mu}{L\theta_k}\hspace*{-0.08cm}\right)\hspace*{-0.18cm}\left(\hspace*{-0.09cm}\frac{L-\mu}{L\theta_k\hspace*{-0.09cm}-\hspace*{-0.09cm}\mu}\y^k\hspace*{-0.09cm}-\hspace*{-0.09cm}\frac{(\hspace*{-0.05cm}1\hspace*{-0.09cm}-\hspace*{-0.09cm}\theta_k\hspace*{-0.05cm})L}{L\theta_k\hspace*{-0.09cm}-\hspace*{-0.09cm}\mu}\x^k\hspace*{-0.09cm}-\hspace*{-0.09cm}\x^*\hspace*{-0.09cm}\right)\hspace*{-0.09cm}\right\|_F^2\\
\overset{c}\leq&\frac{\mu\theta_k}{2}\|\y^k\hspace*{-0.09cm}-\hspace*{-0.09cm}\x^*\|_F^2\hspace*{-0.07cm}+\hspace*{-0.08cm}\frac{(\hspace*{-0.04cm}L\theta_k\hspace*{-0.09cm}-\hspace*{-0.09cm}\mu\hspace*{-0.04cm})\theta_k}{2}\hspace*{-0.09cm}\left\|\hspace*{-0.04cm}\frac{L-\mu}{L\theta_k\hspace*{-0.09cm}-\hspace*{-0.09cm}\mu}\y^k\hspace*{-0.09cm}-\hspace*{-0.09cm}\frac{(\hspace*{-0.05cm}1\hspace*{-0.09cm}-\hspace*{-0.09cm}\theta_k\hspace*{-0.05cm})L}{L\theta_k\hspace*{-0.09cm}-\hspace*{-0.09cm}\mu}\x^k\hspace*{-0.09cm}-\hspace*{-0.09cm}\x^*\hspace*{-0.03cm}\right\|_F^2\\
\overset{d}=&\frac{\mu\theta_k}{2}\|\y^k-\x^*\|_F^2+\frac{(L\theta_k-\mu)\theta_k}{2}\left\|\w^k-\x^*\right\|_F^2,
\end{aligned}
\end{eqnarray}
where we let $\frac{\mu}{L\theta_k}\leq 1$, and use Jensen's inequality for $\|\cdot\|_F^2$ in $\overset{c}\leq$, and the first identify in Lemma \ref{lemma08} in $\overset{d}=$. Plugging it into the above inequality and using the bounds for $\|\delta^k\|_F$ and $\|\Pi\delta^k\|_F$ in Lemma \ref{precision_lemma}, we get (\ref{lemma_nc}).
\hfill$\Box$\end{proof}

Due to the term $\|\w^{k+1}-\x^*\|_F$ on the right hand side of (\ref{lemma_nc}), recursion (\ref{lemma_nc}) cannot be directly telescoped unless we assume the boundness of $\|\w^{k+1}-\x^*\|_F$. Lemma \ref{schnidt-lemma} can be used to avoid such boundness assumption. Now, we use Lemmas \ref{lemma0} and \ref{schnidt-lemma} to analyze procedure (\ref{apg1})-(\ref{apg3}). The following theorem shows the convergence for strongly convex problems.
\begin{theorem}\label{theorem_DNC_sc}
Assume that Assumptions \ref{assumption_f}, \ref{assumption_w} and (\ref{inexact_assumption}) hold with $\mu>0$, and $\varepsilon_k\leq (1-(1+\tau)\theta)^{k+1}$ holds for all $k\leq K$, where $1>\tau>0$ can be any small constant. Let sequences $\{\theta_k\}$ and $\{\vartheta_k\}$ satisfy $\theta_k=\theta=\sqrt{\frac{\mu}{L}}$ for all $k$, and $\vartheta_k=(1-\theta)^{k+1}$. Then, we have
\begin{eqnarray}
\begin{aligned}\notag
&f(\x^{K+1})-f(\x^*)\leq C_2(1-\theta)^{K+1},\\
&\|\Pi\x^{K+1}\|_F\leq C_3(1-\theta)^{K+1},\\
&\|\x^{K+1}-\x^*\|_F^2\leq C_4(1-\theta)^{K+1},\\
&f\left(\alpha(\x^{K+1})\right)-f(\x^*)\leq C_5(1-\theta)^{K+1}+\frac{LC_3^2}{2}(1-\theta)^{2K+2},
\end{aligned}
\end{eqnarray}
where $C_2=C_6+\|\lambda^*\|_FC_3$, $C_3=\frac{\sqrt{2\beta_0C_6}+\|\lambda^*\|_F}{\beta_0}$, $C_4=\frac{2C_6}{\mu}$, $C_5=(\|\nabla f(\x^*)\|_F+L\sqrt{C_4})C_3+C_2$ and $C_6=\frac{18}{\tau^2\theta^2}+2\left(f(\x^0)-f(\x^*)+\<\lambda^*,\Pi\x^0\>\right)+\frac{1}{\beta_0}\|\beta_0\Pi\x^0-\lambda^*\|_F^2+\mu\|\x^0-\x^*\|_F^2$.
\end{theorem}
\begin{proof}
The setting of $\theta=\sqrt{\frac{\mu}{L}}$ satisfies
\begin{equation}
(L\theta-\mu)\theta=L\theta^2(1-\theta).\label{cont12}
\end{equation}
Sequences $\{\theta_k\}$ and $\{\vartheta_k\}$ satisfy the requirement in Lemma \ref{lemma0}. Define the Lyapunov function $\ell_{k+1}$ as follows:
\begin{eqnarray}
\begin{aligned}\notag
\ell_{k+1}^2=\frac{\rho_{k+1}+\frac{\vartheta_k}{2\beta_0}\left\|\frac{\beta_0}{\vartheta_k}\Pi\x^{k+1}-\lambda^*\right\|_F^2+\frac{L\theta^2}{2}\|\w^{k+1}-\x^*\|_F^2}{(1-\theta)^{k+1}},
\end{aligned}
\end{eqnarray}
where $\rho_k$ is defined in (\ref{def_rho}). Dividing both sides of (\ref{lemma_nc}) by $(1-\theta)^{k+1}$, and using (\ref{cont12}) and $\vartheta_k=(1-\theta)\vartheta_{k-1}$, we have
\begin{eqnarray}
\begin{aligned}\notag
\ell_{k+1}^2-\ell_k^2\leq\frac{\varepsilon_k}{(1-\theta)^{k+1}}+\frac{L\theta}{(1-\theta)^{k+1}}\sqrt{\frac{2\varepsilon_k}{L}}\|\w^{k+1}-\x^*\|_F.
\end{aligned}
\end{eqnarray}
Summing over $k=0,1,\cdots,K$, we have
\begin{eqnarray}
\begin{aligned}\notag
&\ell_{K+1}^2-\ell_0^2\\
\leq&\sum_{k=0}^K\frac{\varepsilon_k}{(1-\theta)^{k+1}}+\sum_{k=0}^K\frac{L\theta}{(1-\theta)^{k+1}}\sqrt{\frac{2\varepsilon_k}{L}}\|\w^{k+1}-\x^*\|_F\\
=&\sum_{k=0}^K\frac{\varepsilon_k}{(1-\theta)^{k+1}}+\sum_{k=1}^{K+1}\frac{2\sqrt{\varepsilon_{k-1}}}{(1-\theta)^{k/2}}\sqrt{\frac{L\theta^2}{2(1-\theta)^{k}}}\|\w^k-\x^*\|_F\\
\overset{a}\leq&\sum_{k=0}^K\frac{\varepsilon_k}{(1-\theta)^{k+1}}+\sum_{k=1}^{K+1}\frac{2\sqrt{\varepsilon_{k-1}}}{(1-\theta)^{k/2}}\ell_k,
\end{aligned}
\end{eqnarray}
where we use the definition of $\ell_k$ and $\rho_k\geq 0$ in $\overset{a}\leq$. Letting $s_{k+1}=\sum_{t=0}^k\frac{\varepsilon_t}{(1-\theta)^{t+1}}+\ell_0^2$ and $\alpha_k=\frac{2\sqrt{\varepsilon_{k-1}}}{(1-\theta)^{k/2}}$, then we have $\ell_{k+1}^2\leq s_{k+1}+\sum_{i=1}^{k+1}\alpha_i\ell_i$ and $\ell_0^2= s_0$. From Lemma \ref{schnidt-lemma}, we have $\ell_{k+1}\leq \frac{1}{2}\sum_{i=1}^{k+1}\alpha_i+\sqrt{\left(\frac{1}{2}\sum_{i=1}^{k+1}\alpha_i\right)^2+s_{k+1}}$. Letting $\varepsilon_k\leq(1-(1+\tau)\theta)^{k+1}$, and after some simple computing, we obtain
\begin{eqnarray}
\begin{aligned}\notag
\ell_{K+1}^2\leq&\left(\sum_{k=1}^{K+1}\frac{2\sqrt{\varepsilon_{k-1}}}{(1-\theta)^{k/2}}\right)^2+\sum_{k=0}^K\frac{2\varepsilon_k}{(1-\theta)^{k+1}}+2\ell_0^2\\
\leq&\frac{18}{\tau^2\theta^2}+2\ell_0^2\equiv C_6.
\end{aligned}
\end{eqnarray}
From the definition of $\ell_{K+1}$ and $\rho_k\geq 0$, we get the second conclusion. From the definition of $\rho_{k+1}$, we have $f(\x^{k+1})-f(\x^*)\leq \rho_{k+1}+\|\lambda^*\|_F\|\Pi\x^{k+1}\|_F$, which further leads to the first conclusion. Since $f(\x)+\<\lambda^*,\Pi\x\>$ is $\mu$-strongly convex over $\x$ and $\x^*=\argmin_{\x} f(\x)+\<\lambda^*,\Pi\x\>$, we have $\frac{\mu}{2}\|\x^{K+1}-\x^*\|_F^2\leq f(\x^{K+1})+\<\lambda^*,\Pi\x^{K+1}\>-f(\x^*)-\<\lambda^*,\Pi\x^*\>=\rho_{K+1}\leq C_6(1-\theta)^{K+1}$, i.e., the third conclusion. For the fourth conclusion, we have
\begin{eqnarray}
\begin{aligned}\label{cont13}
&f\left(\alpha(\x^{K+1}\right)-f(\x^*)\\
=&f\left(\alpha(\x^{K+1}\right)-f(\x^{K+1})+f(\x^{K+1})-f(\x^*)\\
\overset{b}\leq&\hspace*{-0.07cm} \<\nabla  f(\x^{K+1}),-\Pi\x^{K+1}\>\hspace*{-0.09cm}+\hspace*{-0.09cm}\frac{L}{2}\left\|\Pi\x^{K+1}\right\|_F^2\hspace*{-0.09cm}+\hspace*{-0.09cm}f(\x^{K+1})\hspace*{-0.09cm}-\hspace*{-0.09cm}f(\x^*)\hspace*{-0.1cm}\\
\overset{c}\leq&\hspace*{-0.05cm}\left(\hspace*{-0.05cm}\|\nabla f(\x^*)\|_F\hspace*{-0.08cm}+\hspace*{-0.08cm}L\|\x^{K+1}\hspace*{-0.08cm}-\hspace*{-0.08cm}\x^*\|_F\hspace*{-0.05cm}\right)\hspace*{-0.08cm}\|\Pi\x^{K+1}\|_F\hspace*{-0.08cm}+\hspace*{-0.08cm}\frac{L}{2}\|\Pi\x^{K+1}\|_F^2\hspace*{-0.1cm}\\
&+f(\x^{K+1})-f(\x^*),
\end{aligned}
\end{eqnarray}
where we use the smoothness of $f(\x)$ and the definition of $\Pi$ in $\overset{b}\leq$ and $\overset{c}\leq$.
\hfill$\Box$\end{proof}

In the following theorem, we consider the case that $f(\x)$ is nonstrongly convex.
\begin{theorem}\label{theorem_DNC_ns}
Assume that Assumptions \ref{assumption_f}, \ref{assumption_w} and (\ref{inexact_assumption}) hold with $\mu=0$ and $\varepsilon_k\leq\frac{1}{(k+1)^{6}}$ for all $k\leq K$. Let sequences $\{\theta_k\}$ and $\{\vartheta_k\}$ satisfy $\theta_0=1$, $\frac{1-\theta_k}{\theta_k^2}=\frac{1}{\theta_{k-1}^2}$, and $\vartheta_k=\theta_k^2$. Then, we have
\begin{eqnarray}
\begin{aligned}\notag
&f(\x^{K+1})-f(\x^*)\leq \frac{C_7}{(K+2)^2},\\
&\|\Pi\x^{K+1}\|_F\leq \frac{C_8}{(K+2)^2},\\
&\|\x^{K+1}-\x^*\|_F^2\leq C_9,\\
&f\left(\alpha(\x^{K+1})\right)-f(\x^*)\leq \frac{C_{10}}{(K+2)^2}+\frac{LC_8^2}{2(K+2)^4},
\end{aligned}
\end{eqnarray}
where $C_7=4C_{11}+\|\nabla f(\x^*)\|_FC_8$, $C_8=\frac{4\sqrt{2\beta_0C_{11}}+4\|\nabla f(\x^*)\|_F}{\beta_0}$, $C_9=\frac{2C_{11}}{L}$, $C_{10}=(\|\nabla f(\x^*)\|_F+L\sqrt{C_9})C_8+C_7$, and $C_{11}=5+\frac{\|\nabla f(\x^*)\|_F^2}{\beta_0}+L\|\x^0-\x^*\|_F^2$.
\end{theorem}
\begin{proof}
Define the following Lyapunov function $\ell_{k+1}$
\begin{eqnarray}
\begin{aligned}\notag
&\ell_{k+1}^2=\frac{\rho_{k+1}}{\theta_k^2}+\frac{1}{2\beta_0}\left\|\frac{\beta_0}{\vartheta_k}\Pi\x^{k+1}-\lambda^*\right\|_F^2+\frac{L}{2}\|\w^{k+1}-\x^*\|_F^2.\\
\end{aligned}
\end{eqnarray}
Dividing both sides of (\ref{lemma_nc}) by $\theta_k^2$, using $\vartheta_k=\theta_k^2$ and $\frac{1-\theta_k}{\theta_k^2}=\frac{1}{\theta_{k-1}^2}$, we have
\begin{eqnarray}
\begin{aligned}\notag
\ell_{k+1}^2-\ell_k^2\leq\frac{\varepsilon_k}{\theta_k^2}+\frac{L}{\theta_k}\sqrt{\frac{2\varepsilon_k}{L}}\|\w^{k+1}-\x^*\|_F.
\end{aligned}
\end{eqnarray}
Similar to the proof of Theorem \ref{theorem_DNC_sc}, we obtain
\begin{eqnarray}
\begin{aligned}\notag
\ell_{K+1}^2-\ell_0^2\leq\sum_{k=0}^K\frac{\varepsilon_k}{\theta_k^2}+\sum_{k=1}^{K+1}\frac{2\sqrt{\varepsilon_{k-1}}}{\theta_{k-1}}\ell_k.
\end{aligned}
\end{eqnarray}
From Lemma \ref{schnidt-lemma} and a similar induction to Theorem \ref{theorem_DNC_sc}, we have
\begin{eqnarray}
\begin{aligned}\notag
&\ell_{K+1}^2\\
\leq&\left(\sum_{k=1}^{K+1}\frac{2\sqrt{\varepsilon_{k-1}}}{\theta_{k-1}}\right)^2+\sum_{k=0}^K\frac{2\varepsilon_k}{\theta_k^2}+2\ell_0^2\\
\overset{a}\leq&\left(\sum_{k=1}^{K+1}\hspace*{-0.09cm}2k\sqrt{\varepsilon_{k-1}}\hspace*{-0.09cm}\right)^2\hspace*{-0.15cm}+\hspace*{-0.09cm}\sum_{k=0}^K2\varepsilon_k(k+1)^2\hspace*{-0.09cm}+\hspace*{-0.09cm}\frac{\|\lambda^*\|_F^2}{\beta_0}\hspace*{-0.09cm}+\hspace*{-0.09cm}L\|\w^0\hspace*{-0.09cm}-\hspace*{-0.09cm}\x^*\|_F^2,
\end{aligned}
\end{eqnarray}
where we use $\frac{1}{k+1}\leq\theta_k\leq\frac{2}{k+2}$ and $\frac{1}{\theta_{-1}^2}=0$ in $\overset{a}\leq$, which can be derived from $\frac{1-\theta_k}{\theta_k^2}=\frac{1}{\theta_{k-1}^2}$ and $\theta_0=1$. Letting $\varepsilon_k\leq\frac{1}{(k+1)^{4+2\tau}}$, then we have $\sum_{k=0}^K2\varepsilon_k(k+1)^2\leq\frac{2}{1+2\tau}$ and $\sum_{k=1}^{K+1}2k\sqrt{\varepsilon_{k-1}}\leq\frac{2}{\tau}$. So
\begin{eqnarray}
\begin{aligned}\notag
\ell_{K+1}^2\leq\frac{4}{\tau^2}+\frac{4}{1+2\tau}+\frac{\|\lambda^*\|_F^2}{\beta_0}+L\|\x^0-\x^*\|_F^2\equiv C_{11},
\end{aligned}
\end{eqnarray}
where we let $\tau=1$ for simplicity and use Lemma \ref{lemma02}. From the definition of $\w^{k+1}=\frac{\x^{k+1}}{\theta_k}-\frac{1-\theta_k}{\theta_k}\x^k$, we have $\|\x^{k+1}-\x^*\|_F=\|\theta_k\w^{k+1}+(1-\theta_k)\x^k-\x^*\|_F\leq \theta_k\|\w^{k+1}-\x^*\|_F+(1-\theta_k)\|\x^k-\x^*\|_F$. By induction, we can prove $\|\x^{K+1}-\x^*\|_F^2\leq \frac{2C_{11}}{L}$ for any $k$. Similar to the proof of Theorem \ref{theorem_DNC_sc} and using Lemma \ref{lemma02}, we have the remaining conclusions.
\hfill$\Box$\end{proof}

\subsubsection{Total Numbers of Communications and Computations}
Based on Theorems \ref{theorem_DNC_sc} and \ref{theorem_DNC_ns}, and the inner loop iteration number given in (\ref{tk}), we can establish the gradient computation and communication complexities for Algorithm \ref{D-NC}. We first consider the strongly convex case and prove Theorem \ref{theorem_D_NC_sc}.

\hspace*{-0.7cm}\begin{proof}
$\|\Pi\z^k\|_F$ appears in (\ref{tk}). We first prove that $\|\Pi\z^k\|_F$ is bounded for any $k$ given $T_k=\frac{1}{-2\log\left(1-\sqrt{1-\sigma_2(\W)}\right)}\log \left(\frac{\beta_0}{2\vartheta_k\varepsilon_k}\left(\frac{1}{L}\|\nabla f(\x^*)\|_F+6\sqrt{C_4}\right)^2\right)$, where $C_4$ is defined in Theorem \ref{theorem_DNC_sc}. We prove $\|\Pi\z^k\|_F\leq \frac{1}{L}\|\nabla f(\x^*)\|_F+6\sqrt{C_4}$ by induction. The case for $k=0$ can be easily verified since $\|\Pi\z^0\|_F=\|\Pi\x^0-\Pi\x^*\|_F\leq\|\x^0-\x^*\|_F$. Assume that the conclusion holds for all $k\leq K$. Then from (\ref{tk}) we know that (\ref{inexact_assumption}) holds for $k\leq K$. From Theorem \ref{theorem_DNC_sc}, we have $\|\x^K-\x^*\|_F\leq \sqrt{C_4}$ and $\|\x^{K+1}-\x^*\|_F\leq \sqrt{C_4}$. Thus,
\begin{eqnarray}
\begin{aligned}\notag
&\|\Pi\z^{K+1}\|_F\\
\overset{a}\leq& \|\Pi\y^{K+1}\|_F+\frac{1}{L}\|\nabla f(\y^{K+1})\|_F\\
\overset{b}\leq& \|\Pi(\y^{K+1}-\x^*)\|_F+\frac{1}{L}\left(\|\nabla f(\x^*)\|_F+L\|\y^{K+1}-\x^*\|_F\right)\\
\leq& \frac{1}{L}\|\nabla f(\x^*)\|_F+2\|\y^{K+1}-\x^*\|_F\\
\overset{c}\leq& \frac{1}{L}\|\nabla f(\x^*)\|_F+4\|\x^{K+1}-\x^*\|_F+2\|\x^K-\x^*\|_F\\
\leq& \frac{1}{L}\|\nabla f(\x^*)\|_F+6\sqrt{C_4},
\end{aligned}
\end{eqnarray}
where we use (\ref{apg2}) in $\overset{a}\leq$, the smoothness of $f(\x)$ and $\Pi\x^*=\0$ in $\overset{b}\leq$, and $\y^k=\x^k+\frac{\sqrt{L}-\sqrt{\mu}}{\sqrt{L}+\sqrt{\mu}}(\x^k-\x^{k-1})$ in $\overset{c}\leq$, which is equivalent to (\ref{apg1}) with the special setting of $\theta_k$. So we get the conclusion.

From Theorem \ref{theorem_DNC_sc}, to find a solution satisfying (\ref{cont14}), we know that the number of gradient computations, i.e., the number of outer iterations, is $O\left(\sqrt{\frac{L}{\mu}}\log\frac{1}{\epsilon}\right)$. From (\ref{tk}), we have
\begin{eqnarray}
\begin{aligned}\notag
T_k=&O\left(\frac{1}{-\log\left(1-\sqrt{1-\sigma_2(\W)}\right)}\log\frac{1}{(1-\theta)^{2(k+1)}}\right)\\
=&O\left(\frac{k\log\left(1-\sqrt{\mu/L}\right)}{\log\left(1-\sqrt{1-\sigma_2(\W)}\right)}\right)\overset{c}=O\left(\frac{k\sqrt{\mu/L}}{\sqrt{1-\sigma_2(\W)}}\right),
\end{aligned}
\end{eqnarray}
where we use $\log\left((1\hspace*{-0.05cm}-\hspace*{-0.05cm}\sqrt{1\hspace*{-0.05cm}-\hspace*{-0.05cm}\sigma_2(\W)})\right)\hspace*{-0.05cm}\approx\hspace*{-0.05cm} -\sqrt{1\hspace*{-0.05cm}-\hspace*{-0.05cm}\sigma_2(\W)}$ and $\log\left(1\hspace*{-0.05cm}-\hspace*{-0.05cm}\sqrt{\mu/L}\right)\hspace*{-0.05cm}\approx\hspace*{-0.05cm} -\sqrt{\mu/L}$ in $\overset{c}=$ from Taylor expansion when $\sqrt{1-\sigma_2(\W)}$ and $\sqrt{\mu/L}$ are small. Thus, the total number of communications, i.e., the total number of inner iterations, is
\begin{equation}
\sum_{k=0}^{\sqrt{L/\mu}\log\frac{1}{\epsilon}}\hspace*{-0.09cm}O\hspace*{-0.09cm}\left(\hspace*{-0.09cm} k\sqrt{\frac{\mu}{L(1\hspace*{-0.09cm}-\hspace*{-0.09cm}\sigma_2(\W))}}\hspace*{-0.09cm}\right)\hspace*{-0.09cm}=\hspace*{-0.09cm}O\hspace*{-0.09cm}\left(\hspace*{-0.09cm}\sqrt{\frac{L}{\mu(1\hspace*{-0.09cm}-\hspace*{-0.09cm}\sigma_2(\W))}}\log^2\frac{1}{\epsilon}\hspace*{-0.09cm}\right)\hspace*{-0.1cm}. \notag
\end{equation}
The proof is complete.
\hfill$\Box$\end{proof}

Similar to the proof of Theorem \ref{theorem_D_NC_sc}, we can also prove Theorem \ref{theorem_D_NC_ns} for the nonstrongly case.
\begin{proof}
Similar to the above proof of Theorem \ref{theorem_D_NC_sc} and given the similar $T_k$ replacing $C_4$ by $C_9$, we know that $\|\Pi\z^k\|_F$ is also bounded for all $k$. Let $\beta_0\geq L+L\|\nabla f(\x^*)\|_F^2$, and assume $L\geq 1$ and $\|\nabla f(\x^*)\|_F\geq 1$ for simplicity. Using the constants in (\ref{constantR}), we know $C_7=O(mLR_1^2)$, $C_8=O(\sqrt{m}R_1)$, $C_9=O(mR_1^2)$, and $C_{10}=O(mLR_1^2)$. Let $\epsilon=\frac{LR_1^2}{(K+2)^2}$. From Theorem \ref{theorem_DNC_ns}, we know that Algorithm \ref{D-NC} needs $O\left(\sqrt{\frac{L}{\epsilon}}\right)$ gradient computations such that $\frac{1}{m}\left(f\left(\alpha(\x^{K+1})\right)-f(\x^*)\right)\leq\epsilon$ and $\frac{1}{m}\|\Pi\x^{K+1}\|_F^2\leq\epsilon^2$, i.e., (\ref{cont14}) holds. From (\ref{tk}), we have
\begin{equation}
T_k\hspace*{-0.05cm}=\hspace*{-0.05cm}O\hspace*{-0.05cm}\left(\hspace*{-0.05cm}\frac{\log(k+1)^8}{-\hspace*{-0.05cm}\log\hspace*{-0.05cm}\left(\hspace*{-0.05cm}1\hspace*{-0.05cm}-\hspace*{-0.05cm}\sqrt{1\hspace*{-0.05cm}-\hspace*{-0.05cm}\sigma_2(\W)}\hspace*{-0.05cm}\right)}\hspace*{-0.05cm}\right)\hspace*{-0.05cm}=\hspace*{-0.05cm}O\hspace*{-0.05cm}\left(\hspace*{-0.05cm}\frac{\log k}{\sqrt{1\hspace*{-0.05cm}-\hspace*{-0.05cm}\sigma_2(\W)}}\hspace*{-0.05cm}\right)\hspace*{-0.05cm},\notag
\end{equation}
Thus, the total number of communications is
\begin{equation}
\sum_{k=0}^KT_k\hspace*{-0.075cm}=\hspace*{-0.075cm}\sum_{k=0}^{\sqrt{L/\epsilon}}\hspace*{-0.075cm}O\hspace*{-0.075cm}\left(\hspace*{-0.075cm}\frac{ \log k}{\sqrt{1\hspace*{-0.075cm}-\hspace*{-0.075cm}\sigma_2(\W)}}\hspace*{-0.075cm}\right)\hspace*{-0.075cm}=\hspace*{-0.075cm}O\hspace*{-0.075cm}\left(\hspace*{-0.075cm}\sqrt{\frac{L}{\epsilon(1-\sigma_2(\W))}}\log\frac{1}{\epsilon}\hspace*{-0.075cm}\right)\hspace*{-0.075cm}. \notag
\end{equation}
The proof is complete.
\hfill$\Box$\end{proof}

\subsection{Complexity Analysis for Algorithm \ref{D-NG}}\label{section:dng}
Now we prove Theorem \ref{the1}. Similar to Section \ref{section:dnc}, we define
\begin{equation}
\rho_{k+1}=F(\x^{k+1})-F(\x^*)+\<\lambda^*,\U\x^{k+1}\>,\notag
\end{equation}
where $(\x^*,\lambda^*)$ is a pair of KKT points of saddle point problem $\min_{\x}\max_{\lambda} F(\x)+\<\lambda,\U\x\>$ satisfying Lemma \ref{lemma02}. Define
\begin{equation}
\w^{k+1}\equiv\frac{\x^{k+1}}{\theta_k}-\frac{1-\theta_k}{\theta_k}\x^k\mbox{ for any }k\geq 0\mbox{ and }\w^0=\x^0.\notag
\end{equation}
From the definitions of $\w^{k+1}$ and $\y^k$ in (\ref{apgns1}), we have the following easy-to-identify identities.
\begin{lemma}\label{lemma06}
For procedure (\ref{apgns1})-(\ref{apgns3}), we have
\begin{eqnarray}
\begin{aligned}\label{iden_dng}
&\theta_k\x^*+(1-\theta_k)\x^k-\y^k=\theta_k\left(\x^*-\w^k\right),\\
&\theta_k\x^*+(1-\theta_k)\x^k-\x ^{k+1}=\theta_k\left(\x^*-\w^{k+1}\right).
\end{aligned}
\end{eqnarray}
\end{lemma}
We use the same notations of $\rho_{k+1}$ and $\w^k$ with Section \ref{section:dnc} for easy analogy. Different from Section \ref{section:dnc}, we define a new variable
\begin{equation}
\v^{k,t}\equiv\frac{\z^{k,t}}{\theta_k}-\frac{1-\theta_k}{\theta_k}\x^k. \label{def_v}
\end{equation}
The proof of Theorem \ref{the1} is based on the following Lyapunov function
\begin{eqnarray}
\begin{aligned}\notag
&\ell_{k+1}=\frac{\rho_{k+1}}{\theta_k}+\frac{1}{2\beta_0}\left\|\frac{\beta_0}{\vartheta_k}\U\x^{k+1}-\lambda^*\right\|_F^2\\
&+\hspace*{-0.08cm}\left(\hspace*{-0.08cm}\frac{L\theta_{k+1}}{2}\hspace*{-0.08cm}+\hspace*{-0.08cm}\frac{\beta_0}{2}\hspace*{-0.08cm}\right)\hspace*{-0.08cm}\|\w^{k+1}\hspace*{-0.08cm}-\hspace*{-0.08cm}\x^*\|_F^2\hspace*{-0.08cm}+\hspace*{-0.08cm}\frac{M}{2\hspace*{-0.08cm}\sqrt{1\hspace*{-0.08cm}-\hspace*{-0.08cm}\sigma_2(\W)}}\|\v^{k+1,0}\hspace*{-0.08cm}-\hspace*{-0.08cm}\x^*\|_F^2.
\end{aligned}
\end{eqnarray}
Analogy to Lemma \ref{lemma0}, we give the following lemma, which describes a progress in one iteration of Algorithm \ref{D-NG}.
\begin{lemma}\label{lemma09}
Assume that Assumptions \ref{assumption_f}, \ref{assumption_h} and \ref{assumption_w} hold with $\mu=0$. Let sequences $\{\theta_k\}$ and $\{\vartheta_k\}$ satisfy $\theta_0=1$, $\frac{1-\theta_k}{\theta_k}=\frac{1}{\theta_{k-1}}$, and $\vartheta_k=\theta_k$. Assume the following equation holds
\begin{equation}\label{Tk_equ}
\frac{\theta_k}{\eta_kT_k}=\frac{M}{\sqrt{1-\sigma_2(\W)}}.
\end{equation}
Then, for Algorithm \ref{D-NG}, we have
\begin{equation}\label{ell_equ}
\ell_{k+1}\leq \ell_k+\frac{mM^2\eta_k}{2\theta_k}.
\end{equation}
\end{lemma}
The proof of Lemma \ref{lemma09} is based on the following lemma.
\begin{lemma}\label{lemma10}
Assume that Assumptions \ref{assumption_f} and \ref{assumption_w} hold. Define $\widetilde\x^{k,*}=(1-\theta_k)\x^k+\theta_k\x^*$. Then, for Algorithm \ref{D-NG}, we have
\begin{eqnarray}
\begin{aligned}\notag
&\rho_{k+1}-(1-\theta_k)\rho_k\\
\leq& \<\nabla f(\y^k),\x^{k+1}-\widetilde\x^{k,*}\>+\frac{L}{2}\|\x^{k+1}-\y^k\|_F^2\\
&+h(\x^{k+1})-h(\widetilde\x^{k,*})+\<\lambda^*,\U\x^{k+1}-\U\widetilde\x^{k,*}\>.
\end{aligned}
\end{eqnarray}
\end{lemma}
\begin{proof}
From (\ref{cont5}) with $\mu=0$, we have
\begin{eqnarray}
\begin{aligned}\notag
&f(\x^{k+1})\leq f(\x)+\<\nabla f(\y^k),\x^{k+1}-\x\>+\frac{L}{2}\|\x^{k+1}-\y^k\|_F^2.\notag
\end{aligned}
\end{eqnarray}
Firstly let $\x=\x^k$ and then $\x=\x^*$, we obtain two inequalities. Multiplying the first inequality by $(1-\theta_k)$, multiplying the second by $\theta_k$, and adding them together, we have
\begin{eqnarray}
\begin{aligned}\notag
&f(\x^{k+1})-(1-\theta_k)f(\x^k)-\theta_kf(\x^*)\\
\leq& \<\nabla f(\y^k),\x^{k+1}-(1-\theta_k)\x^k-\theta_k\x^*\>+\frac{L}{2}\|\x^{k+1}-\y^k\|_F^2.
\end{aligned}
\end{eqnarray}
Adding $h(\x^{k+1})-(1-\theta_k)h(\x^k)-\theta_kh(\x^*)+\<\lambda^*,\U\x^{k+1}-(1-\theta_k)\U\x^k\>$ to both sides, and using the definition of $\rho_k$, we have
\begin{eqnarray}
\begin{aligned}\label{cont7}
&\rho_{k+1}-(1-\theta_k)\rho_k\\
=&F(\x^{k+1})-(1-\theta_k)F(\x^k)-\theta_kF(\x^*)\\
&+\<\lambda^*,\U\x^{k+1}-(1-\theta_k)\U\x^k\>\\
\leq& \<\nabla f(\y^k),\x^{k+1}\hspace*{-0.05cm}-\hspace*{-0.05cm}(1\hspace*{-0.05cm}-\hspace*{-0.05cm}\theta_k)\x^k\hspace*{-0.05cm}-\hspace*{-0.05cm}\theta_k\x^*\>\hspace*{-0.05cm}+\hspace*{-0.05cm}\frac{L}{2}\|\x^{k+1}\hspace*{-0.05cm}-\hspace*{-0.05cm}\y^k\|_F^2\\
&+h(\x^{k+1})-(1-\theta_k)h(\x^k)-\theta_kh(\x^*)\\
&+\<\lambda^*,\U\x^{k+1}-(1-\theta_k)\U\x^k\>.
\end{aligned}
\end{eqnarray}
From the definition of $\widetilde\x^{k,*}$, $\U\x^*=\0$, and the convexity of $h(\x)$, we have
\begin{eqnarray}
\begin{aligned}\notag
&\x^{k+1}-\widetilde\x^{k,*}=\x^{k+1}-(1-\theta_k)\x^k-\theta_k\x^*,\\
&\U\x^{k+1}-\U\widetilde\x^{k,*}=\U\x^{k+1}-(1-\theta_k)\U\x^k,\\
&h(\widetilde\x^{k,*})\leq (1-\theta_k)h(\x^k)+\theta_kh(\x^*).
\end{aligned}
\end{eqnarray}
Plugging them into (\ref{cont7}), we have the conclusion.
\hfill$\Box$\end{proof}

Now, we give the proof of Lemma \ref{lemma09}.
\begin{proof}
From the fact that $h(\x)$ is $(\sqrt{m}M)$-Lipchitz continuous derived by Assumption \ref{assumption_h}, similar to the induction in (\ref{cont5}), we have
\begin{eqnarray}
\begin{aligned}\label{cont8}
&h(\z^{k,t+1})\\
\leq& h(\z^{k,t})\hspace*{-0.08cm}+\hspace*{-0.08cm}\<\hat\nabla h(\z^{k,t}),\z^{k,t+1}\hspace*{-0.08cm}-\hspace*{-0.08cm}\z^{k,t}\>\hspace*{-0.08cm}+\hspace*{-0.08cm}\sqrt{m}M\|\z^{k,t+1}\hspace*{-0.08cm}-\hspace*{-0.08cm}\z^{k,t}\|_F\hspace*{-1cm}\\
=& h(\z^{k,t})\hspace*{-0.08cm}+\hspace*{-0.08cm}\<\hat\nabla h(\z^{k,t}),\widetilde\x^{k,*}\hspace*{-0.08cm}-\hspace*{-0.08cm}\z^{k,t}\>\\
&\hspace*{-0.08cm}+\hspace*{-0.08cm}\<\hat\nabla h(\z^{k,t}),\z^{k,t+1}\hspace*{-0.08cm}-\hspace*{-0.08cm}\widetilde\x^{k,*}\>\hspace*{-0.15cm}+\sqrt{m}M\|\z^{k,t+1}-\z^{k,t}\|_F\\
\leq& h(\widetilde\x^{k,*})\hspace*{-0.08cm}+\hspace*{-0.08cm}\<\hat\nabla h(\z^{k,t}),\z^{k,t+1}\hspace*{-0.08cm}-\hspace*{-0.08cm}\widetilde\x^{k,*}\>\hspace*{-0.08cm}+\hspace*{-0.08cm}\sqrt{m}M\|\z^{k,t+1}\hspace*{-0.08cm}-\hspace*{-0.08cm}\z^{k,t}\|_F,\hspace*{-1cm}
\end{aligned}
\end{eqnarray}
where $\widetilde\x^{k,*}$ is defined in Lemma \ref{lemma10} and $\hat\nabla h(\z^{k,t})\in\partial h(\z^{k,t})$. On the other hand, from the update role of $\z^{k,t+1}$ in Algorithm \ref{D-NG}, we have
\begin{eqnarray}
\begin{aligned}\label{cont9}
0=&\hat\nabla h(\z^{k,t})+\nabla f(\y^k)+\frac{\beta_0}{\vartheta_k}\U^2\y^k\\
&+\left(L+\frac{\beta_0}{\vartheta_k}\right)(\z^{k,t+1}-\y^k)+\frac{1}{\eta_k}(\z^{k,t+1}-\z^{k,t}).
\end{aligned}
\end{eqnarray}
Thus, we have
\begin{eqnarray}
\begin{aligned}\notag
&\<\nabla  f(\y^k),\z^{k,t+1}-\widetilde\x^{k,*}\>+\frac{L}{2}\|\z^{k,t+1}-\y^k\|_F^2\\
&+h(\z^{k,t+1})-h(\widetilde\x^{k,*})+\<\lambda^*,\U\z^{k,t+1}-\U\widetilde\x^{k,*}\>\\
\overset{a}\leq&\<\nabla f(\y^k)\hspace*{-0.08cm}+\hspace*{-0.08cm}\hat\nabla h(\z^{k,t}),\z^{k,t+1}\hspace*{-0.08cm}-\hspace*{-0.08cm}\widetilde\x^{k,*}\>\hspace*{-0.08cm}+\hspace*{-0.08cm}\<\lambda^*,\U\z^{k,t+1}\hspace*{-0.08cm}-\hspace*{-0.08cm}\U\widetilde\x^{k,*}\>\\
&+\sqrt{m}M\|\z^{k,t+1}-\z^{k,t}\|_F+\frac{L}{2}\|\z^{k,t+1}-\y^k\|_F^2\\
\overset{b}=&-\<\frac{\beta_0}{\vartheta_k}\U^2\y^k+\left(L+\frac{\beta_0}{\vartheta_k}\right)(\z^{k,t+1}-\y^k)\right.\\
&\hspace*{1.95cm}\left.+\frac{1}{\eta_k}(\z^{k,t+1}-\z^{k,t}),\z^{k,t+1}-\widetilde\x^{k,*}\>\\
&+\<\lambda^*,\U\z^{k,t+1}-\U\widetilde\x^{k,*}\>+\sqrt{m}M\|\z^{k,t+1}-\z^{k,t}\|_F\\
&+\frac{L}{2}\|\z^{k,t+1}-\y^k\|_F^2\\
=&-\<\frac{\beta_0}{\vartheta_k}\U\y^k-\lambda^*,\U\z^{k,t+1}-\U\widetilde\x^{k,*}\>\\
&-\left(L+\frac{\beta_0}{\vartheta_k}\right)\<\z^{k,t+1}-\y^k,\y^k-\widetilde\x^{k,*}\>\\
&-\frac{1}{\eta_k}\<\z^{k,t+1}-\z^{k,t},\z^{k,t+1}-\widetilde\x^{k,*}\>\\
&+\sqrt{m}M\|\z^{k,t+1}-\z^{k,t}\|_F-\left(\frac{L}{2}+\frac{\beta_0}{\vartheta_k}\right)\|\z^{k,t+1}-\y^k\|_F^2\\
\overset{c}=&-\frac{\vartheta_k}{\beta_0}\<\frac{\beta_0}{\vartheta_k}\U\y^k-\lambda^*,\frac{\beta_0}{\vartheta_k}\U\z^{k,t+1}-\frac{\beta_0}{\vartheta_{k-1}}\U\x^k\>\\
&-\left(L+\frac{\beta_0}{\vartheta_k}\right)\<\z^{k,t+1}-\y^k,\y^k-(1-\theta_k)\x^k-\theta_k\x^*\>\\
&-\frac{1}{\eta_k}\<\z^{k,t+1}-\z^{k,t},\z^{k,t+1}-(1-\theta_k)\x^k-\theta_k\x^*\>\\
&+\sqrt{m}M\|\z^{k,t+1}-\z^{k,t}\|_F-\left(\frac{L}{2}+\frac{\beta_0}{\vartheta_k}\right)\|\z^{k,t+1}-\y^k\|_F^2,
\end{aligned}
\end{eqnarray}
where we use (\ref{cont8}) in $\overset{a}\leq$, (\ref{cont9}) in $\overset{b}=$, $\frac{1}{\vartheta_{k-1}}=\frac{1-\theta_k}{\vartheta_k}$, and the definition of $\widetilde\x^{k,*}$ in $\overset{c}=$. Applying the identities in Lemma \ref{lemma01} to the two inner products, using $\frac{\vartheta_k}{2\beta_0}\left\|\frac{\beta_0}{\vartheta_k}\U\y^k-\frac{\beta_0}{\vartheta_k}\U\z^{k,t+1}\right\|_F^2\leq \frac{\beta_0}{2\vartheta_k}\|\y^k-\z^{k,t+1}\|_F^2$ and dropping the negative term $-\frac{\vartheta_k}{2\beta_0}\left\|\frac{\beta_0}{\vartheta_k}\U\y^k-\frac{\beta_0}{\vartheta_{k-1}}\U\x^k\right\|_F^2$, we have
\begin{eqnarray}
\begin{aligned}\notag
&\<\nabla f(\y^k),\z^{k,t+1}-\widetilde\x^{k,*}\>+\frac{L}{2}\|\z^{k,t+1}-\y^k\|_F^2\\
&+h(\z^{k,t+1})-h(\widetilde\x^{k,*})+\<\lambda^*,\U\z^{k,t+1}-\U\widetilde\x^{k,*}\>\\
\leq&\frac{\vartheta_k}{2\beta_0}\left[ \left\|\frac{\beta_0}{\vartheta_{k-1}}\U\x^k-\lambda^*\right\|_F^2-\left\|\frac{\beta_0}{\vartheta_k}\U\z^{k,t+1}-\lambda^*\right\|_F^2\right]\\
&+\left(\frac{L}{2}+\frac{\beta_0}{2\vartheta_k}\right)\left[\|\y^k-(1-\theta_k)\x^k-\theta_k\x^*\|_F^2\right.\\
&\hspace*{2.5cm}\left.-\|\z ^{k,t+1}-(1-\theta_k)\x^k-\theta_k\x^*\|_F^2\right]\\
&+\frac{1}{2\eta_k}\left[ \|\z^{k,t}-(1-\theta_k)\x^k-\theta_k\x^*\|_F^2\right.\\
&\hspace*{1.2cm}\left. - \|\z^{k,t+1}-(1-\theta_k)\x^k-\theta_k\x^*\|_F^2 \right]\\
&+\sqrt{m}M\|\z^{k,t+1}-\z^{k,t}\|_F-\frac{1}{2\eta_k}\|\z^{k,t+1}-\z^{k,t}\|_F^2\\
\overset{d}\leq&\frac{\vartheta_k}{2\beta_0}\left[ \left\|\frac{\beta_0}{\vartheta_{k-1}}\U\x^k-\lambda^*\right\|_F^2-\left\|\frac{\beta_0}{\vartheta_k}\U\z^{k,t+1}-\lambda^*\right\|_F^2\right]\\
&+\left(\frac{L}{2}+\frac{\beta_0}{2\vartheta_k}\right)\left[\|\y^k-(1-\theta_k)\x^k-\theta_k\x^*\|_F^2\right.\\
&\hspace*{2.5cm}\left.-\|\z ^{k,t+1}-(1-\theta_k)\x^k-\theta_k\x^*\|_F^2\right]\\
&+\frac{1}{2\eta_k}\left[ \|\z^{k,t}-(1-\theta_k)\x^k-\theta_k\x^*\|_F^2\right.\\
&\hspace*{1.2cm}\left. - \|\z^{k,t+1}-(1-\theta_k)\x^k-\theta_k\x^*\|_F^2 \right]+\frac{mM^2\eta_k}{2},
\end{aligned}
\end{eqnarray}
where we use $-\frac{a}{2}t^2+bt\leq \frac{b^2}{2a}$ for any $a>0$ in $\overset{d}\leq$. Summing over $t=0,\cdots,T_k-1$ and dividing both sides by $T_k$, letting $\x^{k+1}=\frac{\sum_{t=0}^{T_k-1}\z^{k,t+1}}{T_k}$, and from the convexity of $h(\x)$ and $\|\cdot\|_F^2$, we have
\begin{eqnarray}
\begin{aligned}\notag
&\<\nabla f(\y^k),\x^{k+1}-\widetilde\x^{k,*}\>+\frac{L}{2}\|\x^{k+1}-\y^k\|_F^2\\
&+h(\x^{k+1})-h(\widetilde\x^{k,*})+\<\lambda^*,\U\x^{k+1}-\U\widetilde\x^{k,*}\>\\
\leq&\frac{\vartheta_k}{2\beta_0}\left[ \left\|\frac{\beta_0}{\vartheta_{k-1}}\U\x^k-\lambda^*\right\|_F^2-\left\|\frac{\beta_0}{\vartheta_k}\U\x^{k+1}-\lambda^*\right\|_F^2\right]\\
&+\left(\frac{L}{2}+\frac{\beta_0}{2\vartheta_k}\right)\left[\|\y^k-(1-\theta_k)\x^k-\theta_k\x^*\|_F^2\right.\\
&\hspace*{2.5cm}\left.-\|\x^{k+1}-(1-\theta_k)\x^k-\theta_k\x^*\|_F^2\right]\\
&+\frac{1}{2\eta_kT_k}\left[ \|\z^{k,0}-(1-\theta_k)\x^k-\theta_k\x^*\|_F^2\right.\\
&\hspace*{1.5cm}\left. - \|\z^{k,T_k}-(1-\theta_k)\x^k-\theta_k\x^*\|_F^2 \right]+\frac{mM^2\eta_k}{2}\\
\overset{e}=&\frac{\vartheta_k}{2\beta_0}\left[ \left\|\frac{\beta_0}{\vartheta_{k-1}}\U\x^k-\lambda^*\right\|_F^2-\left\|\frac{\beta_0}{\vartheta_k}\U\x^{k+1}-\lambda^*\right\|_F^2\right]\\
&+\left(\frac{L}{2}+\frac{\beta_0}{2\vartheta_k}\right)\theta_k^2\left[\|\w^k-\x^*\|_F^2-\|\w^{k+1}-\x^*\|_F^2\right]\\
&+\frac{\theta_k^2}{2\eta_kT_k}\left[ \|\v^{k,0}-\x^*\|_F^2 - \|\v^{k+1,0}-\x^*\|_F^2 \right]+\frac{mM^2\eta_k}{2},
\end{aligned}
\end{eqnarray}
where $\overset{e}=$ follows from the identities in Lemma \ref{iden_dng}, the definition of $\v^{k,t}$ in (\ref{def_v}), and $\z^{k+1,0}=\z^{k,T_k}$. Dividing both sides by $\theta_k$ and letting $\vartheta_k=\theta_k$, from Lemma \ref{lemma10}, $\frac{1}{\theta_{k-1}}=\frac{1-\theta_k}{\theta_k}$, $\frac{\theta_k}{2\eta_kT_k}=\frac{M}{2\sqrt{1-\sigma_2(\W)}}$, $\theta_{k+1}\leq\theta_k$, and the definition of $\ell_k$, we have the conclusion.
\hfill$\Box$\end{proof}

Based on Lemma \ref{lemma09}, we can prove Theorem \ref{the1}.
\begin{proof}
The settings of $T_k=K(1-\sigma_2(\W))$ and $\eta_k=\frac{\theta_k}{KM\sqrt{1-\sigma_2(\W)}}$ satisfy (\ref{Tk_equ}). Plugging them into (\ref{ell_equ}), we have
\begin{equation}
\ell_{k+1}\leq \ell_k+\frac{mM}{2K\sqrt{1-\sigma_2(\W)}}.\notag
\end{equation}
Summing over $k=0,\cdots,K-1$, we have
\begin{eqnarray}
\begin{aligned}\notag
\ell_{K}\leq& \ell_0+\frac{mM}{2\sqrt{1-\sigma_2(\W)}}\\
=&\frac{1}{2\beta_0}\|\lambda^*\|_F^2+\frac{L+\beta_0}{2}\|\x^0-\x^*\|_F^2\\
&+\frac{M}{2\sqrt{1-\sigma_2(\W)}}\|\x^0-\x^*\|_F^2+\frac{mM}{2\sqrt{1-\sigma_2(\W)}}\\
\equiv& C_{12},
\end{aligned}
\end{eqnarray}
where we use $\theta_0=1$, $\frac{1}{\theta_{-1}}=\frac{1-\theta_0}{\theta_0}=0$, $\w^0=\x^0$, and $\v^{0,0}=\x^0$. Similar to the proofs of Theorems \ref{theorem_DNC_sc} and \ref{theorem_DNC_ns}, from the definition of $\ell_k$ and $\theta_{k-1}=\frac{1}{k}$, we have
\begin{eqnarray}
\begin{aligned}\notag
&\|\U\x^K\|_F\leq \frac{1}{\beta_0K}\left(\sqrt{2\beta_0C_{12}}+\|\lambda^*\|_F\right),\\
&F(\x^K)-F(\x^*)\leq\frac{C_{12}}{K}+\|\lambda^*\|_F\|\U\x^K\|_F
\end{aligned}
\end{eqnarray}
and $\|\x^K-\x^*\|_F^2\leq \frac{2C_{12}}{\beta_0}$. Similar to (\ref{cont13}), we also have
\begin{eqnarray}
\begin{aligned}\notag
&F\left(\alpha(\x^K)\right)-F(\x^*)\\
\overset{a}\leq& \left(\|\nabla f(\x^*)\|+L\sqrt{\frac{2C_{12}}{\beta_0}}\right)\|\Pi\x^K\|_F+\frac{L}{2}\|\Pi\x^K\|_F^2\\
&+2\sqrt{m}M\|\Pi\x^K\|_F+F(\x^K)-F(\x^*),
\end{aligned}
\end{eqnarray}
where we use the fact that $h(\x)$ is $(\sqrt{m}M)$-Lipchitz continuous in $\overset{a}\leq$, i.e, $\|\widetilde\nabla h(\x)\|_F\leq\sqrt{m}M,\forall \widetilde\nabla h(\x)\in\partial h(\x)$. From Lemma \ref{lemma04}, we can further bound $\|\Pi\x^K\|_F$ by $\frac{\|\U\x^K\|_F}{\sqrt{1-\sigma_2(\W)}}$. From Lemma \ref{lemma02}, we know $\|\lambda^*\|_F\leq \frac{\sqrt{m}M+\|\nabla f(\x^*)\|_F}{\sqrt{1-\sigma_2(\W)}}\equiv\frac{1}{\chi}$. From the setting of $\beta_0$, we have $\beta_0\geq\frac{L}{\sqrt{1-\sigma_2(\W)}}\geq L$ and $\beta_0\geq\frac{M}{\sqrt{1-\sigma_2(\W)}}$. Combing with (\ref{constantR}) and $R_1\geq 1$, we have $\frac{1}{\chi}\leq \sqrt{m}\beta_0\left(R_1+\frac{R_2}{L}\right)$ and
\begin{eqnarray}
\begin{aligned}\notag
&C_{12}\leq \frac{1}{2\beta_0\chi^2}+\frac{3\beta_0mR_1^2}{2}+\frac{\beta_0 m}{2}\leq 2.5\beta_0m\left(R_1+\frac{R_2}{L}\right)^2,\\
&\|\U\x^K\|_F\hspace*{-0.05cm}\leq\hspace*{-0.05cm} \frac{1}K\hspace*{-0.05cm}\left(\hspace*{-0.05cm}\sqrt{5m}\hspace*{-0.05cm}\left(\hspace*{-0.05cm}R_1\hspace*{-0.05cm}+\hspace*{-0.05cm}\frac{R_2}{L}\hspace*{-0.05cm}\right)\hspace*{-0.05cm}+\hspace*{-0.05cm}\frac{1}{\chi \beta_0}\hspace*{-0.05cm} \right)\hspace*{-0.05cm}\leq\hspace*{-0.05cm} \frac{4\sqrt{m}}{K}\hspace*{-0.05cm}\left(\hspace*{-0.05cm}R_1\hspace*{-0.05cm}+\hspace*{-0.05cm}\frac{R_2}{L}\hspace*{-0.05cm}\right),\\
&\|\Pi\x^K\|_F\hspace*{-0.08cm}\leq\hspace*{-0.08cm} \frac{4\beta_0\sqrt{m}}{KL}\hspace*{-0.1cm}\left(\hspace*{-0.1cm}R_1\hspace*{-0.08cm}+\hspace*{-0.08cm}\frac{R_2}{L}\hspace*{-0.08cm}\right), \|\Pi\x^K\|_F\hspace*{-0.08cm}\leq\hspace*{-0.08cm} \frac{4\beta_0\sqrt{m}}{KM}\hspace*{-0.1cm}\left(\hspace*{-0.1cm}R_1\hspace*{-0.08cm}+\hspace*{-0.08cm}\frac{R_2}{L}\hspace*{-0.08cm}\right),\\
&F(\x^K)-F(\x^*)\leq  \frac{7\beta_0m}{K}\left(R_1+\frac{R_2}{L}\right)^2.
\end{aligned}
\end{eqnarray}
Thus, we further have
\begin{eqnarray}
\begin{aligned}\notag
&F\left(\alpha(\x^K)\right)-F(\x^*)\\
\leq& \left(\sqrt{m}R_2+L\left(R_1+\frac{R_2}{L}\right)\sqrt{5m}\right)\frac{4\beta_0\sqrt{m}}{KL}\left(R_1+\frac{R_2}{L}\right)\\
&+\frac{8\beta_0m}{K^2\hspace*{-0.05cm}\sqrt{1\hspace*{-0.05cm}-\hspace*{-0.05cm}\sigma_2(\W)}}\hspace*{-0.05cm}\left(\hspace*{-0.05cm}R_1\hspace*{-0.05cm}+\hspace*{-0.05cm}\frac{R_2}{L}\hspace*{-0.05cm}\right)^2\hspace*{-0.05cm}+\hspace*{-0.05cm}2\sqrt{m}\frac{4\beta_0\sqrt{m}}{K}\hspace*{-0.05cm}\left(\hspace*{-0.05cm}R_1\hspace*{-0.05cm}+\hspace*{-0.05cm}\frac{R_2}{L}\hspace*{-0.05cm}\right)\\
&+\frac{7\beta_0m}{K}\left(R_1+\frac{R_2}{L}\right)^2\\
\leq&\left(\frac{31\beta_0m}{K}+\frac{8\beta_0m}{K^2\sqrt{1-\sigma_2(\W)}}\right)\left(R_1+\frac{R_2}{L}\right)^2.
\end{aligned}
\end{eqnarray}
The proof is complete.
\hfill$\Box$\end{proof}

Similar to the proof of Theorem \ref{the1}, in the following, we give the proof of Corollary \ref{lemma_cons-noncons}.
\begin{proof}
The settings of $T_k=\frac{1-\sigma_2(\W)}{\theta_k}$ and $\eta_k=\frac{\theta_k^2}{M\sqrt{1-\sigma_2(\W)}}$ satisfy (\ref{Tk_equ}). Plugging them into (\ref{ell_equ}) and using $\theta_k=\frac{1}{k+1}$, we have
\begin{equation}
\ell_{k+1}\leq \ell_k+\frac{mM}{2(k+1)\sqrt{1-\sigma_2(\W)}}.\notag
\end{equation}
Summing over $k=0,\cdots,K-1$, we have
\begin{equation}
\ell_{K}\leq \ell_0+\frac{mM(\log K+1)}{\sqrt{1-\sigma_2(\W)}}.\notag
\end{equation}
Similar to the proof of Theorem \ref{the1}, we have the conclusion.
\hfill$\Box$\end{proof}

\section{Numerical Experiments}\label{sec:exp}
\subsection{Smooth Problem}\label{sec:exp_smooth}
We test the performance of the proposed algorithms on the following least square regression problem
\begin{eqnarray}\label{exp_smooth_problem}
\min_{x\in\R^{n}} \sum_{i=1}^m f_i(x)\quad \mbox{with}\quad f_i(x)\equiv\frac{1}{2}\|\A_i^Tx-\b_i\|^2+\frac{\mu}{2}\|x\|^2.
\end{eqnarray}
We generate $\A_i\in\R^{n\times N/m}$ from the uniform distribution with each entry in $[0,1]$ and normalize each column of $\A_i$ to be 1, where $N$ is the sample size. We set $N=1000$, $n=500$, $m=100$, and $\b_i=\A_i^Tx$ with some $x$ generated from the Gaussian distribution. We consider both the strongly convex objective ($\mu>0$) and nonstrongly convex objective ($\mu=0$).

\begin{figure*}
\centering
\hspace*{-0.31cm}\begin{tabular}{@{\extracolsep{0.001em}}c@{\extracolsep{0.001em}}c@{\extracolsep{0.001em}}c@{\extracolsep{0.001em}}c@{\extracolsep{0.001em}}c@{\extracolsep{0.001em}}c}
\hspace*{-0.31cm}\includegraphics[width=0.232\textwidth,keepaspectratio]{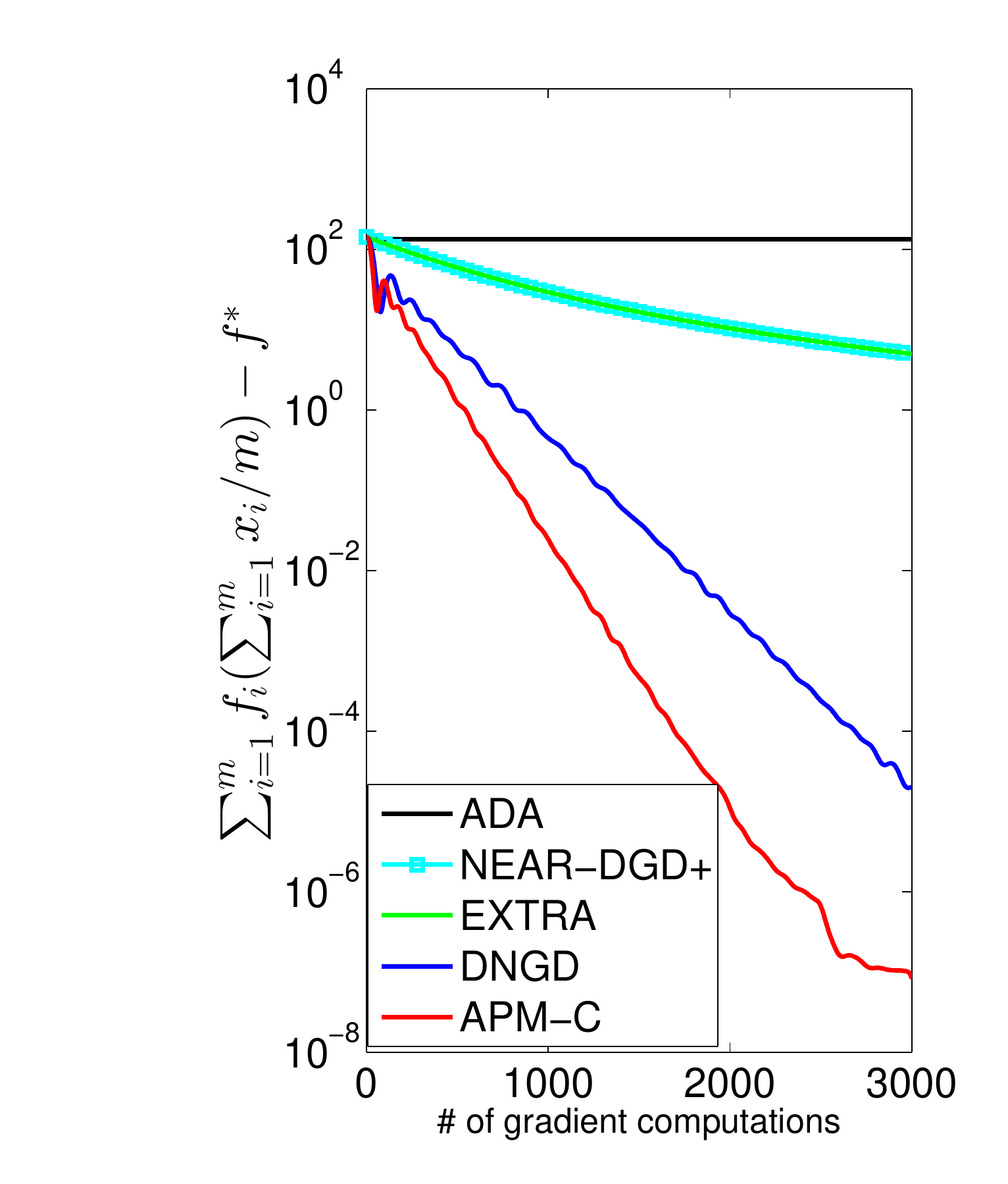}
&\hspace*{-0.05cm}\includegraphics[width=0.161\textwidth,keepaspectratio]{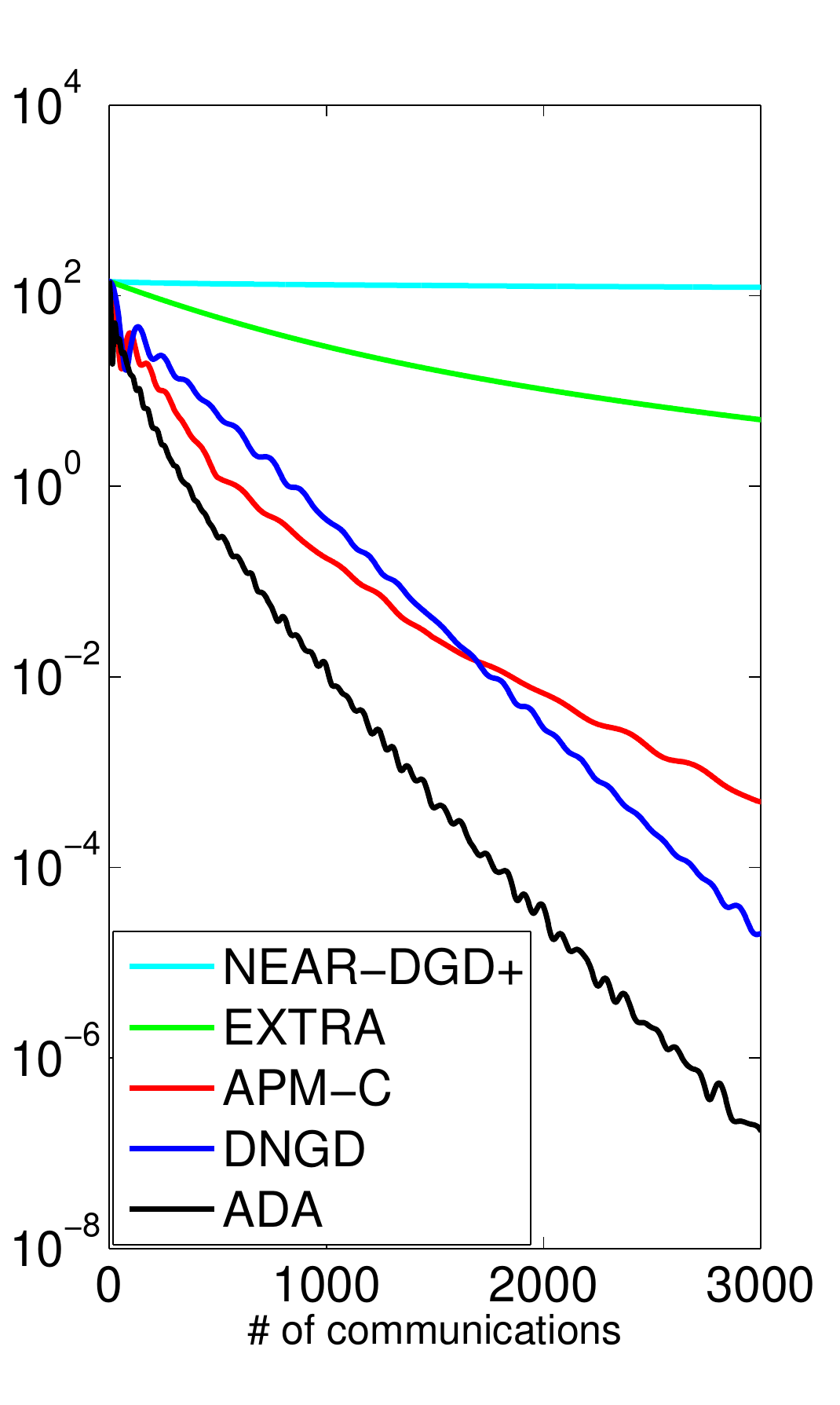}
&\hspace*{-0.05cm}\includegraphics[width=0.161\textwidth,keepaspectratio]{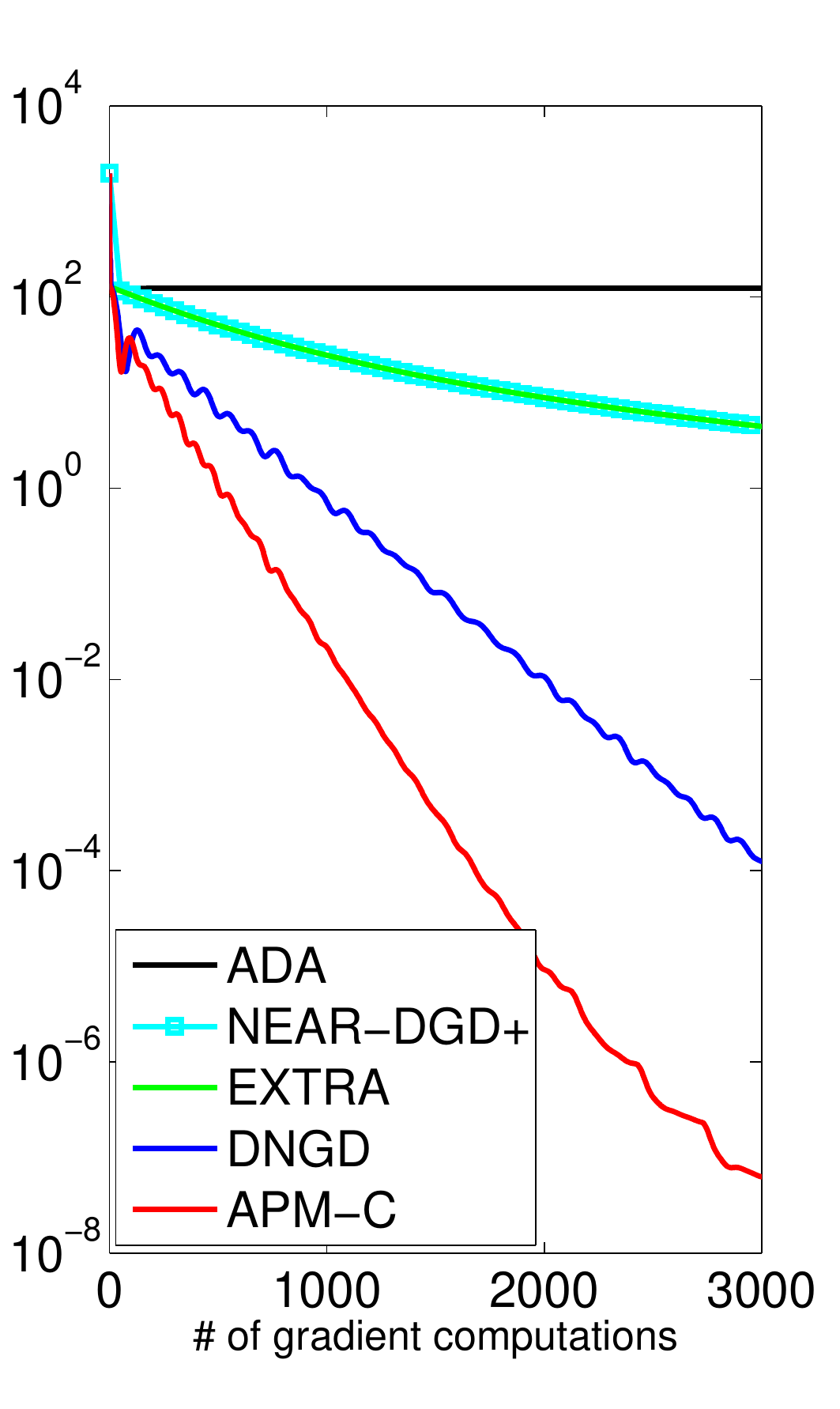}
&\hspace*{-0.05cm}\includegraphics[width=0.161\textwidth,keepaspectratio]{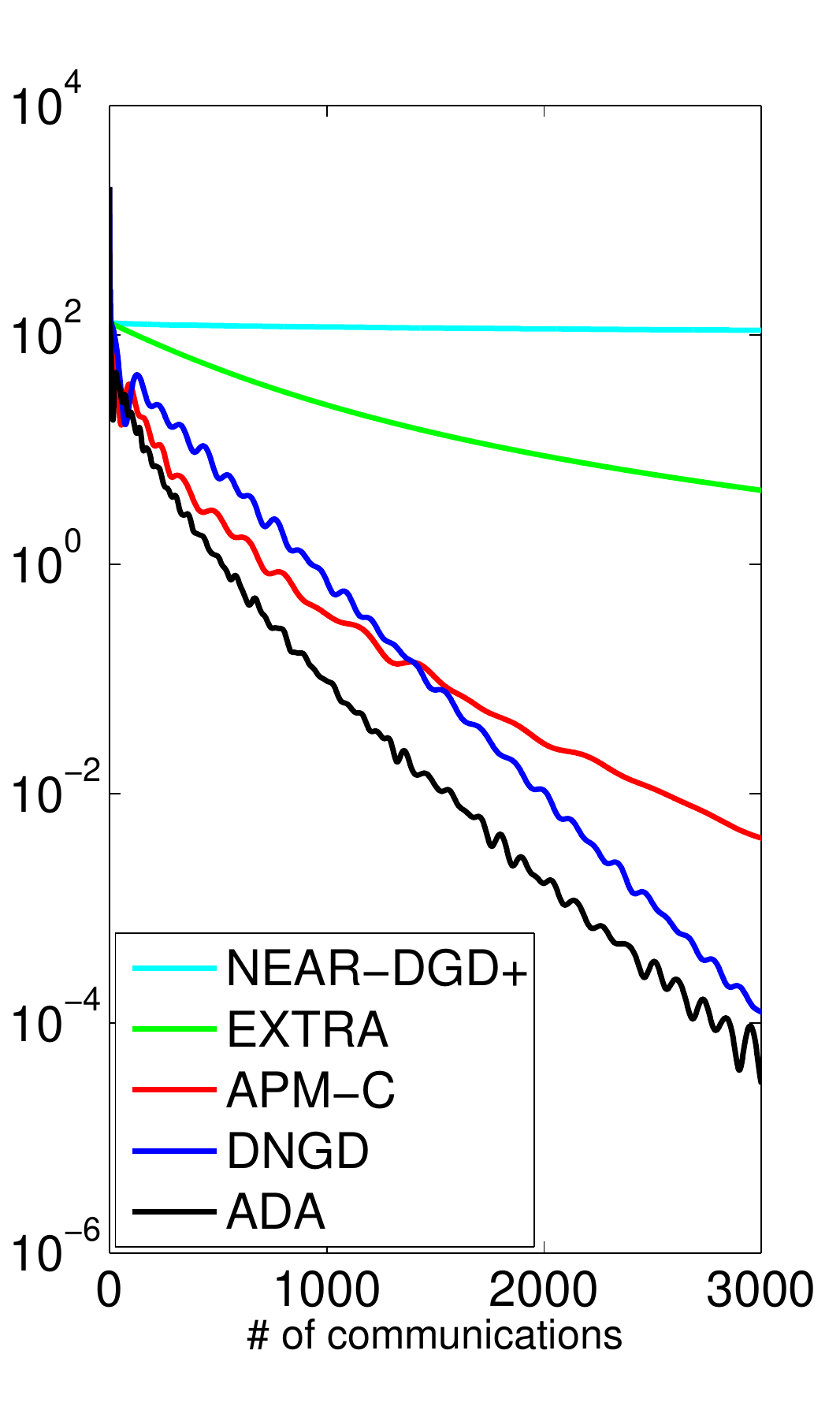}
&\hspace*{-0.05cm}\includegraphics[width=0.161\textwidth,keepaspectratio]{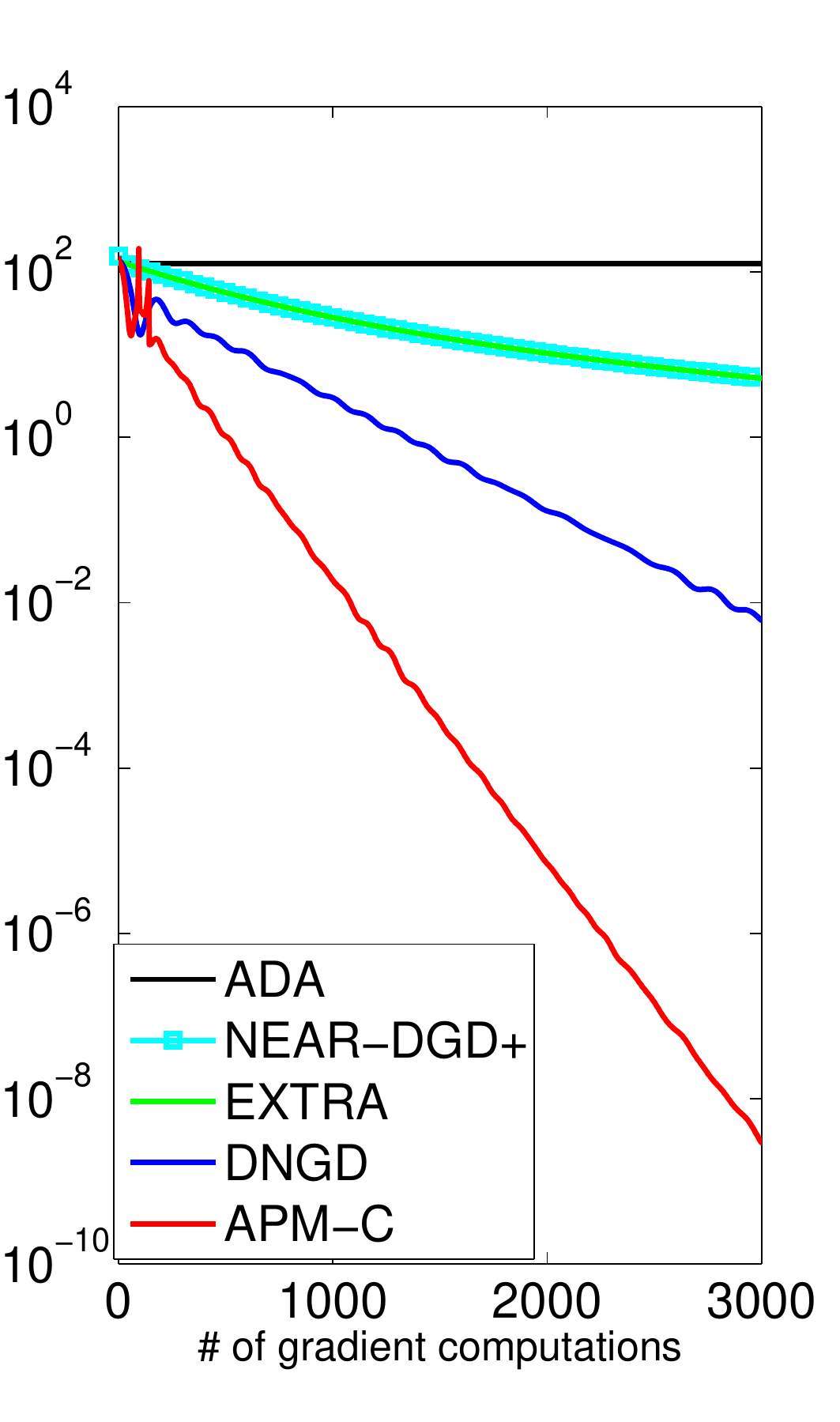}
&\hspace*{-0.05cm}\includegraphics[width=0.161\textwidth,keepaspectratio]{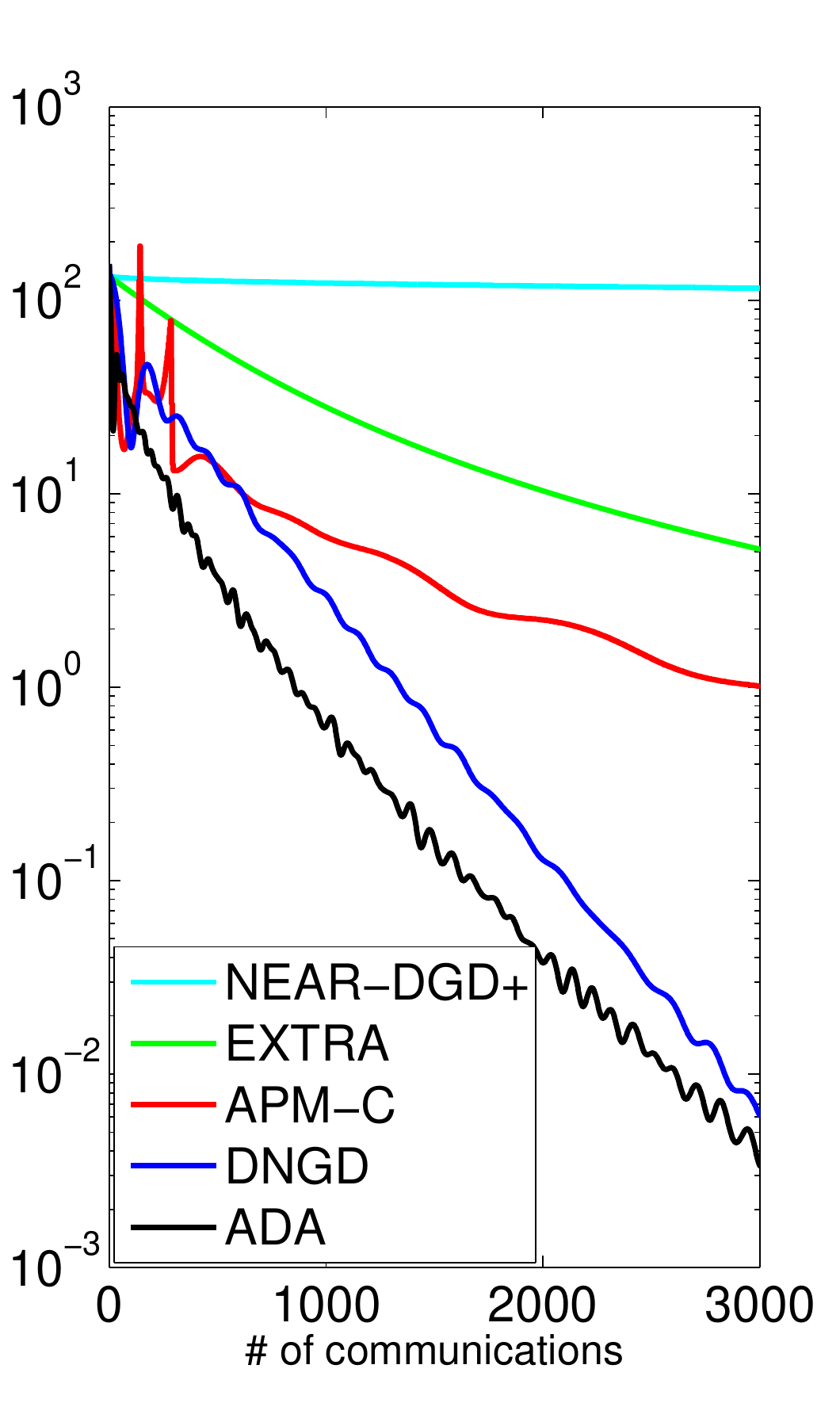}\\
\vspace*{-0.1cm}
\end{tabular}
\caption{Comparisons on the strongly convex problem (\ref{exp_smooth_problem}) and Erd\H{o}s$-$R\'{e}nyi random network with $p=0.5$ (left two), $p=0.1$ (middle two), and $p=0.05$ (right two).}\label{fig2}
\end{figure*}

\begin{figure*}
\centering
\hspace*{-0.35cm}\begin{tabular}{@{\extracolsep{0.001em}}c@{\extracolsep{0.001em}}c@{\extracolsep{0.001em}}c@{\extracolsep{0.001em}}c@{\extracolsep{0.001em}}c@{\extracolsep{0.001em}}c}
\hspace*{-0.35cm}\includegraphics[width=0.24\textwidth,keepaspectratio]{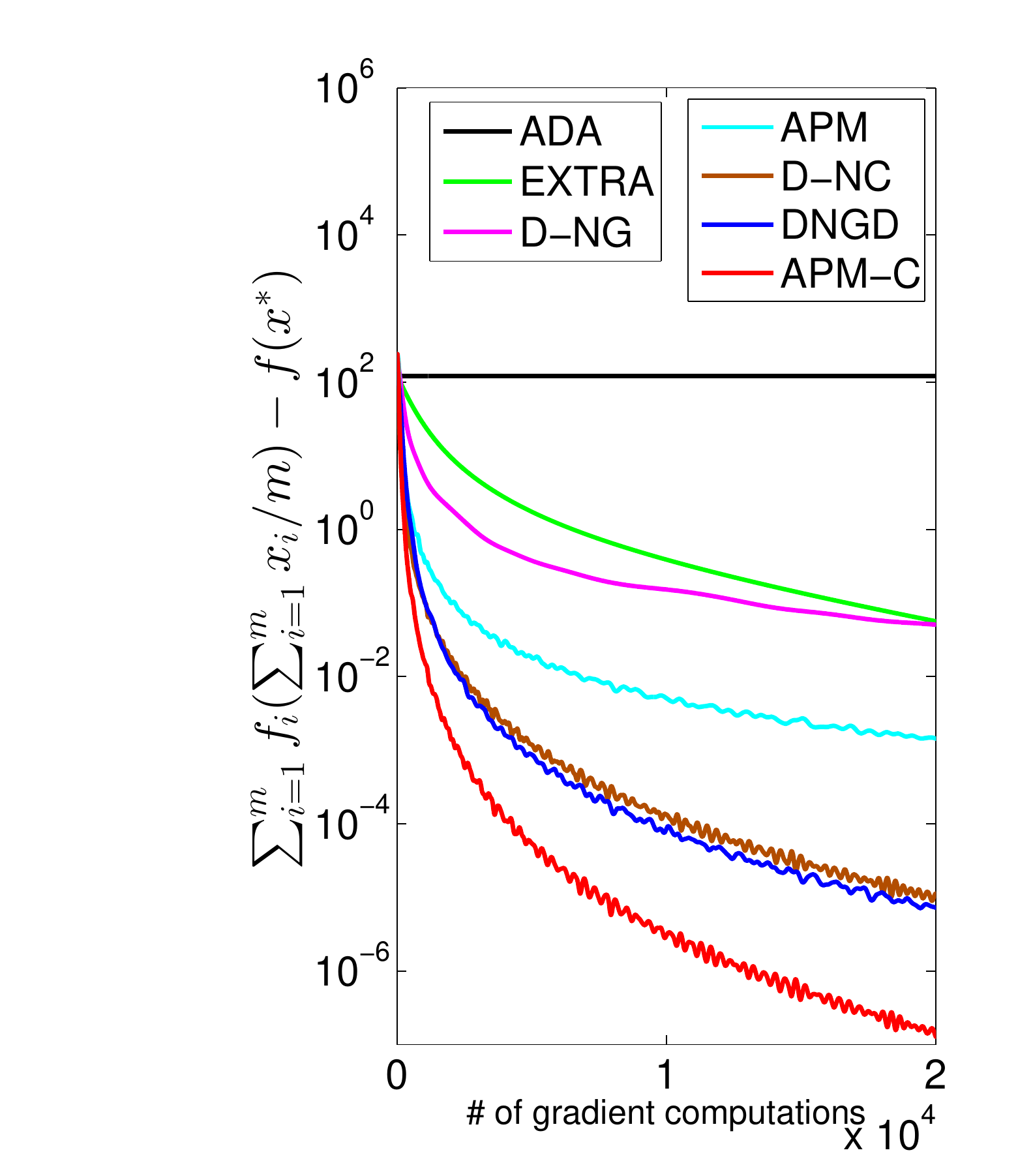}
&\hspace*{-0.05cm}\includegraphics[width=0.161\textwidth,keepaspectratio]{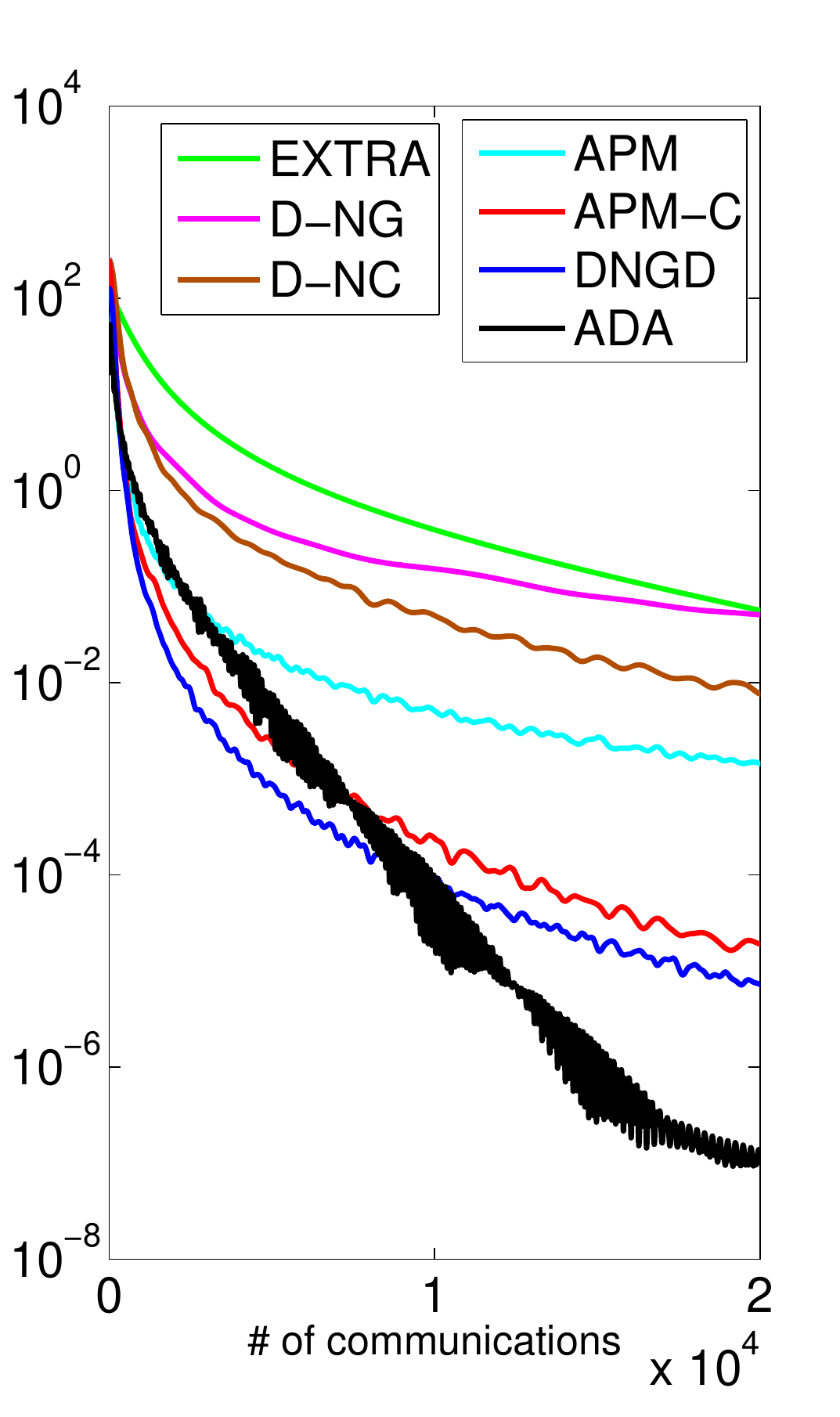}
&\hspace*{-0.05cm}\includegraphics[width=0.161\textwidth,keepaspectratio]{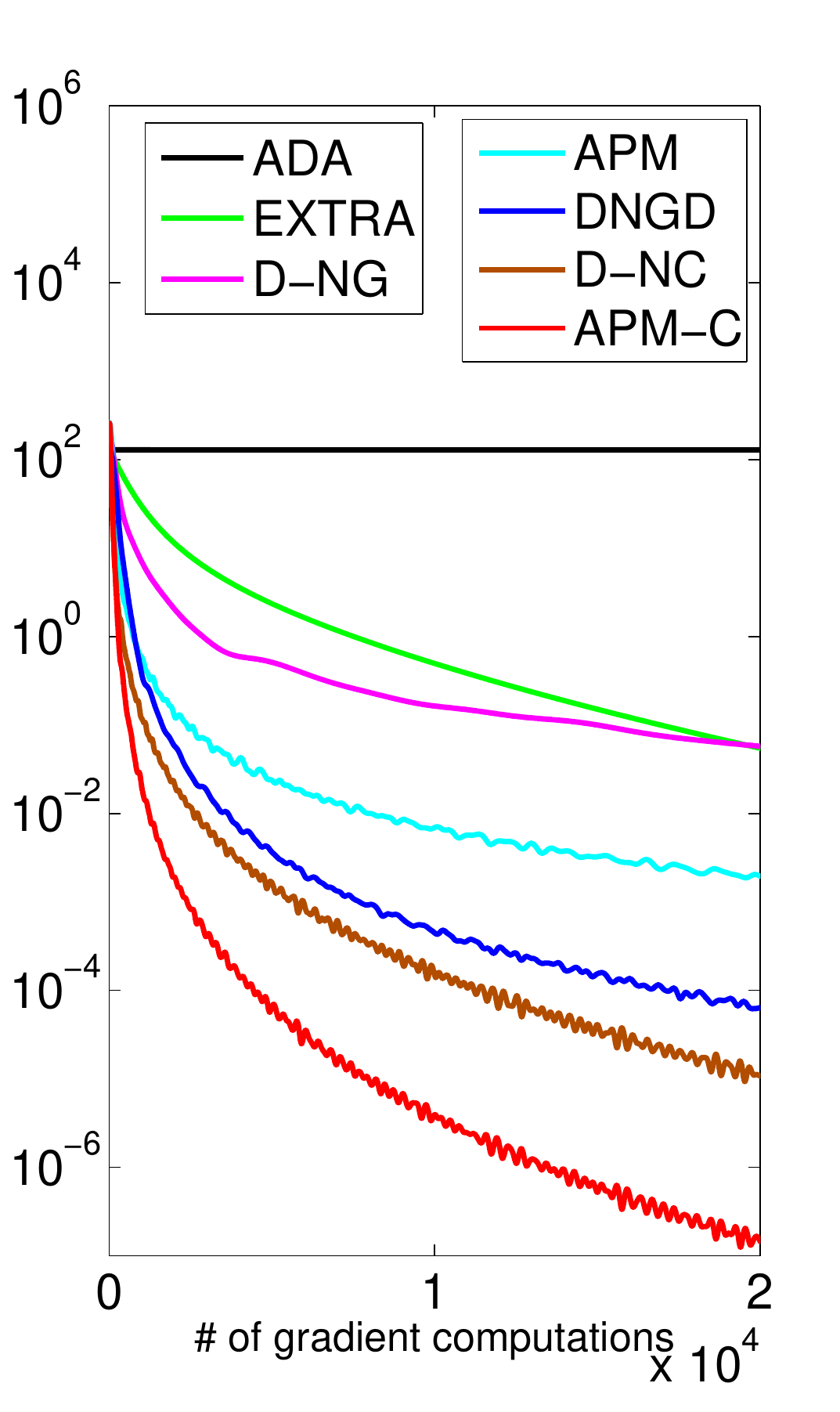}
&\hspace*{-0.05cm}\includegraphics[width=0.161\textwidth,keepaspectratio]{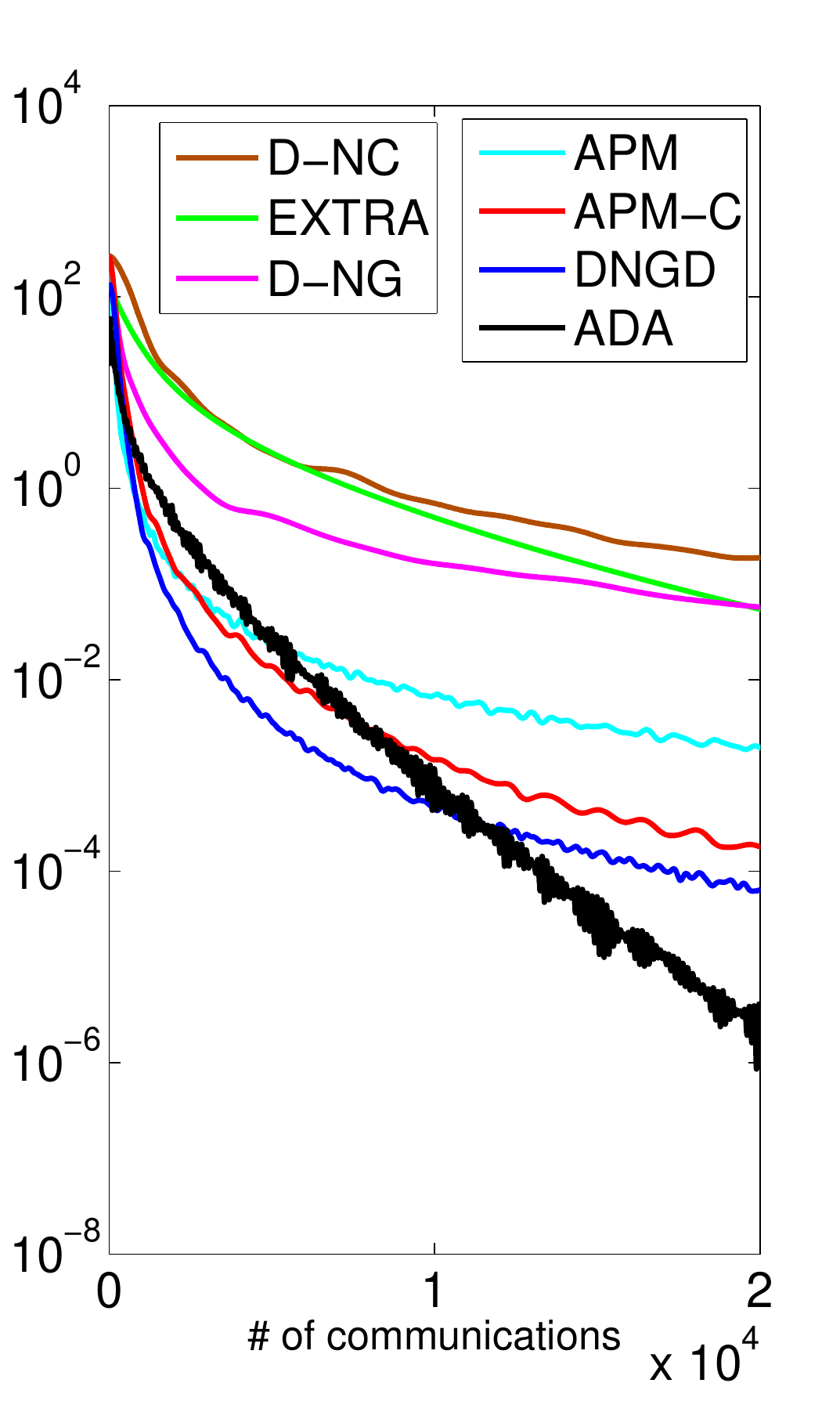}
&\hspace*{-0.05cm}\includegraphics[width=0.161\textwidth,keepaspectratio]{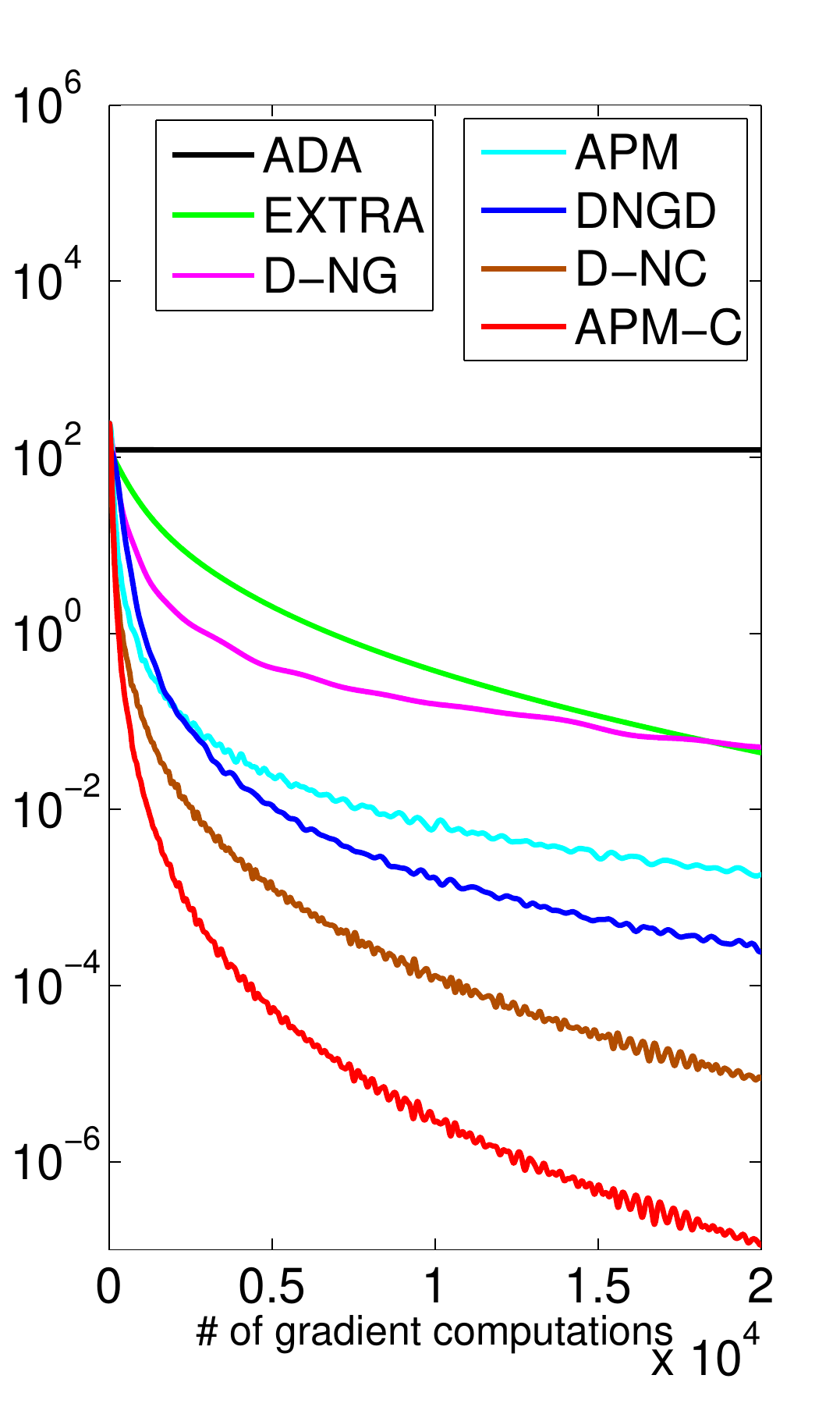}
&\hspace*{-0.05cm}\includegraphics[width=0.161\textwidth,keepaspectratio]{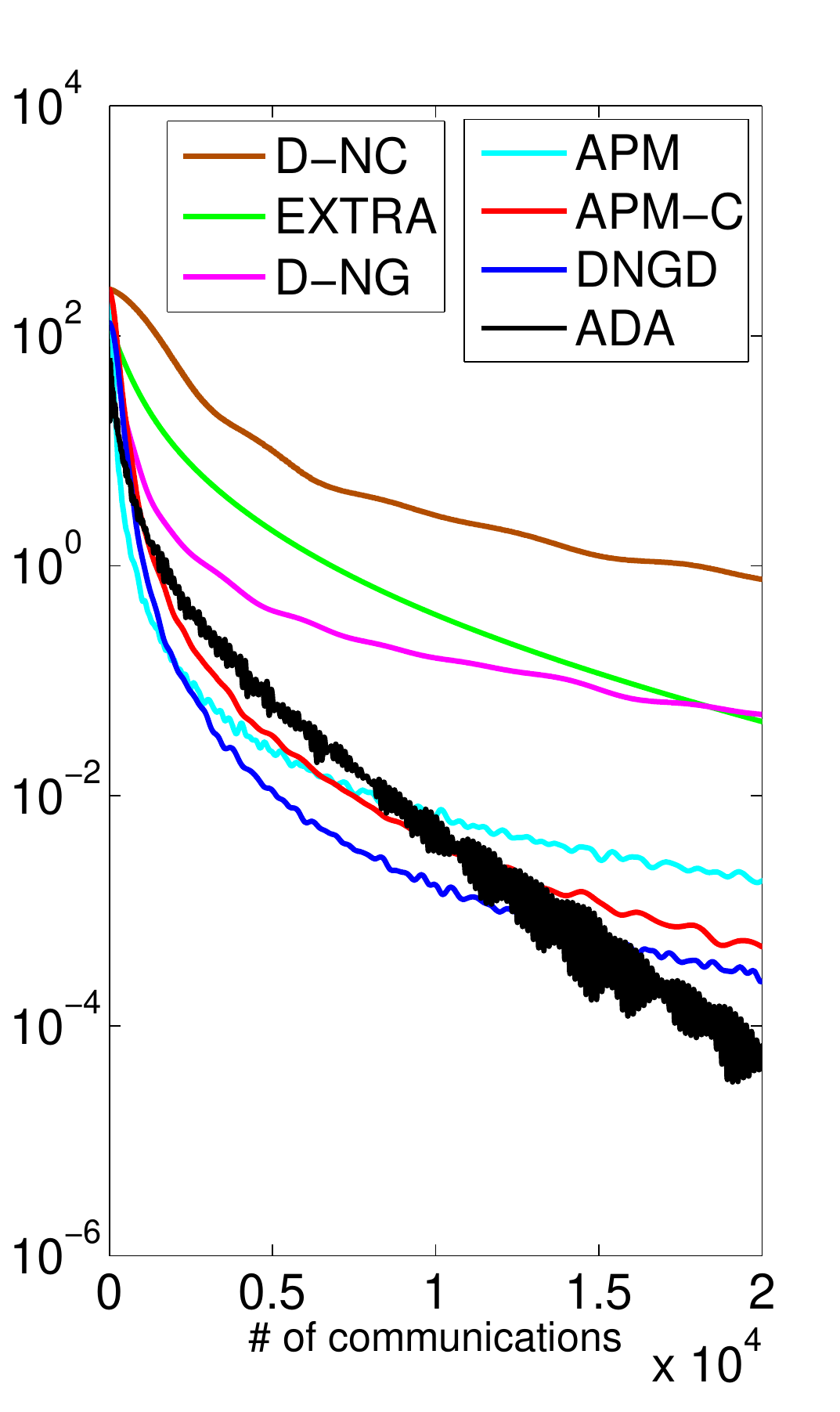}
\end{tabular}
\caption{Comparisons on the nonstrongly convex problem (\ref{exp_smooth_problem}) and Erd\H{o}s$-$R\'{e}nyi random network with $p=0.5$ (left two), $p=0.1$ (middle two), and $p=0.05$ (right two).}\label{fig1}
\end{figure*}

We consider the Erd\H{o}s$-$R\'{e}nyi random graph where each pair of agents has a connection with the probability of $p$. Almost all Erd\H{o}s$-$R\'{e}nyi random graph with $p=\frac{2\log m}{m}$ is connected and $\frac{1}{1-\sigma_2(\W)}=O(1)$ \cite[Proposition 5]{nedic-2018}. We test the performance with $p=0.5$, $p=0.1$, and $p=0.05$, and observe that $1-\sigma_2(\W)=0.33$, $1-\sigma_2(\W)=0.13$, and $1-\sigma_2(\W)=0.04$, respectively. We set $\W=\frac{\I+\M}{2}$, where $\M$ is the Metropolis weight matrix \cite{boyd2004}.

For the strongly convex objective, we compare APM-C with the accelerated dual ascent (ADA) \cite{dasent}, distributed Nesterov's gradient descent (DNGD) \cite{qu2017-2}, EXTRA \cite{Shi-2015}, and NEAR-DGD+ \cite{Wei-2018}. NEAR-DGD+ can be seen as a counterpart of APM-C without Nesterov's acceleration scheme and accelerated average consensus. We set $\mu=0.0001$ and leave the test on different condition numbers in our supplementary material. We set the inner iteration number $T_k$ as $\lceil\frac{k\sqrt{\mu/L}}{3\sqrt{1-\sigma_2(\W)}}\rceil$, $\beta_0=100$ and the stepsize as $\frac{1}{L}$ for APM-C, where $\lceil\cdot\rceil$ means the top integral function. For ADA, we follow the theory in \cite{Uribe-2017} to set the inner iteration number as $\lceil\sqrt{\frac{L}{\mu}}\log\frac{L}{\mu}\rceil$ (we leave the test on the impact of smaller inner iteration numbers in our supplementary material) and the stepsize as $\mu$. We tune the best stepsize as $\frac{1}{L}$ and $\frac{0.5}{L}$ for EXTRA and DNGD, respectively. We follow \cite{Wei-2018} to set $T_k=k$ for NEAR-DGD+. We initialize $\x^0$ at $\0$ for all the compared methods.

Figure \ref{fig2} plots the comparisons. We can see that APM-C has the lowest computation cost and ADA has the lowest communication cost, which match the theory. Thus, APM-C is more suited to the environment where computation is the bottleneck of the overall performance. Due to the large $T_k$ for ADA, it only performs several outer iterations after $3000$ gradient computations and thus has almost no decreasing in the first, third and fifth plots of Figure \ref{fig2}. APM-C has a higher communication cost than DNGD but a lower computation cost for $p=0.1$ and $p=0.5$. APM-C performs better than NEAR-DGD+ and it verifies that Nesterov's acceleration scheme is critical to improve the performance. From Figure \ref{fig2}, we observe that APM-C is more suited to the network with small $\frac{1}{\sqrt{1-\sigma_2(\W)}}$, otherwise, the communication costs will be high, e.g., see the right two plots in Figure \ref{fig2}. In fact, when $\frac{1}{\sqrt{1-\sigma_2(\W)}}$ is small, $\frac{\sqrt{\mu/L}}{\sqrt{1-\sigma_2(\W)}}$ will also be small, e.g., $0.01$ in our experiment with $p=0.1$. Thus the required $T_k$ is small, e.g., $T_{3000}=11$ in our experiment. As a comparison, NEAR-DGD+ suggests $T_k=k$ and thus it increases quickly, which leads to almost no decreasing in the second, fourth and sixth plots of Figure \ref{fig2}. In practice, we can use the expander graph \cite{chow-2016} which satisfies $\frac{1}{1-\sigma_2(\W)}=O(1)$ \cite{nedic-2018}. The Erd\H{o}s$-$R\'{e}nyi random graph is a special case of the expander graph and can be easily implemented.

For the nonstrongly convex objective, we test the performance of APM, APM-C, D-NG \cite{Jakovetic-2014}, D-NC \cite{Jakovetic-2014}, DNGD \cite{qu2017-2}, EXTRA \cite{Shi-2015} and ADA \cite{Uribe-2017}. We set $T_k$ as $\lceil\frac{\log(k+1)}{5\sqrt{1-\sigma_2(\W)}}\rceil$ and $\lceil\frac{\log(k+1)}{-5\log\sigma_2(\W)}\rceil$ for APM-C and D-NC, respectively. We set the stepsize as $\frac{1}{L}$ for the two algorithms and $\beta_0=100$ for APM-C. We set $\frac{\beta_0}{\vartheta_k}=\frac{k+1}{c}$ with $c=50$ for APM and tune the best $c=1$ for D-NG. Larger $c$ makes D-NG diverge. We tune the best stepsize as $\frac{1}{L}$ for EXTRA, $\frac{0.05}{L}$ for DNGD with $p=0.05$, $\frac{0.1}{L}$ for DNGD with $p=0.1$, and $\frac{0.2}{L}$ for DNGD with $p=0.5$, respectively. For ADA, we follow \cite{Uribe-2017} to add a small regularizer of $\frac{\epsilon}{2}\|\x\|^2$ to each $f_i(\x)$ and solve a regularized problem with $\epsilon=10^{-7}$. We set the inner iteration number as $T_k=\lceil\sqrt{\frac{L}{\epsilon}}\log\frac{L}{\epsilon}\rceil$.

\begin{figure*}
\centering
\hspace*{-0.3cm}\begin{tabular}{@{\extracolsep{0.001em}}c@{\extracolsep{0.001em}}c@{\extracolsep{0.001em}}c@{\extracolsep{0.001em}}c@{\extracolsep{0.001em}}c@{\extracolsep{0.001em}}c}
\hspace*{-0.3cm}\includegraphics[width=0.235\textwidth,keepaspectratio]{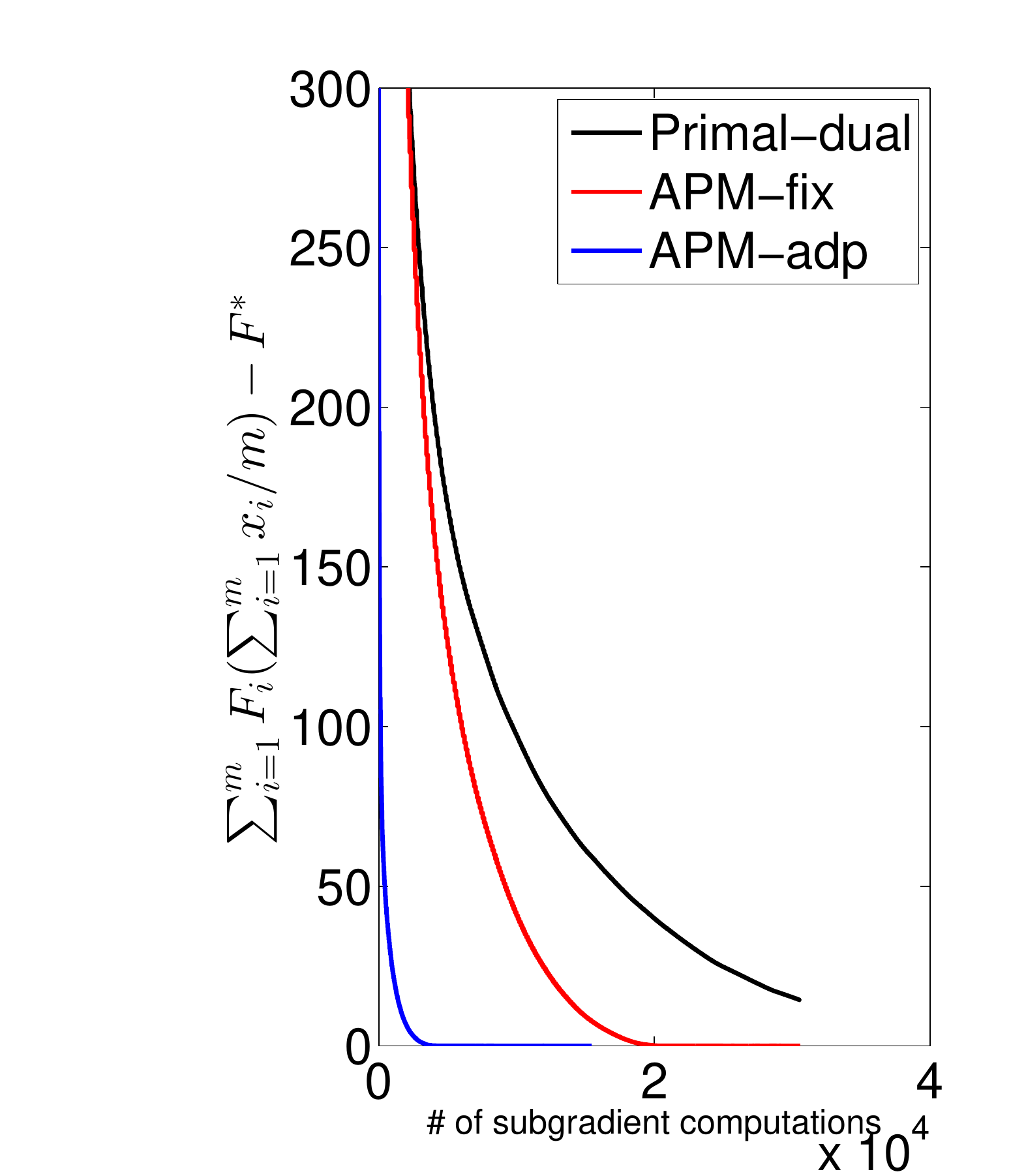}
&\hspace*{-0.05cm}\includegraphics[width=0.16\textwidth,keepaspectratio]{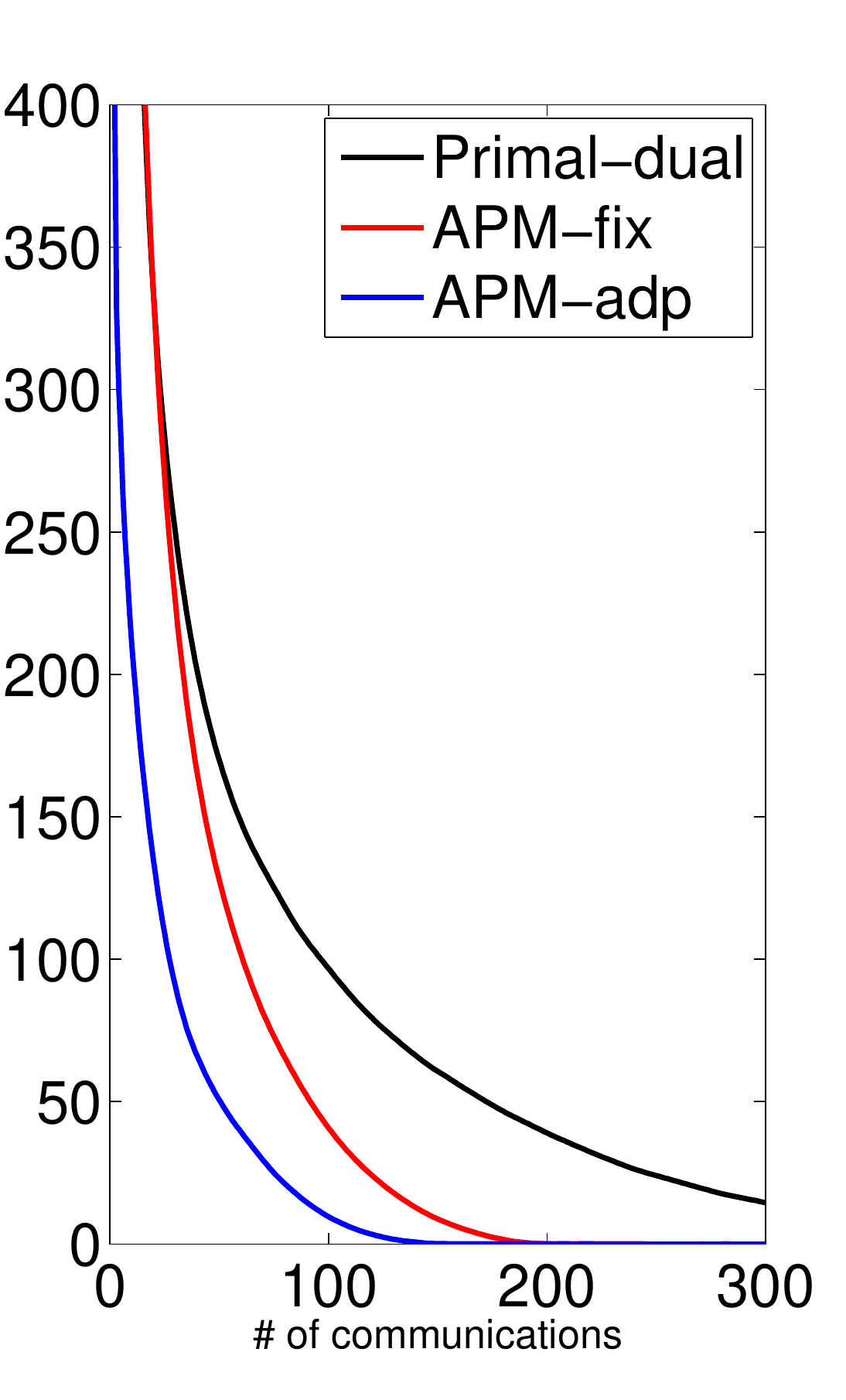}
&\hspace*{-0.05cm}\includegraphics[width=0.16\textwidth,keepaspectratio]{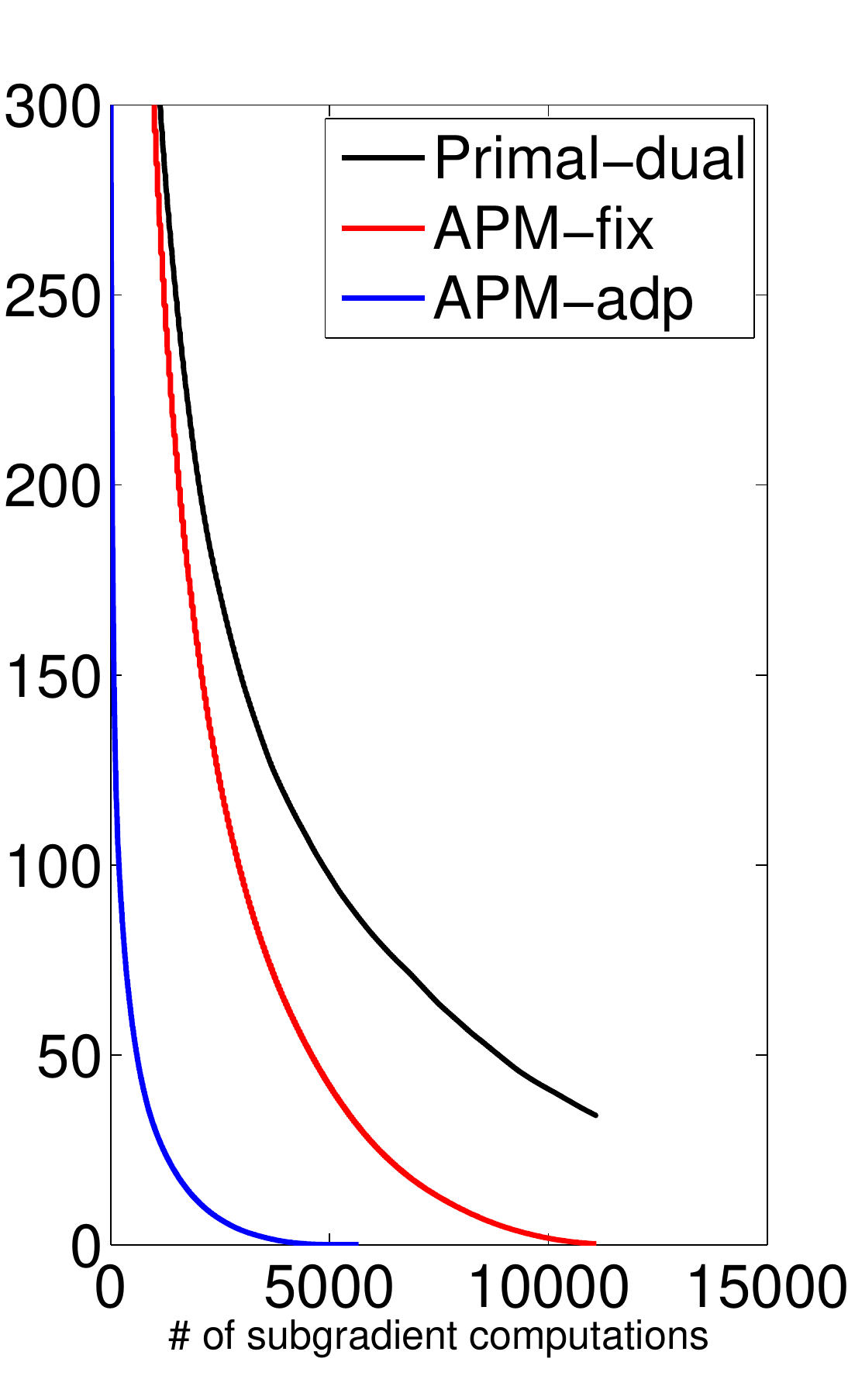}
&\hspace*{-0.01cm}\includegraphics[width=0.16\textwidth,keepaspectratio]{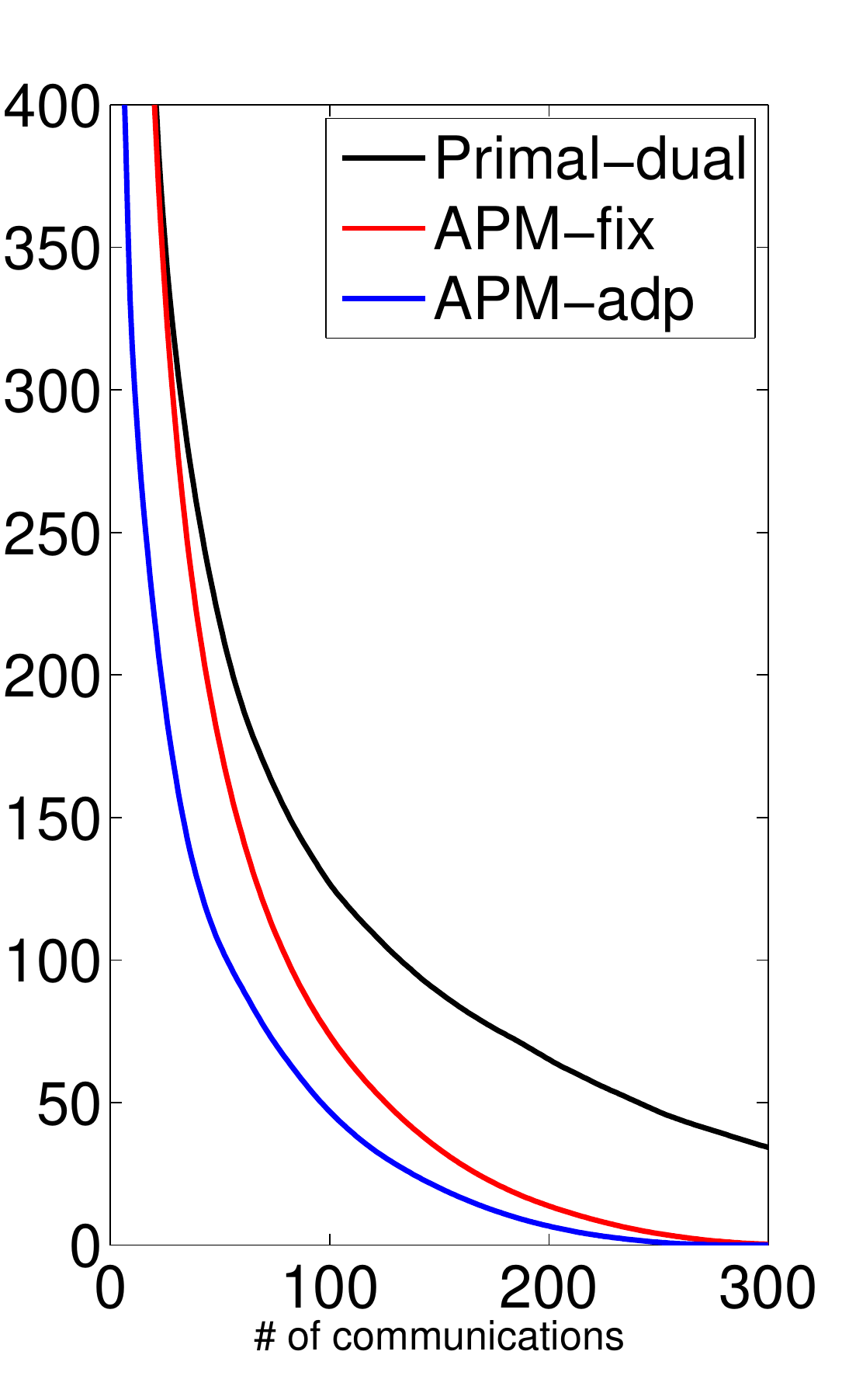}
&\hspace*{-0.05cm}\includegraphics[width=0.16\textwidth,keepaspectratio]{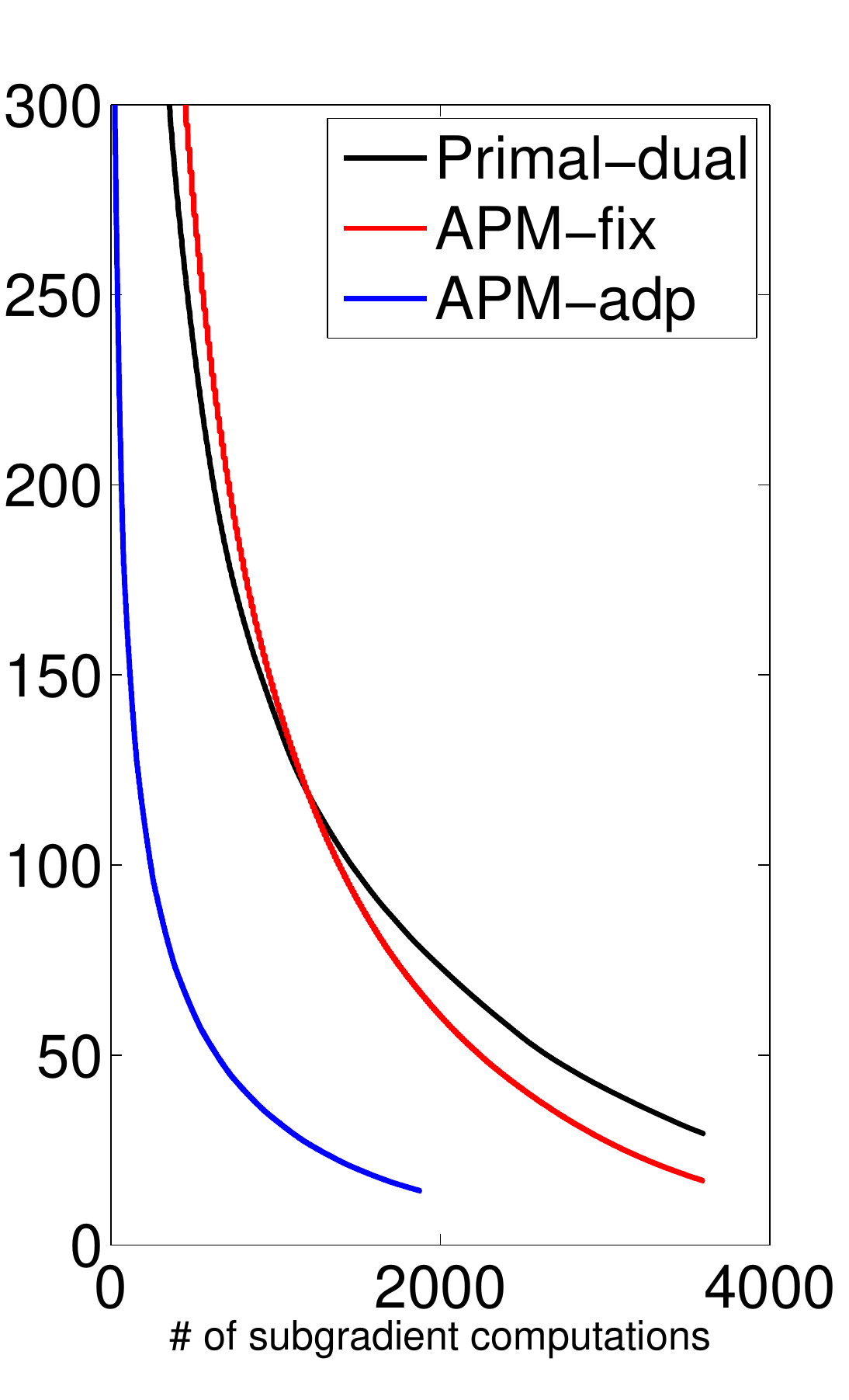}
&\hspace*{-0.05cm}\includegraphics[width=0.16\textwidth,keepaspectratio]{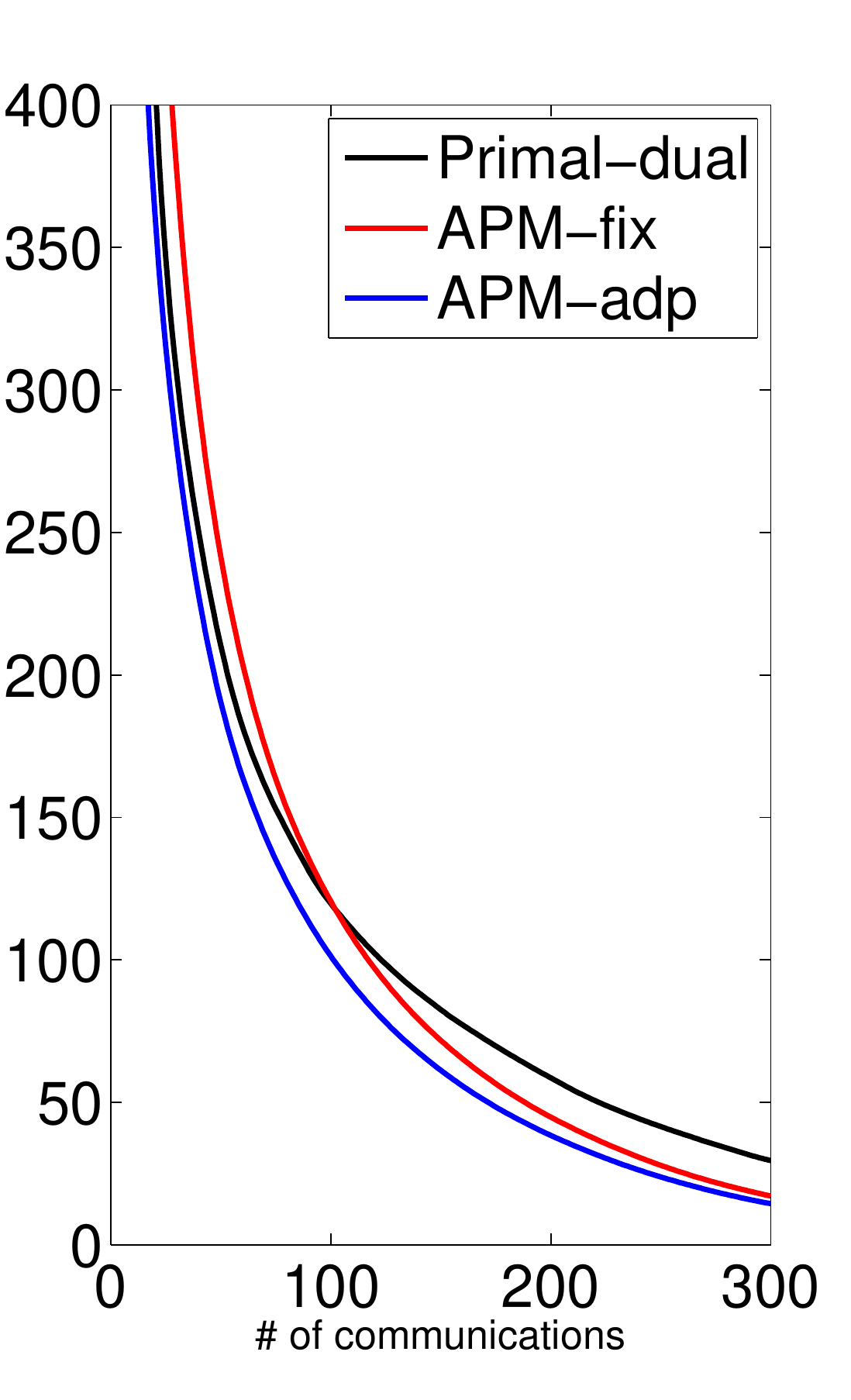}
\end{tabular}
\caption{Comparisons on the nonsmooth problem (\ref{exp_nonsmooth_problem}) and Erd\H{o}s$-$R\'{e}nyi random network with $p=0.5$ (left two), $p=0.1$ (middle two), and $p=0.05$ (right two).}\label{fig0}
\end{figure*}

From figure \ref{fig1}, we can see that APM-C also has the lowest computation cost. APM performs better than D-NG because APM allows to use a larger stepsize in practice, which can reduce the negative impacts from the diminishing stepsize. APM is more suited to the environment where high precision is not required, otherwise, the diminishing stepsize makes the algorithm slow. ADA has the lowest communication cost. However, ADA needs to predefine $\epsilon$ to set the algorithm parameter and thus it only achieves an approximate optimal solution in the precision of $\epsilon$ due to the weakness of the regularization trick. From Figure \ref{fig1}, we can see that the value of $\frac{1}{\sqrt{1-\sigma_2(\W)}}$ has less impact on the performance of APM-C than that in the strongly convex setting.

\subsection{Non-smooth Problem}
In this section, we follow \cite{Lam-2017} to test Algorithm \ref{D-NG} on the following decentralized linear Support Vector Machine (SVM) model
\begin{eqnarray}\label{exp_nonsmooth_problem}
\min_{x\in\R^{n}} \sum_{i=1}^m f_i(x)\quad \mbox{with}\quad f_i(x)\equiv \max\{0,1-\b_i\A_i^Tx\}.
\end{eqnarray}
The problem setting is similar to Section \ref{sec:exp_smooth} and the only difference is that we set $\b_i=\mbox{Sign}(\A_i^Tx)$ for some $\x$ generated from the Gaussian distribution. We also consider the Erd\H{o}s$-$R\'{e}nyi random graph with $p=0.05$, $p=0.1$, and $p=0.5$, respectively. We compare APM with the primal-dual method \cite{scaman-2019}. We test two different parameter settings for APM. For the first one, we follow Corollary \ref{lemma_cons-noncons} to set $\beta_0=\frac{0.01}{\sqrt{1-\sigma_2(\W)}}$, $T_k=\lceil k(1-\sigma_2(\W))\rceil$, and $\eta_k=\frac{5000}{k^2\sqrt{1-\sigma_2(\W)}}$, and name it APM with adaptive parameters (APM-adp). For the second one, we follow Theorem \ref{the1} to set $\beta_0=\frac{0.01}{\sqrt{1-\sigma_2(\W)}}$, $T_k=\lceil K(1-\sigma_2(\W))\rceil$, and $\eta_k=\frac{5000}{kK\sqrt{1-\sigma_2(\W)}}$ with $K=300$ and name it APM with fix parameters (APM-fix). For the primal-dual method, we set the number of inner iterations as $\lceil K(1-\sigma_2(\W))\rceil$ and tune the best parameters of $\sigma=1$ and $\eta=0.5$ in \cite[Alg 3]{scaman-2019}. Figure \ref{fig0} plots the result. We can see that APM performs better than the primal-dual method, and APM-adp needs less communications and subgradient computations than APM-adp.

\section{Conclusion}\label{sec:conclusion}

In this paper, we study the distributed accelerated gradient methods from the perspective of the accelerated penalty method with increasing penalty parameters. Two algorithms are proposed. The first algorithm achieves the optimal gradient computation complexities and near optimal communication complexities for both strongly convex and nonstrongly convex smooth distributed optimization. Our communication complexities are only worse by a factor of $\log\frac{1}{\epsilon}$ than the lower bounds. Our second algorithm obtains both the optimal subgradient computation and communication complexities for nonsmooth distributed optimization. Our APM-C is not suited to the network with large $\frac{1}{\sqrt{1-\sigma_2(\W)}}$ for strongly convex problems, in which case the communication cost is high.

\bibliographystyle{ieeetr}
\bibliography{APM}

\ifCLASSOPTIONcaptionsoff
  \newpage
\fi

\end{document}